\documentclass[11pt]{article}
\usepackage[utf8]{inputenc}
\usepackage{amsmath,amssymb,amscd,amsthm,amsfonts,amsxtra}
\usepackage{appendix}
\usepackage{tipa} 
\usepackage{float}
\usepackage{enumerate}
\usepackage{verbatim}
\usepackage[mathscr]{euscript}
\usepackage{geometry, verbatim}
\usepackage{tikz-cd}
\usepackage{bm}
\usepackage{hyperref}
\usepackage[shortlabels]{enumitem}
\usepackage{mathtools}
\usepackage[all]{xy}
\usepackage{graphicx}
\usepackage{mathrsfs}
\usepackage{soul}
\usepackage[colorinlistoftodos, textsize=tiny]{todonotes}
\usepackage{morefloats}
\usepackage{pdfpages}
\usepackage{csquotes}
\usepackage{colonequals}

\newtheorem{innercustomthm}{Theorem}
\newenvironment{customthm}[1]
  {\renewcommand\theinnercustomthm{#1}\innercustomthm}
  {\endinnercustomthm}
  
  \newtheorem{innercustomcor}{Corollary}
\newenvironment{customcor}[1]
  {\renewcommand\theinnercustomcor{#1}\innercustomcor}
  {\endinnercustomcor}

\RequirePackage{color}
\definecolor{myred}{rgb}{0.75,0,0}
\definecolor{mygreen}{rgb}{0,0.5,0}
\definecolor{myblue}{rgb}{0,0,0.65}

\usetikzlibrary{matrix,arrows,decorations.pathmorphing}

\newcommand{\R}{\mathbb{R}}
\newcommand{\C}{\mathbb{C}}

\newcommand{\Q}{\mathbb{Q}}
\newcommand{\A}{{\mathbb{A}}}
\newcommand{\Z}{\mathbb{Z}}
\newcommand{\F}{\mathbb{F}}
\renewcommand{\P}{\mathbb{P}}

\DeclareMathOperator{\Gal}{Gal}
\DeclareMathOperator{\Spec}{Spec}

\DeclareMathOperator{\Tr}{Tr}

\DeclareMathOperator{\rank}{rank}
\DeclareMathOperator{\lcm}{lcm}
\DeclareMathOperator{\norm}{N}

\theoremstyle{plain}
\newtheorem{theorem}{Theorem}[section]
\newtheorem{proposition}[theorem]{Proposition}
\newtheorem{lemma}[theorem]{Lemma}
\newtheorem{corollary}[theorem]{Corollary}
\theoremstyle{definition}
\newtheorem{definition}[theorem]{Definition}
\newtheorem{remark}[theorem]{Remark}
\newtheorem{example}[theorem]{Example}

\theoremstyle{remark}

\numberwithin{equation}{section}

\DeclareFontFamily{U}{wncy}{}
    \DeclareFontShape{U}{wncy}{m}{n}{<->wncyr10}{}
    \DeclareSymbolFont{mcy}{U}{wncy}{m}{n}
    \DeclareMathSymbol{\Sh}{\mathord}{mcy}{"58}
    \DeclareMathSymbol{\Den}{\mathord}{mcy}{"44}
    \DeclareMathSymbol{\Num}{\mathord}{mcy}{"4E}
 
\hypersetup{citecolor=blue}
\DeclareMathOperator{\ord}{ord}
\renewcommand{\O}{\mathcal{O}}
\newcommand{\Qbar}{\overline{\Q}}
\newcommand{\ie}{\textit{i.e.}}
\newcommand{\cchi}{\hm{\chi}}

\newcommand{\trivcar}{\hm{1}}
\newcommand{\Reg}{\mathrm{Reg}}
\newcommand{\gp}{\mathfrak{p}}

\newcommand{\Gauor}[1]{\mathbf{G}\left(#1\right)}
\newcommand{\Gauss}[3]{\mathrm{G}_{#1}{\left(#2, #3\right)}}
\newcommand{\oomega}{\hm \omega}
\newcommand{\llambda}{\hm \lambda}
\newcommand{\pval}{\nu_{\mathfrak{p}}}

\newcommand{\charpol}[2]{\det\left(1- {#1}\, T \left| #2\right.\right)}

\newcommand{\Jcal}{\mathcal{J}}
\newcommand{\TrivLat}{\Lambda}
\DeclareMathOperator{\Frob}{Fr}

\definecolor{nicolor}{RGB}{0,0,205}

\usepackage{hyperref}

\title{On the arithmetic of a family of superelliptic curves}
\author{Sarah Arpin, Richard Griffon, Libby Taylor, and Nicholas Triantafillou}

\begin{document}
\maketitle
\begin{abstract}
    Let $p$ be a prime, let $r$ and $q$ be powers of $p$, and let $a$ and $b$ be relatively prime integers not divisible by $p$. Let $C/\mathbb F_{r}(t)$ be the superelliptic curve with affine equation $y^b+x^a=t^q-t$. 
    Let $J$ be the Jacobian of $C$. By work of Pries--Ulmer \cite{PriesUlmer2016}, $J$ satisfies the Birch and Swinnerton-Dyer conjecture (BSD).
    Generalizing work of Griffon--Ulmer \cite{GriffonUlmer}, we compute the $L$-function of $J$ in terms of certain Gauss sums. 
    In addition, we estimate 
    several arithmetic invariants of $J$ appearing in BSD, including the rank of the Mordell--Weil group $J(\mathbb F_{r}(t))$, the Faltings height of $J$, and the Tamagawa numbers of $J$ in terms of the parameters $a,b,q$. 
    For any $p$ and $r$, we show that for certain $a$ and $b$ depending only on $p$ and $r$, these Jacobians provide new examples of families of simple abelian varieties of fixed dimension and with unbounded  analytic and algebraic rank as $q$ varies through powers of $p$.  
    Under a different set of criteria on $a$ and $b$, we prove that the order of the Tate--Shafarevich group $\Sh(J)$ grows quasilinearly in $q$ as $q \to \infty.$
\end{abstract}

\section{Introduction} 

Let $p$ be a prime number, let $r$ be a power of $p$, let $\mathbb F_{r}$ denote the finite field with $r$ elements, and let $K = \mathbb \F_{r}(t)$. Let $J/K$ be a principally polarized abelian variety of dimension $g$.

The Birch and Swinnerton-Dyer conjecture (abbreviated as BSD in what follows) is a sweeping statement that predicts a relationship between several important analytic and arithmetic quantities associated to $J$. On the analytic side, the central object of study is the $L$-function $L(J,T)$, a meromorphic function on the complex plane which encodes the action of Frobenius elements.
 
The order of vanishing $\ord_{T=r^{-1}}L(J,T)$ of $L(J,T)$ at the `central point' and the leading coefficient $L^*(J)$ of $L(J,T)$ expanded as a power series at $T = r^{-1}$ are of particular interest.
On the arithmetic side, $J(K)$ is a finitely generated abelian group by the Mordell--Weil theorem. Its rank, $\rank J(K) \colonequals \dim_{\mathbb Q} J(K) \otimes \mathbb Q$ is conjectured to equal $\ord_{T=r^{-1}}L(J,T)$. Other terms include the size of the torsion subgroup $J(K)_{\mathrm{tors}}$, the regulator $\Reg(J)$, the Tate--Shafarevich group $\Sh(J)$, the local Tamagawa numbers $c_{v}(J)$, and the exponential Faltings height $H(J)$.  
In this article, we study the BSD invariants for a family of abelian varieties $J/K$, which we now describe.

Let $q$ be a power of $p$ and let $a,b> 1$ be coprime integers which are both coprime to $p$. Let $C/K$ be the unique (up to isomorphism) smooth projective curve containing the affine curve defined by
\begin{equation}\label{MainEq}
  y^b+x^a=t^q-t
\end{equation}
as a dense open subset.  The curve $C$ is a cyclic Galois cover of $\mathbb P^1$, i.e. a \emph{superelliptic} curve. 
Let $J$ be the Jacobian of $C$. Since $J$ satisfies BSD by  \cite[Corollary~3.1.4]{PriesUlmer2016}, it is 
particularly interesting to study its $L$-function and BSD invariants.

Our main results include: an explicit formula for $L(J,T)$ in terms of Gauss sums, an analogue of the Brauer--Siegel theorem relating the asymptotic growth of $\Sh(J), \Reg(J)$, and $H(J)$ for $J$, and a criterion on $a$ and $b$ depending only on $r$ so that $\rank J(K)$ grows quasi-linearly in $q$. This last result provides new explicit examples of families of simple abelian varieties of fixed dimension, but unbounded rank. Under different criteria on $a$ and $b$, we prove that $\rank J(K) = 0$ and (via our Brauer--Siegel analogue for $J$) that the order of the Tate--Shafarevich group $\Sh(J)$ is unbounded as $q\to \infty$. In fact, by computing the Faltings height $H(J)$, we are able to provide explicit asymptotics for $\Sh(J) \cdot \Reg(J)$ more generally.

We also study a number of other arithmetic and geometric properties of $J$. For instance, we show that $J$ is simple if and only if $a$ and $b$ are both primes. We also compute the minimal proper regular simple normal crossings model of $J$ (using the method described in \cite{Dokchitser2018}) and apply it to show that at any place $v$ of bad reduction, $J$ has unipotent reduction, to determine that the Tamagawa numbers $c_{v}$ of $J$ are all equal to $1$, to compute the conductor $N(J)$, and to give an explicit formula for the the Faltings height of $J$.

In Section~\ref{sec:bsd_conj_for_j}, we include a discussion of the Birch and Swinnerton-Dyer conjecture for $J$:

\begin{theorem}\label{thm:bsd}
The Jacobian $J$ of $C$ satisfies the Birch and Swinnerton-Dyer conjecture. That is:

\begin{itemize}
    \item The algebraic and analytic ranks of $J$ coincide:
    $\ord_{T=r^{-1}}L(J,T)=\rank J(K)$.
    \item The Tate--Shafarevich group $\Sh(J)$ is finite.
    \item The BSD formula holds:
    \begin{equation}\label{eq:BSD.formula}
    L^*(J)=\frac{|\Sh(J)|\, \Reg(J)\, \prod_v c_v(J)}{H(J)\, r^{-g}\, |J(K)_{\mathrm{tors}}|^2},
    \end{equation}
    where the $c_v(J)$ are the local Tamagawa numbers of $J$ and $\Reg(J)$ is the regulator.
\end{itemize}
\end{theorem}

Theorem~\ref{thm:bsd} follows from \cite[Theorem 3.1.2]{PriesUlmer2016}. In our setting, BSD opens up a powerful analytic approach to computing $\rank J(K)$. The strategy is to determine the $L$-function sufficiently explicitly so that one can compute/bound $\ord_{T=r^{-1}}L(J,T)$. In several cases, this strategy has led to new families of abelian varieties of fixed dimension but with unbounded rank.  In~\cite{Ulmer2002}, Ulmer used this strategy to produce the first non-isotrival families of elliptic curves over $\F_p(t)$ satisfying BSD and with arbitrarily large analytic rank. (Isotrivial families of elliptic curves over $\F_{p}(t)$ with unbounded rank had previously been constructed by more algebraic methods in \cite{tate1967rank}.) In~\cite{Ulmer2006}, Ulmer proves an analogue of the previous results for abelian varieties of larger dimension; in particular, he proves that for every $g>0$ and for every prime $p$, there is an absolutely simple, non-isotrivial abelian variety of dimension $g$ over $\F_p(t)$ satisfying BSD and of arbitrarily large analytic rank. These two papers use Kummer towers of field extensions to produce the abelian varieties. In \cite{BHPPPSSU2015}, the authors prove similar results for another family of curves over function fields. They develop new algebro-geometric techniques involving explicit subgroups of divisors on the Jacobian over towers of function fields, thereby expanding the tools used to study curves of arbitrary genus over function fields.

Following \cite{GriffonUlmer}, we compute the $L$-function using two different techniques: once using the arithmetic of Gauss sums (Section~\ref{sec:Lfunction}) and a second time via a cohomological computation (Section~\ref{sec:Cohomological.Lfunction}). In \cite{GriffonUlmer}, the authors were able to apply results of Shioda \cite{shioda1992} to compute the $L$-functions of their family of elliptic curves. Since Shioda's results depend upon the classification of reduction types of elliptic curves, they do not apply directly to higher genus curves, such as our family of superelliptic curves. Fortunately, we have a detailed description of the minimal proper regular SNC model (Section~\ref{sec:geometry}), which we use to extend Shioda's argument to compute the $L$-function of our family.

Other work has studied ranks of Jacobians of curves when the field varies in Artin--Schreier towers, which corresponds to varying $q$ in our setup. Given rational functions $f, g \in \F_{r}(t)$, \cite{PriesUlmer2016} includes a study of curves with affine model $f(x) - g(y) = t^q - t$. Under genericity conditions on $f$ and $g$, including critical points having multiplicity $1$ and restrictions on the order of poles, they prove that the rank of the Jacobian is unbounded as $q$ varies through powers of $p$. The case $f(x) = x^2$ satisfies their genericity assumptions, so their work applies to generic hyperelliptic curves. However, the critical points of $f(x) = x^a$ are not generic when $a > 2$, so their work does not apply to most superelliptic curves. In fact, \cite{PriesUlmer2016} shows that many families of superelliptic curves over $\F_{r}(t)$ have Jacobians with bounded rank as $q$ varies. More recently, \cite{GriffonUlmer} studied the family of elliptic (and superelliptic) curves with affine model $y^2 = x^3 + t^q - t$. In this case, they show that, as $q$ varies, either the the rank is always $0$ or the rank is unbounded, depending only on the congruence class of $p$ modulo $6$.

In this article, we generalize the work of \cite{GriffonUlmer}, showing that the rank of $J$ is sometimes $0$ and sometimes unbounded as $q$ varies, depending on $r$, $a$ and $b$. To state our results, we define $o_{p}(n)$ to be the order of $p$ in $\mathbb Z/n\mathbb Z$ and recall that an integer $n$ is said to be \emph{supersingular} for $p$ if some power of $p$ is congruent to $-1$ modulo $n$. Note that if $n$ is supersingular for $p$, then $o_{p}(n)$ is automatically even.

In Section~\ref{sec:rank.0}, we prove:
\begin{theorem}\label{thm:rankzero}
Suppose that the pair $(a,b)$ satisfies one of the following:
\begin{enumerate}[(1)]
    \item $a o_p(a)$ and $b o_p(b)$ are relatively prime;
    \item  $a o_p(a)$ is odd, and $b$ is supersingular for $p$; or 
    \item  $a$ is supersingular for $p$, and $b o_p(b)$ is odd.
\end{enumerate}
Then, for any power $q$ of $p$, we have 
$\ord_{T=r^{-1}} L(J,T) = \rank J(K) = 0$.
\end{theorem}

For any prime $p$, the hypotheses of Theorem~\ref{thm:rankzero} are satisfied for infinitely many pairs of primes $a, b$, as we show in  Lemma~\ref{lem:infinitude-of-bad-omega}. In Section~\ref{sec:rank.large}, we prove:
\begin{theorem}
\label{thm:analytic-rank-lower-bounds}
Let $p \neq 2$ be an odd prime.
Let $a$ and $b$ be relatively prime positive integers which are both supersingular for $p$. Let $\nu_a, \nu_b\geq 1$ be the least positive integers such that $p^{\nu_{a}} \equiv -1 \pmod{a}$ and $p^{\nu_{b}} \equiv -1 \pmod{b}$.
Suppose also that $[\F_{r}:\F_{p}]$ is a multiple of both $4\nu_{a}$ and $4\nu_{b}$.

Then, we have
\[
(a-1)(b-1) \left \lceil \frac{1}{\log_p(q)}\left( \frac{q-1}{ab}  - \frac{p\sqrt{q} - 1}{p-1}\right)\right \rceil \leq \rank J(K)\,.
\]
\end{theorem}

For any $p$, there are infinitely many pairs of primes $a,b$ satisfying the hypotheses of Theorem~\ref{thm:analytic-rank-lower-bounds}. Fixing such a pair, as $q$ varies among powers of $p$, Theorem~\ref{thm:analytic-rank-lower-bounds} gives a family of Jacobians of fixed dimension satisfying BSD with unbounded rank.  When $a$ and $b$ are both prime, Theorem~\ref{thm:analytic-rank-lower-bounds} actually gives a family of \emph{simple} abelian varieties with these properties, which we prove in Section~\ref{sec:simplicity}:

\begin{theorem}\label{thm:simplicity}
The Jacobian of $y^b+x^a=t^q-t$ is simple over $\F_r(t)$ if and only if both $a$ and $b$ are prime.
\end{theorem}

Our other major results focus on understanding the BSD invariants and other properties of $C$ and $J$ via their geometry. Most notably, we show that many of these Jacobians are simple abelian varieties with Tate--Shafarevich group unbounded as $q$ varies. Recall that $H(J)$ is the exponential Faltings height of $J$. In Section~\ref{sec:sha}, we prove that for infinitely many $a, b$, the size of $\Sh(J)$ is asymptotic to $H(J)$. 
\begin{theorem}\label{thm:largesha}
Fix parameters $a,b$, and $r$ which satisfy the hypotheses of Theorem~\ref{thm:rankzero}. 
Then, as $q$ runs through powers of $p$, we have 
\[
|\Sh(J)| = H(J)^{1+ o(1)}.
\]
\end{theorem}
Moreover, in Lemma~\ref{lem:height-computation} we show that there is a positive constant $D$ depending only on $a$ and $b$ and a positive constant $E$ depending only on $a$, $b$, and the residue class of $q \mod{ab}$ such that $H(J) = r^{Dq + E}$. In particular, the order of $\Sh(J)$ grows exponentially in $q$ as $q$ varies.

Theorem~\ref{thm:largesha} generalizes~\cite{GriDW20}, which exhibits sequences of elliptic curves over $\F_q(t)$ with arbitrarily large Tate--Shafarevich group, to simple abelian varieties of dimension greater than $1$.

We remark briefly that in contrast to our results in the function field setting, much less is known over number fields, and especially over $\mathbb Q$. Work of Clark and Sharif~\cite{clark2010period} (in the elliptic curve case) and of Creutz~\cite{creutz2011potential} (in the higher-dimensional case, building on previous work of Clark) shows that all principally polarized abelian varieties satisfying a certain technical hypothesis have arbitrarily large $\Sh$ after a suitable extension of the base field. If one restricts the ground field to $\Q$, work of Cassels in the 1960s \cite{cassels1964arithmetic} showed that when $A/\mathbb Q$ is an elliptic curve, $\Sh(A/\mathbb Q)$ can be arbitrarily large. Recent work of Flynn~\cite{flynn2018arbitrarily} extends this to abelian surfaces, but it is not known whether $\Sh(A/\Q)$ can be arbitrarily large when $A$ is a simple abelian variety of dimension greater than $2$.

In contrast, in the function field setting, our results give new examples of simple, principally polarized abelian varieties $A$ of arbitrarily large dimension over $\F_{p}(t)$ and with $\Sh(A/\F_{p}(t))$ arbitrarily large. Previously, the only known examples of such abelian varieties appeared in work of Ulmer~\cite{Ulmer2019}.

The proof of Theorem~\ref{thm:largesha} contains several statements which are of interest in their own right. For instance, in Section~\ref{sec:specialvalue},
we describe the asymptotics of the special value of the $L$-function as $q \to \infty$ via analytic methods, generalizing results from ~\cite{GriffonUlmer} in the elliptic curve case. We prove:
\begin{theorem}\label{thm:specialvalue} 
For fixed $a,b$, and $r$, as $q \to \infty$ runs through powers of $p$, 
\[
\frac{\log L^*(J)}{\log H(J) }=o(1).
\]
\end{theorem}
In particular, note that this theorem does not require special assumptions on $a$ and $b$.

On the algebraic side, we are able to compute many BSD invariants of $J$ by studying the geometry of $C$. To begin, we use recent machinery from \cite{Dokchitser2018} to compute the minimal regular proper simple normal crossings model of our curves at any place of bad reduction. In our case, the special fibers of these models have a very simple structure --- all irreducible components have genus $0$ and the dual graph is a tree. From this information, we are able to conclude that $J$ has unipotent reduction at all bad places, to show that the local Tamagawa numbers $c_v(J)$ of $J$ are all equal to $1$, and to compute the conductor divisor of $J$. We also leverage the recipe from~\cite{Dokchitser2018} to compute a formula for the Faltings height $H(J)$ in Lemma~\ref{lem:height-computation}.

Combining these computations with Theorem~\ref{thm:specialvalue}, we deduce an
analogue of the Brauer--Siegel theorem for the family of Jacobians $(J_{a,b,q})_q$. (See \cite{HindryPacheco} for a nice explanation of the connection with Brauer--Siegel.) 
In Section~\ref{sec:BrauerSiegel} we prove:

\begin{corollary}\label{cor:brauersiegel}
For fixed $a,b$, and $r$, as $q\to\infty$ runs through powers of $p$,
\[\log\big(|\Sh(J)|\,\Reg(J)\big) \sim \log H(J).\]
\end{corollary}

Theorem~\ref{thm:largesha} follows since $\Reg(J) = 1$ when $\rank J(K) = 0$.

Several sequences of elliptic curves are known to satisfy a similar asymptotic description of   $|\Sh(A)|\Reg(A)$ in terms of the height $H(A)$ as in Corollary~\ref{cor:brauersiegel}. (For instance, see \cite{HindryPacheco, GriffonPHD16, GriffonLegendre18, GriffonAS18, GriffonUlmer}.) However, similar results for simple abelian varieties of higher dimension are much rarer. The only previous examples we are aware of appear in \cite[\S10.4, \S11.4]{Ulmer2019}.

\subsection{Roadmap to this article.}

The paper is organized as follows.  In Section~\ref{sec:geometry}, we study the geometry of $C$ and use~\cite{Dokchitser2018} to compute the minimal regular proper simple normal crossings model of our curves.  This model is used to compute the reduction types, Tamagawa numbers, and Faltings height of these curves.  We also prove Theorem~\ref{thm:simplicity} on the simplicity of $J$ in Section~\ref{sec:geometry}.  In Section ~\ref{sec:gausssums}, we recall classical results on Gauss sums which will be used in the computation of the $L$-function.  In Section~\ref{sec:Lfunction}, we give an explicit computation for the $L$-function of the Jacobian in terms of the valuations of some associated Gauss sums. In Section~\ref{sec:Cohomological.Lfunction}, we provide a second computation of the $L$-function of the Jacobian, this time using the geometry of the minimal proper regular SNC model $\mathcal{S}$ of $C$, confirming our computation in the previous section.  In Section~\ref{sec:rank}, we use $p$-adic valuations of Gauss sums to prove estimates on $\rank J(K)$ in Theorems~\ref{thm:rankzero} and ~\ref{thm:analytic-rank-lower-bounds}.  In Section~\ref{sec:specialvalue} we prove our asymptotic formula for $L^*(J)$ in Theorem~\ref{thm:specialvalue} and our analogue of Brauer--Siegel in Corollary~\ref{cor:brauersiegel}.  Finally, in Section~\ref{sec:sha}, we prove Theorem~\ref{thm:largesha} giving infinitely many families of simple abelian varieties with unbounded $\Sh(J)$ as $q$ varies.

\subsection*{Acknowledgements}

We thank the AMS and the organizers of the 2019 Mathematics Research Communities workshop on \textit{Explicit Methods in Characteristic $p$} for creating a productive working environment in which this project was started.  We thank Douglas Ulmer for his guidance and support during the realization of this project, and for his helpful comments on a previous draft. Thanks are also due to Daniel Litt for providing help with the proof in Appendix~\ref{app:Conductor} and to Rachel Pries and Dino Lorenzini for their careful reading and helpful comments.

The second author was funded by the Swiss National Science Foundation through the SNSF Professorship \#170565 awarded to Pierre Le Boudec, and received additional funding from ANR project ANR-17-CE40-0012 (FLAIR).
The third author was supported by an NSF graduate research fellowship.
The fourth author thanks the National Science Foundation Research Training Group in Algebra, Algebraic Geometry, and Number Theory at the University of Georgia [grant DMS-1344994] for funding this research.

\section[]{Geometry of $C$ and its Jacobian}\label{sec:geometry}

Fix a prime $p$, and let $r$ be a power of $p$. 
Let $\F_r$ be the finite field with $r$ elements, and let $K:=\F_r(t)$ denote the function field of the projective line $\P^1_{\F_r}$. When the field of definition is understood, we write $\mathbb{P}^1$ for $\P^1_{\F_r}$. For any power $q$ of $p$, and any pair  of relatively prime integers $a, b > 1$ which are both coprime to $p$, consider the superelliptic curve $C_{a,b,q}$ over $K$ given by the affine model
\begin{equation*}
    C_{a,b,q}:\qquad y^b + x^a = t^q-t.
\end{equation*}
In other words, $C_{a,b,q}$ is the unique (up to a birational morphism) smooth projective curve over $K$ which contains the affine curve $y^b + x^a = t^q-t$ as a dense open subset.
Let $J_{a,b,q}$ denote the Jacobian variety of $C_{a,b,q}$, which is an abelian variety over $K$.

Throughout the paper, the curve $C_{a,b,q}$ is denoted by $C$, and its Jacobian $J_{a,b,q}$ by $J$. 
We suppress the ``$/K$'' in the notation for invariants of $C$ and $J$, since both of these objects will only be studied over $K$.

\begin{proposition}
The genus of the curve $C=C_{a,b,q}$ 
is $g = (a-1)(b-1)/2$.
\end{proposition}
\begin{proof} The result follows from a direct computation using the Hurwitz genus formula and the assumption that $a$ and $b$ are coprime.
\end{proof}

We prove various geometric properties about $C$ and $J$ in this section. In particular, we use the minimal proper regular SNC model of C to prove that $J$ has unipotent reduction at each place of bad reduction. For more specific information about the reduction type in the elliptic curve case, see \cite{GriffonUlmer}. We also compute the height of $J$, and prove that it is $K$-simple for when both $a$ and $b$ are prime.

\subsection[]{The minimal proper regular SNC model of $C$}\label{subsec:MinPropRegSNCModel}

In this section, we give a brief description of the minimal proper regular simple normal crossings model $\pi: \mathcal S \to \P^1_{\mathbb F_{r}}$ of $C/\mathbb F_{r}(t)$ using the recipe provided in \cite{Dokchitser2018}. This description allows us to read off the reduction of the Jacobian of $J$ at the places of bad reduction, which will in turn be necessary for the computation of the $L$-function. It is also useful for computing the Tamagawa numbers, exponential Faltings height, and conductor of $J$.

We will use notation from \cite{Dokchitser2018} freely throughout this section. The results presented here could alternately be recovered via a toric resolution of singularities.

We now recall the definition of a simple normal crossings model. We note that some authors call this a strict normal crossings model instead. First, recall (e.g. from \cite[\href{https://stacks.math.columbia.edu/tag/0CBN}{Section 0CBN, Definition 41.21.1}]{stacks-project}) that a \emph{simple normal crossings divisor} on a locally Noetherian scheme $\mathcal W$ is an effective Cartier divisor $D \subset \mathcal W$ such that for every prime $w \in D$, the local ring $\mathcal O_{W,w}$ is regular and there exists a regular system of parameters $x_{1}, \dots, x_{d}$ in the maximal ideal $\mathfrak m_{w}$ and $1 \leq r \leq d$ such that $D$ is cut out by the product $x_1 \cdots x_{r}$ in $\mathcal O_{X,p}$. When $\mathcal W$ is a curve over a DVR or a surface over a finite field, these conditions amount to saying that the irreducible components of $D$ are smooth and any singular points of $D$ `look like' the intersection of the coordinate axes in $\mathbb A^2$. More generally, an effective Cartier divisor $E$ on $\mathcal W$ is \emph{supported on a simple normal crossings divisor} if there is some simple normal crossing divisor $D$ on $\mathcal W$ such that $E \subset D$ set-theoretically. In this situation, if $D$ decomposes into irreducible components as $\bigcup_{i \in I} D_{i}$, then $E = \sum_{i \in I} a_{i} D_{i}$ for some integers $a_{i} \geq 0$.

\begin{definition}
Given a smooth proper curve $W$ over the fraction field $K_{v}$ of a discrete valuation ring $\mathcal O_{K_{v}}$, a \emph{simple normal crossings model} of $W$ is a scheme $\mathcal W$ over $\mathcal O_{K_v}$ such that the generic fiber $\mathcal W_{K_{v}}$ is isomorphic to $W$ and the special fiber $\mathcal W_{k_{v}}$, viewed as a Cartier divisor on $\mathcal W$, is supported on a simple normal crossing divisor.

More generally, given a smooth proper curve $W/\mathbb F_{r}(t)$, a \emph{simple normal crossings model} of $W$ is a surface $\mathcal W/\mathbb F_{r}$ equipped with a map $\pi: \mathcal W \to \mathbb P^1_{\F_{r}}$ such that the fiber over the generic point of $\mathbb P^1_{\F_{r}}$ is isomorphic to $W$ and the fiber $\mathcal W_{v}$ over any closed point of $v \in \mathbb P^1_{\F_{r}}$ is supported on a simple normal crossings divisor of $\mathcal W$.
\end{definition}

For $v \in \P^1$ a closed point, we study the fiber $\mathcal S_{v}$ of the minimal proper regular simple normal crossings model $\pi: \mathcal S \to \P^1_{\mathbb F_{r}}$ of $C/\mathbb F_{r}(t)$. Taking $K_{v}^{\mathrm{unram.}}$ to be the maximal unramified extension of the completion of $K$ at $v$, we will also describe the special fiber of the minimal proper regular simple normal crossings model of the base change $C \otimes_{\Spec K} \Spec K_{v}^{\mathrm{unram.}}$. We call this special fiber $\mathcal S_{\overline{v}}$. As we shall see, $\mathcal S_{\overline{v}} \cong \mathcal S_{v} \otimes_{\Spec k_{v}} \Spec \overline{k_{v}}$.

We abuse notation slightly by writing $v \in \F_{q} \cup \{\infty\}$ to mean that $v$ decomposes into degree one points over the compositum $\F_{r}\F_{q}$.
Equivalently, $v \in \F_{q} \cup \{\infty\}$ if every element of $v(\overline{\F_{q}})$ is fixed by the $\Gal(\overline{\F_{q}}/\F_{q})$-action on $\P^1(\overline{\F_{q}})$.

When $v \notin \F_{q} \cup \{\infty\}\subset \P^1$, the curve $C$ has good reduction, so $\mathcal S_{v}/k_{v}$ and $\mathcal S_{\overline{v}}/\overline{k_{v}}$ are smooth curves of genus $g$.

When $v \in \F_{q} \cup \{\infty\} \subset \P^1$, the curve $C$ has bad reduction at $v$. 
Set $Q = 1$ if $v \in \mathbb F_{q}$ and $Q = -q$ if $v = \infty$. In the notation of \cite{Dokchitser2018}, the Newton polytopes associated to $C$ at $v$ are 
\[
\Delta = \text{convex hull}(\{(0,0), (a,0), (0,b)\}) \subset \mathbb R^2
\]
and
\[
\Delta_{v} = \text{lower convex hull}(\{(0,0,Q), (a,0,0), (0,b,0)\}) \subset \mathbb R^2 \times \mathbb R \,.
\]
The polytope $\Delta_{v}$ consists of three $0$-dimensional vertices $(a,0,0), (0,b,0),$ and $(0,0,Q)$;  three $1$-dimensional (open) edges
\begin{itemize} 
    \item $L_{3}$ connecting $(a,0,0)$ to $(0,b,0)$ with denominator $\delta_{L_3} = 1$,
    \item $L_{2}$ connecting $(0,b,0)$ to $(0,0,Q)$ with denominator $\delta_{L_2} = b$, and
    \item $L_{1}$ connecting $(a,0,0)$ to $(0,0,Q)$ with denominator $\delta_{L_1} = a$; and
\end{itemize}
a single $2$-dimensional (open) face $F$ with denominator $\delta_{F} = ab$. Moreover, $F(\mathbb Z)_{\mathbb Z} \subset F \cap \mathbb Z^3 = \emptyset$, so $|F(\mathbb Z)_{\mathbb Z}| = 0\,.$ The face-polynomial $X_{F}$ and the side polynomials $X_{L_{i}}$ are all smooth, so $C$ is $\Delta_{v}$-regular, as defined in \cite[Definition~3.9]{Dokchitser2018}. As a result, we can read off the structure of $\mathcal S_{v}$ using \cite[Theorem~3.13]{Dokchitser2018}. 

We find that $\mathcal S_{v}$ consists of three chains of $\P^1$s (corresponding to the edges $L_1, L_2,$ and $L_{3}$) branching off of a central curve corresponding to the face $F$. Since the interior of $F$ contains no lattice points, $|F(\mathbb Z)_{\mathbb Z}| = 0\,.$ Moreover, $\delta_{F} = ab$, so the central curve has genus $0$ and multiplicity $ab.$ For $i = 1,2,3$, every curve in the chain of $\P^1$s corresponding to $L_{i}$ has multiplicity a multiple of $\delta_{i}$. The final curve in the chain has multiplicity exactly $\delta_{i}$. For a more precise description of the multiplicities of the components, see \cite{Dokchitser2018}. We give an examples of the resulting special fiber $\mathcal S_{v}$ when $v$ is a  finite place of bad reduction or $v = \infty$ in the case $a = 7, b = 5, q = 67$ in Figure~\ref{fig:SNC_03}.

Moreover, we note that the Newton polytopes associated to $C \otimes_{\Spec K} K_{v}^{\mathrm{unram.}}$ are the same as those associated to $C$ at $v$. In particular, $\mathcal S_{\overline{v}}$ admits the same description as a tree of $\mathbb P^1$s with multiplicity as does $\mathcal S_{v}$. It follows immediately that $\mathcal S_{\overline{v}}$ is obtained from $\mathcal S_{v}$ via base change to $\overline{k_{v}}$. More precisely, $\mathcal S_{\overline{v}} \cong \mathcal S_{v} \otimes_{\Spec k_{v}} \Spec \overline{k_{v}}$.

For later use, we note that the final component in $\mathcal S_{v}$ of the chain corresponding to $L_{3}$ always has multiplicity $1$. In particular, the gcd of the multiplicities of the components of $\mathcal S_{v}$ is $1$. This means that $\mathcal S_{K_{v}^\mathrm{unram.}}$ is a $(\Spec\, \O_{K_{v}^{\mathrm{unram.}}})$-curve (or $S$-curve) in the notation of \cite{Lorenzini1990}. 
 
   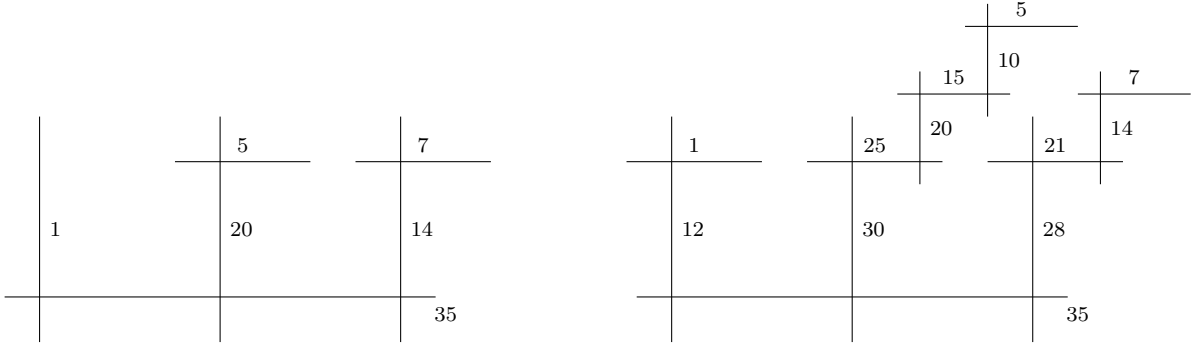
\begin{figure}[h]
        \centering
        \begin{tikzpicture}[scale=0.6,auto=center]
        \node (n1) at (0,10) {};
        \node (n2) at (10,10) {};
        \draw (1,9) -- (1,14)
        node[draw=none,fill=none,font=\scriptsize, midway, right] {$1$};
        \draw (n1) -- (n2)
        node[draw=none,fill=none,font=\scriptsize,right,below] {$35$};
        \draw (5,9) -- (5,14)
        node[draw=none,fill=none,font=\scriptsize,midway,right] {$20$};
        \draw (9,9) -- (9,14)
        node[draw=none,fill=none,font=\scriptsize,midway,right] {$14$};
        \draw (4,13) -- (7,13)
        node[draw=none,fill=none,font=\scriptsize,midway,above] {$5$};
        \draw (8,13) -- (11,13)
        node[draw=none,fill=none,font=\scriptsize,midway,above] {$7$};
        \end{tikzpicture}
        \qquad \qquad
        \begin{tikzpicture}[scale=0.6,auto=center]
        \node (n1) at (0,10) {};
        \node (n2) at (10,10) {};
        \draw (0,13) -- (3,13)
        node[draw=none,fill=none,font=\scriptsize, midway,above] {$1$};
        \draw (1,9) -- (1,14)
        node[draw=none,fill=none,font=\scriptsize, midway, right] {$12$};
        \draw (n1) -- (n2)
        node[draw=none,fill=none,font=\scriptsize,right,below] {$35$};
        \draw (5,9) -- (5,14)
        node[draw=none,fill=none,font=\scriptsize,midway,right] {$30$};
        \draw (9,9) -- (9,14)
        node[draw=none,fill=none,font=\scriptsize,midway,right] {$28$};
        \draw (4,13) -- (7,13)
        node[draw=none,fill=none,font=\scriptsize,midway,above] {$25$};
        \draw (6.5,12.5) -- (6.5,15)
        node[draw=none,fill=none,font=\scriptsize,midway,right] {$20$};
        \draw (6,14.5) -- (8.5,14.5)
        node[draw=none,fill=none,font=\scriptsize,midway,above] {$15$};
        \draw (8,14) -- (8,16.5)
        node[draw=none,fill=none,font=\scriptsize,midway,right] {$10$};
        \draw (7.5,16) -- (10,16)
        node[draw=none,fill=none,font=\scriptsize,midway,above] {$5$};
        \draw (8,13) -- (11,13)
        node[draw=none,fill=none,font=\scriptsize,midway,above] {$21$};
        \draw (10.5,12.5) -- (10.5,15)
        node[draw=none,fill=none,font=\scriptsize,midway,right] {$14$};
        \draw (10,14.5) -- (12.5,14.5)
        node[draw=none,fill=none,font=\scriptsize,midway,above] {$7$};
        \end{tikzpicture}
        \caption{Fibers of the minimal proper regular SNC model of $y^5 + x^7 = t^{67} - t$ over $\mathbb P^1_{\mathbb F_{67}}$ at finite places of bad reduction (left) and at infinity (right)}
        \label{fig:SNC_03}
    \end{figure}

\subsection[]{Unipotent reduction of $J$ at bad places.}

We give an analysis of the reduction types of $J$ at the finite places and the infinite place. 

\begin{proposition}\label{prop:unipotentreduction}
The Jacobian $J$ has potentially good, unipotent reduction above any $v \in \mathbb F_{q} \cup \{\infty\}\subset\P^1$, and it has good reduction elsewhere.
\end{proposition}

\begin{proof}
The roots of $t^q - t$ lie in $\mathbb F_{q}$, so $C$ has good reduction away from $\mathbb F_{q} \cup \{\infty\}$. Moreover, $C$ is isotrivial and becomes isomorphic to $y^b + x^a = 1$ over $\mathbb F_{r}(\hspace{-3pt}\sqrt[ab]{t^q - t})\,$ so $C$ has potentially good reduction everywhere.

When $v \in \mathbb F_{q} \cup \{\infty\}$, we can read off the reduction of the Jacobian from the special fiber of the simple normal crossings model $\mathcal S$. Write $\mathcal J/\F_{r}$ for the (global) N\'{e}ron model of $J$. Given a point $v \in \mathbb P^1$, let $k_{v}$ denote the residue field at $v$ and let $\mathcal J_{v}^{0}$ denote the connected component of the identity of the fiber of $\mathcal J$ above $v$. 

Similarly, let $\mathcal J_{\overline{v}}^{0}$ denote the connected component of the identity of the special fiber of the N\'{e}ron model of the base change $J_{K_{v}^{\mathrm{unram.}}}$. Since $\mathcal S_{\overline{v}} \cong \mathcal S_{v} \otimes_{\Spec k_{v}} \Spec \overline{k_{v}}$, we have $\mathcal J_{\overline{v}}^{0} \cong (\mathcal J_{v} \otimes_{\Spec k_{v}} \Spec \overline{k_{v}})^{0}$. The advantage of passing to a N\'{e}ron model over $K_{v}^{\mathrm{unram.}}$ is that we may apply results from \cite{Lorenzini1990}, which requires an algebraically closed residue field. 

We recall some facts on the structure of $\mathcal J_{\overline{v}}^{0}$ from Section 1 of \cite{Lorenzini1990}.

Above any point $v \in \P^1$, there is a unipotent group scheme $U$, a torus $T$ and an abelian variety $A$ fitting into the following exact sequence of group schemes over $t_{0}$:
\[
0 \to U \times T \to \mathcal J_{\overline{v}}^{0} \to A \to 0\,.
\]

Since the $\mathcal{S}_{v}$ is the special fiber of a simple normal crossings model of a curve over $K_{v}^{\mathrm{unram.}}$, Corollary~1.4 of \cite{Lorenzini1990} states that $\dim(T)$ is equal to the first Betti number of the dual graph of $\mathcal S_{\overline{v}}$. The dual graph of $\mathcal S_{\overline{v}}$ is a tree, so it has trivial homology. Hence, $T$ is trivial.

Also, if $\mathcal S_{\overline{v}}$ has irreducible components $C_{1}, \dots, C_{r}$, then $\dim A = \sum_{i=1}^{r} \text{genus}(C_{i})$. For $v\in \mathbb F_{q} \cup \{\infty\}$, all of the components of $\mathcal S_{\overline{v}}$ have genus $0$, so $\dim A = 0$ as well. 

In summary, for any place $v$ of bad reduction for $C$, the group scheme $\mathcal J_{\overline{v}}^{0}$ is unipotent, since both the toric and abelian parts are trivial. We conclude that, up to twist, the same is true of $\mathcal J_{v}^{0}$.

\end{proof}

\subsection[]{Tamagawa numbers of $J$.}

From our description of the reduction of $J$ at bad places, we deduce an explicit expression for 
another important invariant of $J$: its Tamagawa number. First, recall the definition:

Given an abelian variety $A/K$ and any place $v$ of $K$, let $\mathcal A/\O_{v}$ be the N\'{e}ron model of $A_{K_{v}}$. The special fiber $\mathcal A_{v}$ of $\mathcal A$ may have multiple components. Let $\mathcal A_{v}^{0}$ be the component containing the identity. The quotient $\mathcal A_{v}/\mathcal A_{v}^{0}$ is a finite group scheme.

\begin{definition}[Tamagawa Number]
For any abelian variety $A/K$ and place $v$ of $K$, the \textit{local Tamagawa number} is defined by $c_v(A) \colonequals \#\left(\mathcal A_{v}/\mathcal A_{v}^{0}\right)(k_{v})\,.$  Equivalently, $c_{v}(A)$ is the number of irreducible components of $\mathcal A_{v}/k_{v}$ which remain irreducible after base change to $\overline{k_{v}}$. 
The \textit{Tamagawa number}  $\mathcal{T}(J/K)$ of $J$ is defined as the product $\prod_v c_v(J)$ over all places of $K$.
\end{definition}

\begin{proposition}\label{prop:TamagawaNumber1}
For $J=J_{a,b,q}$, the Tamagawa number $\mathcal{T}(J/K)$ is equal to $1$.
\end{proposition}

This fact is used in Section~\ref{sec:BrauerSiegel}. 
\begin{proof}
If $v$ is a place of good reduction for $J$, then $c_{v}(J) = 1$. 

To compute the local Tamagawa numbers from the simple normal crossings model at each place of bad reduction, we show that $\#\left(\mathcal J_{v}/\mathcal J_{v}^{0}\right)(\overline{k_{v}}) = 1$. Since $1 \leq \#\left(\mathcal J_{v}/\mathcal J_{v}^{0}\right)(k_{v}) \leq \#\left(\mathcal J_{v}/\mathcal J_{v}^{0}\right)(\overline{k_{v}})$, it will follow that $c_{v}(J) = 1$ as well.

Let $\mathcal J_{\overline{v}}$ be the special fiber of the N\'{e}ron model of the base change $J \otimes_{\Spec K_{v}} \Spec K_{v}^{\mathrm{unram.}}$. 
As in the proof of Proposition~\ref{prop:unipotentreduction}, since $\mathcal S_{\overline{v}} \cong \mathcal S_{v} \otimes_{\Spec k_{v}} \Spec \overline{k_{v}}$, we have $\mathcal J_{\overline{v}} \cong \mathcal J_{v} \otimes_{\Spec k_{v}} \Spec \overline{k_{v}}$. 
In particular, we have $\#\left(\mathcal J_{v}/\mathcal J_{v}^{0}\right)(\overline{k_{v}}) \leq \#\left(\mathcal J_{\overline{v}}/\mathcal J_{\overline{v}}^{0}\right)(\overline{k_{v}})$\,.

The advantage of base change to $K_{v}^{\mathrm{unram.}}$ is that we may apply Corollary~1.5 of \cite{Lorenzini1990} to compute the local Tamagawa numbers from the simple normal crossings models at the places of bad reduction. We recall this result here for convenience: If the special fiber of the SNC model is given by $\sum_{i=1}^nr_iC_i$, let $d_i:= \sum_{i\neq j}C_i\cdot C_j$. If the associated Jacobian has toric dimension 0, the local Tamagawa number is given by
\[c_{v}(J) = \prod\limits_{i=1}^n r_i^{d_i-2}.\]

Proposition~\ref{prop:unipotentreduction} says that $\mathcal J_{v}$ (and so also $\mathcal J_{\overline{v}}$) has toric dimension 0, so we may apply this result. We recall the relevant intersection numbers and multiplicities from Section~\ref{subsec:MinPropRegSNCModel}. At each place of bad reduction, there is one fiber of multiplicity $ab$ with 3 intersections, and three fibers of multiplicities $a,b,$ and $1$ with 1 intersection. All other fibers have 2 intersections, so the local Tamagawa number is $\#\left(\mathcal J_{\overline{v}}/\mathcal J_{\overline{v}}^{0}\right)(\overline{k_{v}}) = (ab)^1 a^{-1}b^{-1}1^{-1} = 1$. We conclude that $c_{v}(J) = 1$ as well.

Since all of the local Tamagawa numbers are equal to $1$,  we conclude $\mathcal{T}(J/K) = 1$.
\end{proof}

\subsection[]{Conductor of $J$}

We also use the reduction type of $J$ to compute the conductor divisor $N_J\in\mathrm{Div}(\P^1)$ of $J/K$ in Proposition~\ref{prop:DegOfCond}. In Section~\ref{sec:Lfunction}, we use this computation to verify the degree of $L(J,T)$.  

We refer the reader to \cite{Serre} for the construction of $N_{J}$. Fix, once and for all, a prime $\ell\neq p$ and let $V = V_{\ell}(J)$ be the $\ell$-adic Tate module of $J$ viewed as a representation of $\mathrm{Gal}(\overline{K}/K)$. Given a place $v \in \mathbb P^1$, let $I_{v}$ be the inertia subgroup and denote by $V^{I_{v}}$ the subspace fixed by $I_{v}$.

\begin{proposition}\label{prop:DegOfCond}
The conductor $N_J$ is an effective divisor on $\P^1$, supported on $\F_q\cup\{\infty\}$, with
\[\deg N_J = (a-1)(b-1)(q+1) = 2 g (q+1).\]
\end{proposition}

\begin{proof}
From the definition of $N_{J}$, we see that
\[\deg(N_J)= \sum_{v \text{ bad reduction}}(2g - \dim(V^{I_v})) \deg v\,.\]
By Proposition~\ref{prop:unipotentreduction}, the places of bad reduction of $J$ are exactly those closed points $v$ of $\P^1$ with $v \in \F_{q}\cup\{\infty\}$.  At each of those places, the Jacobian $J$ has unipotent reduction, hence $V^{I_v}$ is trivial by \cite[\S3]{SerreTate}. Therefore, $2g-\dim(V^{I_v})=2g$ at every such place $v$. 
So, 
\[
\sum_{v \text{ bad reduction}}(2g - \dim(V^{I_v})) \deg v = 2g \sum_{v \in \F_{q} \cup\{\infty\}} \deg v = 2g(q+1)\,.
\]
\end{proof}

\subsection[]{Height of $J$}

In this section, we compute the Faltings height of $J$. 
Let $\mathcal J \to \mathbb P^1$ be the (global) N\'{e}ron model of $J/\mathbb F_{r}(t)$. Let $z: \mathbb P^1 \to \mathcal J$ be the identity section. Let $\Omega^g_{\mathcal J/ \mathbb P^1}$ be the relative dualizing sheaf on $\mathcal J$. This sheaf pulls back to a line bundle $\omega_{J} \colonequals z^{*} \Omega^g_{\mathcal J/\mathbb P^1}$ on $\mathbb P^1$. The Faltings height of $J$ is defined as 
\[h(J) \colonequals \deg(\omega_{J}) \]
and the exponential Faltings height of $J$ is defined as $H(J) \colonequals r^{h(J)}$.

\begin{lemma}\label{lem:height-computation}
There is a positive $D \in \mathbb Q$ depending only on $a$ and $b$ and a positive $E \in \mathbb Q$ depending only on $a$, $b$, and the congruence class of $q \bmod{ab}$ such that the Faltings height of $J$ is 
\[
h(J) = D q + E\,.  
\]
The values $D$ and $E$ satisfy
\[
\frac{(ab - a - b)^3}{6 a^2 b^2} < D < \frac{ab}{6}\, \qquad \text{ and } \qquad
0 < E < g_{C}\,.
\]

\end{lemma}

\begin{proof}
Since $J$ is a Jacobian, the Faltings height can be reinterpreted in terms of our regular model $\mathcal S$ for $C$ and the map $\pi: \mathcal S \to \mathbb P^1$. There is a section $s:\mathbb P^1 \to \mathcal S$ which maps $\mathbb P^1$ isomorphically onto the Zariski closure in $\mathcal S$ of the point at infinity on the generic fiber $C$. So, we may apply Proposition~7.4 of \cite{BHPPPSSU2015}, which gives
\[
\omega_{J} \cong \bigwedge^{g} \pi_{*} \Omega^1_{\mathcal S/\mathbb P^1}.
\]

For any integers $i,j\geq 1$, consider the meromorphic differential $\omega_{i,j} \colonequals x^{i-1} y^{j-b} dx\in\Omega^1_{\mathcal S/\mathbb P^1}$.
The set 
\[
\big\{ \omega_{i,j}|_{C} : i > 0, j > 0, \text{ and } ab > bi + aj \big\} 
\]
of differentials restricted to the generic fiber $C$ of $\mathcal{S}\to \P^1$ forms a $K$-basis for $\Omega^1_C$.  
We may thus compute $\deg \omega_{J}$ in terms of the orders of poles/zeros of the relative differential $g$-form on $\mathcal{S}$ defined by 
\[
\eta \colonequals \underset{\substack{(i,j):  i,j > 0 \\ ab > bi + aj \,.}}{\bigwedge}\omega_{i,j} \,.
\]
More precisely, we have
\[
\deg(\omega_{J}) = \sum_{v \in \mathbb P^1} \ord_{v}(\pi_{*} \eta) \deg v\,.
\] 
Since $\pi_{*}\eta$ has finitely many zeros and poles, the sum is finite. Given a point $v$ of $\mathbb P^1$, let $\O_{v}$ denote the local ring at $v$ and let $\mathcal S_{v}$ be the base change of $\mathcal S$ to $\O_{v}$. 
We use \cite[Theorem~8.12]{Dokchitser2018} to understand $\ord_{v}(\pi_{*} \eta)$. For $v \in \mathbb A^1 \subset \mathbb P^1\,,$ set 
\[
V_{i,j,v} = \begin{cases} (ab - bi - aj)/ab \qquad & \text{if} \quad v \in \F_{q}\,,\\
0 \qquad &\text{otherwise.}
\end{cases}
\]
In all cases, $\left\lfloor V_{i,j,v} \right \rfloor = 0$. So, by \cite[Theorem~8.12]{Dokchitser2018} the $\omega_{i,j}|_{\mathcal S_{f}}$ form a $R_{f}$ basis for the relative canonical sheaf on $\mathcal S_{f}$. Hence, the $g$-form $\eta$ is regular and nonvanishing on $\mathcal S_{f}$. In other words, $\ord_{f}(\pi_{*}\eta) = 0$. It follows that $\deg(\omega_{J}) = \ord_{\infty}(\eta)$. 

Set 
\[
V_{i,j, \infty} \colonequals (bi + aj - ab)  \frac{q}{ab} \,.
\]
Taking local parameter $s = t^{-1}$ on the fiber $\mathcal S_{\infty}$ above infinity, Theorem~8.12 of \cite{Dokchitser2018} says that an $\mathbb F_{q}[[s]]$-basis for the relative dualizing sheaf is given by
\[
\{ s^{\left\lfloor V_{i,j,\infty}\right \rfloor} \omega_{i,j} : i > 0, j > 0, ab > bi + aj \} \,.
\]
Hence, 
\[
\ord_{\infty}(\eta) = \sum_{\substack{(i,j):  i,j > 0 \\ ab > bi + aj \,.}} -\left\lfloor V_{i,j,\infty} \right \rfloor = \sum_{\substack{(i,j):  i,j > 0 \\ ab > bi + aj \,.}} -\left\lfloor (bi + aj - ab) \frac{q}{ab} \right \rfloor =  \sum_{\substack{(i,j): i,j>0 \\
ab > bi + aj}} \left \lceil q\frac{ab - (bi + aj)}{ab} \right \rceil \,.
\]
If we set
\[
D \colonequals \sum_{\substack{(i,j): i,j>0 \\
ab > bi + aj}} \frac{ab - (bi + aj)}{ab} 
\]
and 
\[
E \colonequals \sum_{\substack{(i,j): i,j>0 \\
ab > bi + aj}} \left \lceil q  \frac{ab - (bi + aj)}{ab} \right \rceil - q \frac{ab - (bi + aj)}{ab}\,,
\]
then $h(J) = \deg(\omega_{J}) = D q + E$. The definition of $D$ depends only on $a$ and $b$, while $E$ only depends on $a, b$ and the residue class of $q \pmod{ab}$.

To bound $E$, we note that
\[
E = \sum_{\substack{(i,j): i,j>0 \\
ab > bi + aj}} \left \lceil q  \frac{ab - (bi + aj)}{ab} \right \rceil - q\frac{ab - (bi + aj)}{ab} < \sum_{\substack{(i,j): i,j>0 \\
ab > bi + aj}} 1 = g\,.
\]

To bound $D$, we interpret each term $(ab - bi - aj)/ab$ as the volume of a rectangular prism with height $(ab-bi-aj)/ab$ and base a square of side length $1$. If we take as the base the square $[i,i+1] \times [j,j+1]$, then the tops of these prisms lie above the hyperplane $z = (ab - bx - ay)/ab$. If we take as base the square $[i-1,i] \times [j-1,j]$, the tops of these prisms lie below this hyperplane. Hence, we may bound $D$ between the areas of two right triangular pyramids, or equivalently the integrals
\[
\frac{(ab - a - b)^3}{6 a^2 b^2} = \underset{\left\{\substack{(x,y):x,y > 1, \\ ab > bx + ay}\right\}}{\iint} \frac{ab - (bx + ay)}{ab} dxdy < D < \underset{\left\{\substack{(x,y):x,y > 0, \\ ab > bx + ay}\right\}}{\iint} \frac{ab - (bx + ay)}{ab} dxdy =  \frac{ab}{6}\,.
\]
\end{proof}

\begin{remark}
When $a = 2$, we can compute that $D = (b-1)^2/8b$, since
\[
D = \frac{1}{2b}\sum_{j: 0 < ja < b} (b-ja) = \frac{1}{2b} \left(\frac{b-1}{2}\right)^2 = \frac{(b-1)^2}{8b}\,.
\]
\end{remark}

\begin{remark} For a fixed pair $a,b$, note that the ratio $h(J)/q$ is bounded from above and from below by positive constants depending only on $a$ and $b$
as $q$ tends to $+\infty$ through powers of $p$. 
\end{remark}

\subsection{Decomposition of the Jacobian} \label{sec:simplicity} 

In this section, we prove Theorem~\ref{thm:simplicity} on the simplicity of $J$. In Section~\ref{sec:rank.large}, we produce examples of abelian varieties with large rank. In Section~\ref{sec:BrauerSiegel}, we show that our $J$ satisfy a Brauer--Siegel ratio as $q$ varies. In Section~\ref{sec:sha}, we produce examples of abelian varieties with with large order of Tate--Shafarevich. Theorem~\ref{thm:simplicity} shows that all of our examples can be constructed as simple abelian varieties, and are not built as isogeny products of elliptic curves over $K$.

\begin{customthm}{\ref{thm:simplicity}}
The Jacobian $J$ is $K$-simple if and only if $a$ and $b$ are both prime.
\end{customthm}  

Before we begin the proof of Theorem~\ref{thm:simplicity}, we study the $\ell$-adic Tate module of an auxiliary curve. Let $C_{0}/\overline{k}$ be the projective curve with a dense open subset defined by the affine equation 
\[
y^b + x^a = 1\,.
\]
Let $J_{0}/\overline{k}$ be the Jacobian of $C_{0}$. The curve $C_{0}$ admits an action of $\mu_{ab}(\overline{k})$ by $\zeta \cdot (x,y) = (\zeta^b x, \zeta^a y)$. This induces an action of $\mu_{ab}(\overline{k})$ on $J_{0}$ and therefore also on its $\ell$-adic Tate module $V_{\ell}(J_{0})$ for any auxiliary prime $\ell$. We will typically choose $\ell$ not equal to $a,b,$ or $p$. Our first task is to describe $V_{\ell}(J_{0})$ as a representation of the finite abelian group $\mu_{ab}(\overline{k})$.

We begin with an auxiliary lemma.

\begin{lemma}\label{lem:subclaim}
Let $G$ be a finite group, let $X$ be a curve over a field equipped with a $G$ action, let $Y = X/G$, and let $f: X \to Y$ be the quotient map. 
Then, $J_{X}^{G} \sim J_{Y}$. That is, the subabelian variety of $G$-invariants of $J_{X}$ is isogenous to $J_{Y}$.
\end{lemma}

\begin{proof}
Suppose that $P \in J_{X}(\overline{K})^G$ is $G$-invariant. Then, $P$ is represented by some divisor $D$ on $X$, and $\# G \cdot P$ is represented by the $G$-invariant divisor $\sum_{g \in G} g\cdot D$ on $X$, which is the pullback of some divisor on $Y$. In particular, $\# G \cdot P$ is in the image of the finite map $f^{*}: J_{Y} \to J_{X}$. The image of $f^{*}$ is contained in $J_{X}^{G}$. So, the image of $f^{*}$ is finite index in $J_{X}^{G}$. It follows that $J_{X}^{G} \sim J_{Y}$.
\end{proof}

As a consequence of Lemma~\ref{lem:subclaim}, we have a similar result on the level of $\ell$-adic Tate modules. More precisely, we have $V_{\ell}(J_{X})^G = V_{\ell}(J_{Y})$ and $(V_{\ell}(J_{X}) \otimes \overline{\Q_{\ell}})^G = V_{\ell}(J_{Y}) \otimes \overline{\Q_{\ell}}$.

\begin{lemma}\label{lem:claim}
Let $\ell$ be a prime not dividing $abp$ such that $\ell \not\equiv 1$ modulo any odd prime dividing $ab$ and $\ell \not \equiv 1 \pmod{4}$ if $ab$ is even. Then, as a representation of $\mu_{ab}(\overline{k})$, the vector space $V_\ell(J_0)\otimes \overline{\Q_\ell}$ is the direct sum of $2g$ one-dimensional $\overline{\Q}_\ell$-vector spaces indexed by the characters $\chi: \mu_{ab} \to \overline{\Q_\ell}^\times$ which are nontrivial when restricted to both $\mu_a$ and $\mu_b$.  There are $(a-1)(b-1)=2g$ such characters.
\end{lemma}

\begin{proof}
The assumption on $\ell$ implies that for any positive $n$ dividing $ab$, the group $\Gal(\overline{\Q_{\ell}}/\Q_{\ell})$ acts transitively on the set of primitive $n$\textsuperscript{th} roots of unity in $\overline{\Q_{\ell}}$. Note that $\Gal(\overline{\Q_\ell}/\Q_\ell)$ acts on $V_{\ell}(J_{0})\otimes \overline{\Q_\ell}$ by acting trivially on the first factor and by the natural action on the second factor. Since $\mu_{ab}(\overline{k})$ acts only on the first factor, the actions of $\mu_{ab}(\overline{k})$ and  $\Gal(\overline{\Q_{\ell}}/\Q_{\ell})$ commute. 

Since $\mu_{ab}(\overline{k})$ is an abelian group and $\overline{\mathbb Q_{\ell}}$ is an algebraically closed field of characteristic zero,
$V_{\ell}(J_{0})\otimes \overline{\Q_\ell}$ is a direct sum of $2g$ one-dimensional $\overline{\Q_{\ell}}$-representations of $\mu_{ab}$. 
We now consider the multiplicities of these one-dimensional representations.

First, we claim that each representation of a given order occurs to the same multiplicity. 

If $\chi'$ and $\chi$ have the same order, then there is some $c \in (\Z/ab\Z)^\times$ so that $\chi' = \chi^c$. Since $\Gal(\overline{\Q_\ell}/\Q_\ell)$ acts transitively on the primitive $n$\textsuperscript{th} roots of unity for any $n$ dividing $ab$, there is some $\sigma \in \Gal(\overline{\Q_\ell}/\Q_\ell)$ such that $\sigma(\chi(\zeta)) = \chi(\zeta)^c$ for all $\zeta \in \mu_{ab}(\overline{k})$. Now suppose $P \in V_{\ell}(J_{0})\otimes \overline{\Q_\ell}$ satisfies $\zeta \cdot P = \chi(\zeta)$ for all $\zeta \in \mu_{ab}(\overline{k})$. Then,
\[
\zeta \cdot \sigma(P) = \sigma(\zeta \cdot P) = \sigma( \chi(\zeta)P ) = \sigma(\chi(\zeta))\sigma(P) = \chi(\zeta)^c \sigma(P) = \chi'(\zeta) \sigma(P)\,.
\]
So, $\sigma$ defines an isomorphism between the subspaces of $V_{\ell}(J_{0})\otimes \overline{\Q_\ell}$ where $\mu_{ab}(\overline{k})$ acts by $\chi$ and by $\chi'$. Hence, each character of $\mu_{ab}(\overline{k})$ of a given order has the same multiplicity in $V_{\ell}(J_{0})\otimes \overline{\Q_\ell}$.

Next, we prove that each of the primitive characters of $\mu_{ab}(\overline{k})$ appears with multiplicity $1$ in $V_{\ell}(J_{0})\otimes \overline{\Q_\ell}$. 

To do this, we focus on the \emph{imprimitive} characters. For any subgroup $G \subset \mu_{ab}(\overline{k})$, there is some $\alpha$ dividing $a$ and $\beta$ dividing $b$ such that $\mu_{\alpha \beta}(\overline{k}) \cong \mu_{ab}(\overline{k})/G$. Then, the quotient $C_{0}/G$ is the curve over $\overline{k}$ with dense open subset defined by the affine equation $y^{\beta} + x^{\alpha} = 1$, which has genus $(\alpha - 1)(\beta - 1)/2$. Now, the span of the spaces where $\mu_{ab}(\overline{k})$ acts by characters which are trivial on $G$ is $(V_{\ell}(J_{0})\otimes \overline{\Q_\ell})^G$, which has dimension $(\alpha - 1)(\beta - 1)$ by the previous sentence together with Lemma~\ref{lem:subclaim}.

We can therefore determine the dimension of the subspace of $V_{\ell}(J_{0})\otimes \overline{\Q_\ell}$ where $\mu_{ab}(\overline{k})$ acts by primitive characters using an inclusion-exclusion argument on the subgroups of $G$. After a short computation, taking $\phi(n) = \#(\mathbb Z/n\mathbb Z)^\times$, one finds that this space has dimension $\phi(a) \phi(b) = \phi(ab)\,.$ Since there are $\phi(ab)$ primitive characters of $\mu_{ab}(\overline{k})$ and each appears with the same multiplicity, we see that each primitive character of $\mu_{ab}(\overline{k})$ appears in $V_{\ell}(J_{0})\otimes \overline{\Q_\ell}$ with multiplicity $1$.

Applying a similar argument to the curves $y^{\beta} + x^{\alpha} = 1$ with $\alpha > 1$ and $\beta > 1$ we see that in fact, each character of $\mu_{ab}(\overline{k})$ which is nontrivial when restricted to both $\mu_{a}(\overline{k})$ and $\mu_{b}(\overline{k})$ appears in $V_{\ell}(J_{0})\otimes \overline{\Q_\ell}$ with multiplicity $1$. Those characters which are trivial when restricted to either $\mu_{a}(\overline{k})$ or $\mu_{b}(\overline{k})$ do not appear in $V_{\ell}(J_{0})\otimes \overline{\Q_\ell}$, because they would arise from the Jacobians of curves $y + x^{\alpha} = 1$ or $y^{\beta} + x = 1$ of genus $0$, which are trivial. 

We conclude that 
\[
V_{\ell}(J_{0}) \cong \bigoplus_{\substack{\chi: \mu_{ab} \to \Q_{\ell}^{\times} :  \\ \chi|_{\mu_{a}} \text{ nontrivial and} \\
\chi|_{\mu_{b}} \text{ nontrivial.}}} \chi\,.
\]
\end{proof}

\begin{proof}[Proof of Theorem~\ref{thm:simplicity}]
We begin by proving the  ``only if'' direction of the statement:
Assume that at least one of $a$ and $b$ is composite.  By symmetry, assume that $a$ is composite, and let $d$ be one of its nontrivial divisors. Let $C_{d,b}$ be the curve with open affine defined by $y^b + x^d = t^q - t$. The map $(x,y)\mapsto (x^{a/d}, y)$ extends to a nonconstant $K$-morphism $C_{a,b}\to C_{d,b}$. 
The curve $C_{d,b}$ has positive genus since $d>1$. On the other hand, the genus of $C_{d,b}$ is strictly smaller than that of $C_{a,b}$ since $d<a$. The contravariant functoriality of the Jacobian then implies the existence of a morphism of abelian varieties $J_{d,b}\hookrightarrow J_{a,b}$, whose image is a positive-dimensional strict abelian subvariety of $J_{a,b}$ defined over $K$. Hence $J_{a,b}$ is not simple.

For the other direction, we prove the slightly stronger statement that if $a$ and $b$ are prime, then $J$ is simple over the compositum $K' = K \overline{k}$. Let $L' = K'(\hspace{-3pt}\sqrt[ab]{t^q - t})$. Our goal is to prove that $J_{K'}$ is simple by computing the action of $\Gal(L'/K')$ on $V_{\ell}(J_{K'})$.

We begin by observing that $C_{0}$ and $C$ become isomorphic after base change to $L'$. Namely, writing $u = \sqrt[ab]{t^q - t}$, there is an isomorphism $\phi: (C_{0})_{L'} \to C_{L'}$ defined on the affine patches by $(x,y) \mapsto (u^b x, u^a y)\,.$ The morphism $\phi$ commutes with the action of $\mu_{ab}(\overline{k})$ on both curves, but since $\phi$ is only defined over $L'$ and not over $K'$, the Galois group $\Gal(L'/K')$ acts differently on the two curves. If $\sigma \in \Gal(L'/K')$, then there is some primitive $(ab)$\textsuperscript{th} root of unity $\zeta \in \mu_{ab}(\overline{k})$ such that $\sigma(u) = \zeta(u)$. Then, for any $(x,y) \in C_0(\overline{k})$ we have that 
\[
\sigma(\phi(x,y)) = \sigma(u^b x, u^a y)  = (\zeta^b u^b x, \zeta^a u^a y) = \phi((\zeta^b x, \zeta^a y)) = \phi(\zeta \cdot (x,y)) = \zeta \cdot \phi((x,y))\,.
\]
The pushforward $\phi_{*}: (J_{0})_{L'} \to J_{L'}$ is also an isomorphism. For any $P \in J_{0}(\overline{k})$, we have $\sigma(\phi_{*}(P)) = \zeta \cdot \phi_{*}(P).$ Finally, since all of the $\ell$-power torsion of $J_{0}$ is defined over $\overline{k}$, we see that $\sigma \in \Gal(L'/K')$ acts on $V_{\ell}(J_{K'})\otimes \overline{\Q_{\ell}}$ in the same way as $\zeta \in \mu_{ab}(\overline{k})$ acts on $V_{\ell}(J_{0})\otimes \overline{\Q_{\ell}}$. Applying Lemma~\ref{lem:claim} in the case that $a$ and $b$ are both prime, we conclude that each of the primitive characters of $\Gal(L'/K') \cong \mu_{ab}(\overline{k})$ appears with multiplicity $1$, and no other characters appear. Since $\Gal(\overline{\Q_{\ell}}/\Q_{\ell})$ acts transitively on the $(ab)$\textsuperscript{th} roots of unity in $\overline{\Q_{\ell}}$, this can only happen if the representation $V_{\ell}(J_{K'})$ is simple as a $\Gal(L'/K')$-representation valued in $\Q_{\ell}$. We conclude that $J_{K'}$ is simple as an abelian variety over $K'$.
\end{proof}

\begin{remark}
Without much extra effort, one can refine Theorem~\ref{thm:simplicity} to show that $J$ has one $K'$-simple isogeny factor of dimension $(\alpha - 1)(\beta - 1)$ for each pair $(\alpha, \beta)$ such that the positive integer $\alpha$ divides $a$, the positive integer $\beta$ divides $b$ and both $a \neq 1$ and $b \neq 1$. Of course, if $a$ and $b$ are both prime, the only such pair is $(a,b)$.

On the other hand, Theorem~\ref{thm:simplicity} cannot be refined to show that $J$ is geometrically simple. Under certain congruence conditions on $a, b$, and $r$, the Jacobian $J_{0}$ has repeated isogeny factors. For instance, the Jacobian of the genus $2$ curve $y^2 + x^5 = 1$ over $\F_{19}$ is isogenous to the square of a supersingular elliptic curve. Since $J$ and $J_{0}$ become isomorphic after a suitable base change, the Jacobian of the genus $2$ curve $y^2 + x^5 = t^q - t$ over $\F_{19}$ is not geometrically simple.
\end{remark}

\section{Background on Gauss sums}\label{sec:gausssums}

In this section, we gather some facts about Gauss sums which will be useful in future sections.

\subsection[]{Multiplicative and additive characters on extensions of $\F_p$}
\label{sec:characters}

We fix an algebraic closure $\overline{\Q}$ of $\Q$ and denote by $\overline{\Z}$ the ring of algebraic integers.
We choose, once and for all, a prime ideal $\mathfrak{p}$ of $\overline{\Z}$ which lies over the rational prime $p$.
We write $\pval:\overline{\Q}\to\Q$ for the $\mathfrak{p}$-adic valuation on $\overline{\Q}$, normalised so that $\pval(r)=1$.
\\

The quotient $\overline{\Z}/\mathfrak{p}$ is an algebraic closure of $\F_p$, denoted by $\overline{\F_p}$. All finite extensions of $\F_p$ will be viewed as subfields of $\overline{\F_p}$. 
The quotient map $\overline{\Z}\to\overline{\Z}/\mathfrak{p}=\overline{\F_p}$ further induces an isomorphism between the group of roots of unity in $\overline{\Q}$ whose order is prime to $p$, and ${\overline{\F_p}}^\times$. Let $\cchi:\overline{\F_p}^\times\to\overline{\Q}^\times$  denote the inverse of this isomorphism. The isomorphism $\cchi$ is sometimes called the Teichm\"uller character of $\overline{\F_p}$.

\begin{definition}
Let $\F$ be a finite field extension of $\F_p$, and $n$ be a positive integer dividing $|\F^\times|$. We define a multiplicative character $\chi_{\F,n}$ on $\F$ by
\[\chi_{\F, n} : \F^\times \to \overline{\Q}^\times, \quad
x\mapsto \cchi(x)^{|\F^\times|/n}.\]
\end{definition}
A straightforward computation shows that $\chi_{\F,n}$ has exact order $n$.

We fix a nontrivial additive character $\psi_0$ on $\F_p$.
We may, and will, assume that $\psi_0$ takes values in $\Q(\zeta_p)$.
For any finite extension $\F/\F_p$, we denote the relative trace map by  $\Tr_{\F/\F_p} : \F \to\F_p$.
The composition $\psi_0\circ\Tr_{\F/\F_p}$ is then a nontrivial additive character on $\F$. More generally:

\begin{definition}
Let $\F$ be any finite field extension of $\F_p$, and let $\alpha\in\F$. We define an additive character $\psi_{\F,\alpha}$ on $\F$ by 
\[\psi_{\F, \alpha} : \F \to\Q(\zeta_p)^\times, \quad
x\mapsto (\psi_0\circ\Tr_{\F/\F_p})(\alpha x)\,.\]
The character $\psi_{\F,\alpha}$ is nontrivial for any $\alpha\neq0$.
\end{definition}

To lighten expressions, we suppress $\F$ from the notation when it is clear from context.

\subsection{Classical properties of Gauss Sums}\label{ssec:gauss.sums.basic}

We begin by recalling the definition of Gauss sums and some of their classical properties.  

\begin{definition}
Let $\F$ be a finite field of characteristic $p$.
Given an additive character $\psi$ and a multiplicative character $\chi$ on $\F$, we define the Gauss sum $\Gauss{\F}{\chi}{\psi}$ by
\[\Gauss{\F}{\chi}{\psi}=-\sum_{x \in \F^{\times}} \chi(x)\psi(x).\] 
\end{definition}

Let $\F$ be a finite field of characteristic $p$. For any  additive character $\psi$ and any multiplicative character $\chi$ on $\F$, we have the following facts.  
\begin{enumerate}
\item 
    If $\chi$ has order $n$, then $\Gauss{\F}{\chi}{\psi}$ is an algebraic integer in the cyclotomic field $\Q(\mu_{np})$.

\item If $\chi$ is nontrivial, orthogonality of characters implies that in any complex embedding, 
\begin{align} \label{eqn:complex-valuation-Gauss-sum}
    |\Gauss{\F}{\chi}{\psi}| = |\F|^{1/2}\,.
\end{align}

\item 
    For $\alpha \in \F^\times$, in the notation introduced in the previous subsection, 
    \begin{align}\label{eqn:additive-char-Gauss-sum}
   \Gauss{\F}{\chi}{\psi_{\F, \alpha}} 
    = \chi(\alpha)^{-1} \Gauss{\F}{\chi}{\psi_{\F, 1}}\,.
    \end{align}
 
\item (Hasse-Davenport relation) 
    For any finite extension $\F'/\F$, 
    \begin{align} \label{eqn:HasseDavenportRelation}
  \Gauss{\F'}{\chi\circ \norm_{\F'/\F}}{\psi\circ \Tr_{\F'/\F}}=\Gauss{\F}{\chi}{\psi}^{[\F':\F]}.  \end{align}

\end{enumerate}

\subsection{Orbits}\label{sec:orbits}
Let $p$ be a prime number and $r$ be a fixed power of $p$. For any integers $a,b$ which are relatively prime to each other and coprime to~$p$, and for any power $q$ of $p$, define
\[S:= S_{a,b,q} = (\Z/a\Z\smallsetminus\{0\})\times (\Z/b\Z\smallsetminus\{0\}) \times \F_q^\times.\]
The subgroup $\langle r\rangle$ of $\mathbb{Q}^\times$ generated by $r$ acts on $S$ via the rule
\[\forall (i,j,\alpha)\in S, \qquad 
r\cdot (i,j,\alpha) := (ri, rj, \alpha^{1/r}).\]
In other words, $r$ acts on $(\Z/a\Z\setminus\{0\})\times (\Z/b\Z\setminus\{0\})$ by component-wise multiplication and on $\F_q^\times$ by the inverse of the $r$-power Frobenius. 

We denote by $O:=O_{r,a,b,q}$ the set of orbits of $\langle r\rangle$ on $S$. Recall that for $n\geq 1$ coprime to $p$, we have defined $o_{p}(n)$ (resp. $o_{r}(n)$) to be the multiplicative order of $p$ (resp. $r$) modulo $n$. For $n\geq1$ coprime to $p$ and $i\in\Z/n\Z\smallsetminus\{0\}$,  we write $\kappa_{r, n}(i)$ for
the multiplicative order of $r$ modulo $n/\gcd(n,i)$. I.e., 
\[
\kappa_{r,n}(i) := o_r \big(n/\gcd(n,i)\big)\,.
\]

If $o\in O$ is the orbit of $(i,j, \alpha)\in S$, then a computation shows that
\begin{equation}
|o| = \lcm\big(\kappa_{r,a}(i), \kappa_{r,b}(j),  [\F_r(\alpha), \F_r]\big)\,.
\end{equation}

For any integer $n$ coprime to $p$, let
\[S'_n:=(\Z/n\Z\smallsetminus\{0\})\times \F_q^\times\,.\]
Endow $S_{n}'$ with an action of $\langle r\rangle$ \emph{via} the rule $r\cdot (i, \alpha) = (r i, \alpha^{1/r})$.
Write $O'_n$ for the  set of orbits of $S'_n$ under this action.

If $(i, \alpha)\in S'_n$, 
 then the length $|o'|$ of its orbit $o'\in O'_n$ is the smallest integer~$f\geq 1$ such that both~$\alpha \in \F_{r^f}$ and $n$ divides $i(r^f-1)$. 
In other words, 
\begin{align} \label{eqn:orbitsize.prime}
  |o'|=\lcm\big(\kappa_{r,a}(i), [\F_r(\alpha):\F_r]\big)\,.  
\end{align}

With $S_n'$ and $O$ as above, the natural projection maps $S_{a,b,q}\to S'_a$ and $S_{a,b,q}\to S'_b$ commute with the actions of $\langle r\rangle$ on these sets.
These projections therefore induce surjective maps $ \pi_a : O \to O'_a$ and $\pi_b : O \to O'_b$.
For any $o\in O$, we let 
\[\nu_a(o) := |o|/|\pi_a(o)| \qquad \text{ and }\qquad \nu_b(o) := |o|/|\pi_b(o)|. \]

If $o$ is the orbit of $(i,j,\alpha)$, we have
\[\nu_a(o) = \frac{\lcm\big(\kappa_{r,a}(i),\kappa_{r,b}(j), [\F_r(\alpha): \F_r]\big)}{\lcm\big( \kappa_{r,a}(i),  [\F_r(\alpha): \F_r]\big)}
= \frac{\lcm\big(|\pi_a(o)|, \kappa_{r,b}(j)\big)}{|\pi_a(o)|}
= \frac{\kappa_{r,b}(j)}{\gcd\big(|\pi_a(o)|, \kappa_{r,b}(j)\big)}.\]
In particular, $\nu_a(o)$ and $\nu_b(o)$ are integers, and $\nu_a(o)=1$ if and only if $\kappa_{r,b}(j)$ divides $|\pi_a(o)|$.

Since $a$ and $b$ are relatively prime, the Chinese remainder theorem gives a natural isomorphism $\phi: \Z/a\Z\times\Z/b\Z\simeq \Z/ab\Z$. The set $\phi ((\Z/a\Z\smallsetminus\{0\})\times (\Z/b\Z\smallsetminus\{0\}))$ is clearly stable under the action of $\langle r\rangle$ by component-wise multiplication on $\Z/ab\Z\smallsetminus\{0\}$, so the orbit set $O_{r,a,b,q}$ may be viewed as a subset of $O'_{ab}$.

\subsection{Gauss sums associated to orbits}\label{sec:gausssums.orbits}

Recall that we have fixed a nontrivial additive character $\psi_0$ on $\F_p$. Let $n$ be an integer which is coprime to $p$. Consider the set $S'_n$ as above, with its action of $\langle r \rangle$. Let $(i, \alpha)\in S'_n$, and write $o'\in O'_n$ for its orbit under
the action $\langle r\rangle$  on $S'_n$.
Let $\F'$ be the extension of $\F_r$ of degree $|o'|$.
By construction, $\alpha^{r^{|o'|}} = \alpha$. So, we have $\alpha\in\F'$. 
Hence, we may consider the  nontrivial additive character $\Psi_{(i,\alpha)}$ on $\F'$ defined by 
\[\forall x\in\F',\qquad
\Psi_{(i,\alpha)}(x):=\psi_{\F',\alpha}(x) = (\psi_0\circ\Tr_{\F'/\F_p})(\alpha x).\]
 By construction,  $n$ divides $i\,(r^{|o|}-1)= i \, |{\F'}^\times|$. We introduce a nontrivial multiplicative character $\llambda_{(i, \alpha)}$ on $\F'$ defined by 
\[\forall x\in\F', \qquad 
\llambda_{(i, \alpha)}(x) := \cchi(x)^{i  (r^{|o|}-1)/n}. \]
This leads us to consider the Gauss sum $\Gauss{\F'}{\llambda_{(i, \alpha)}}{\Psi_{(i, \alpha)}}$. 

\begin{lemma}\label{lem:InvariantGaussSums}
For all $(i,\alpha) \in S'_{n}$, we have
\[\Gauss{\F'}{\llambda_{(i, \alpha)}}{\Psi_{(i, \alpha)}} 
= \Gauss{\F'}{\llambda_{r\cdot (i, \alpha)}}{\Psi_{r\cdot(i, \alpha)}}. \]
In other words, the value of $\Gauss{\F}{\llambda_{(i, \alpha)}}{\Psi_{(i, \alpha)}}$ is constant along the $\langle r \rangle$-orbit $o'$ of $(i, \alpha)$.
\end{lemma}

\begin{proof} By definition,
\begin{align*}
- \Gauss{\F}{\llambda_{r\cdot (i, \alpha)}}{\Psi_{r\cdot(i, \alpha)}}
&= \sum_{x\in(\F')^\times} \cchi(x)^{r i (r^{|o|}-1)/n}  \, (\psi_{0}\circ\Tr_{\F'/\F_p})(\alpha^{1/r} x)
\end{align*}

Since $x \mapsto x^r$ defines a bijection $(\F')^\times \to (\F')^\times$, we may set $y = x^r$ and reindex. This yields
\begin{align*}
- \Gauss{\F}{\llambda_{r\cdot (i, \alpha)}}{\Psi_{r\cdot(i, \alpha)}}
&= \sum_{y\in(\F')^\times} \cchi(y)^{ i (r^{|o|}-1)/n} \, (\psi_{0}\circ\Tr_{\F'/\F_p})(\alpha^{1/r} y^{1/r})\\
&= \sum_{y\in(\F')^\times} \llambda_{(i, \alpha)}(y)\, (\psi_{0}\circ\Tr_{\F'/\F_p})((\alpha y)^{1/r})\,.
\end{align*}

Since $\F_{r} \subset \F'$, if $z \in \F'$, then $z$ is conjugate to $z^r$ over $\F_{r}$. So,  $\Tr_{\F'/\F_p}(z)=\Tr_{\F'/\F_p}(z^r)$. Hence,
\begin{align*}
- \Gauss{\F}{\llambda_{r\cdot (i, \alpha)}}{\Psi_{r\cdot(i, \alpha)}}
& = \sum_{y\in(\F')^\times} \llambda_{(i, \alpha)}(y) \,(\psi_{0}\circ\Tr_{\F'/\F_p})(\alpha y) \\
& = - \Gauss{\F'}{\llambda_{(i,\alpha)}}{\Psi_{(i, \alpha)}}.
\end{align*}
\end{proof}

Lemma~\ref{lem:InvariantGaussSums} allows us to define Gauss sums associated to $\langle r \rangle$-orbits:
\begin{definition}\label{def:bold_g}
In the above setting, for an orbit $o'\in O'_n$, 
we write $\F'$ for the extension of $\F_r$ of degree $|o'|$, 
and we set
\[\Gauor{o'} := \Gauss{\F'}{\llambda_{(i, \alpha)}}{\Psi_{(i, \alpha)}}\]
for one/any representative $(i, \alpha)\in S'_n$ of $o'$.
\end{definition}
Since $\llambda_{(i, \alpha)}$ is nontrivial, by \eqref{eqn:complex-valuation-Gauss-sum} we have 
\[|\Gauor{o'}| =|\F'|^{1/2} =r^{|o'|/2}\]
in any complex embedding of $\overline{\Q}$.

Now let $a$ and $b$ be  relatively prime integers which are coprime to $p$,
and consider the set $O$ of orbits of $\langle r\rangle$ acting on the set $S_{a,b,q}$ introduced in \S\ref{sec:orbits}. 
Recall that there are surjective maps 
$\pi_a:O\to O'_a$ and $\pi_b : O\to O'_b$. 
We may finally introduce:

\begin{definition}\label{def:oomega}
In the above setting, for any orbit $o\in O$, we let 
\[\oomega(o) := \Gauor{\pi_a(o)}^{\nu_a(o)} \Gauor{\pi_b(o)}^{\nu_b(o)},\]
where $\nu_a(o) = |o|/|\pi_a(o)|$ and  $\nu_b(o) = |o|/|\pi_b(o)|$.
\end{definition}

For any orbit $o\in O$, we have $|\oomega(o)|=  r^{|o|}$ in any complex embedding of $\overline{\Q}$.
\\

For any $a,b$, we let $\theta_{a,b} := \lcm(o_p(a), o_p(b))$. Recall that an algebraic integer $g$ is called a Weil integer of size $p^\theta$ (with $\theta\in \frac{1}{2}\Z$) if and only if $g$ has magnitude $p^\theta$ in any complex embedding of $\overline{\Q}$.
We record the following proposition for future use.

\begin{proposition}\label{prop:oomega.product} 
For any orbit $o\in O$, 
there exist an $(ab)$\textsuperscript{th} root of unity $\zeta_o$ and a Weil integer $g_o$ of size $p^{\theta_{a,b}}$ such that 
\[\oomega(o) =\zeta_o  g_o^{[\F_r:\F_p]\cdot |o|/\theta_{a,b}}.\]
\end{proposition}

\begin{proof}
Let $(i,j,\alpha)\in S$ have orbit $o\in O$. Then, $(i,\alpha)\in S'_a$ is a representative of $o':=\pi_a(o)\in O'_a$ and $(j, \alpha)\in S'_b$ is a representative of $\pi_b(o)\in O'_b$. 
Let $\F'$ be the extension of $\F_r$ of degree $|o'|$.
By the definition of $\Gauor{o'}$ and equation~\eqref{eqn:additive-char-Gauss-sum},  
we have
\[ \Gauor{o'} = \llambda_{(i, \alpha)}(\alpha)^{-1} \Gauss{\F'}{\llambda_{(i, \alpha)}}{\psi_{\F', 1}}.\]
Observe that $\zeta_{o'}:=\llambda_{(i, \alpha)}(\alpha)^{-1}$ is an $a$\textsuperscript{th} root of unity because $\llambda_{(i, \alpha)}$ has order dividing $a$.
Let $\F$ be the extension of $\F_p$ of degree $\kappa_{p,a}(i)=o_p(a/\gcd(i, a))$. 
We note that $[\F':\F] = [\F_r:\F_p]\cdot |o'|/\kappa_{p,a}(i)$.
Moreover, the character $\llambda_{(i, \alpha)}$ is none other than 
$\chi_{\F, |\F^\times|}^{i|\F^\times|/a}\circ\norm_{\F'/\F}$. 

Define $g_{o'}:=\Gauss{\F}{\chi_{\F, |\F|}^{i|\F^\times|/a}}{\psi_{\F, 1}}$. Then, $g_{o'}$ is a Weil integer of size $p^{\kappa_{p,a}(i)/2}$. Applying the Hasse--Davenport relation \eqref{eqn:HasseDavenportRelation} for Gauss sums, we deduce that
\[\Gauor{o'} = \llambda_{(i, \alpha)}(\alpha)^{-1}\left(\Gauss{\F}{\chi_{\F, |\F^\times|}^{i|\F^\times|/a}}{\psi_{\F, 1}}\right)^{[\F':\F]}
= \zeta_{o'}\, g_{o'}^{[\F_r:\F_p]|o'|/\kappa_{p,a}(i)}\,.\]
A similar argument shows that if we define $\zeta_{\pi_{b}(o)} \colonequals \llambda_{(j, \alpha)}(\alpha)^{-1}$ and $g_{\pi_{b}(o)}:=\Gauss{\F}{\chi_{\F, |\F|}^{j|\F^\times|/b}}{\psi_{\F, 1}}$, then 
$\Gauor{\pi_b(b)} = \zeta_{\pi_b(o)} g_{\pi_b(o)}^{[\F_r:\F_p]|\pi_b(o)|/\kappa_{p,b}(j)}$.

By the definition of $\oomega(o)$, we may write
\begin{align*}
\oomega(o) 
&= \zeta_{\pi_a(o)}^{\nu_a(o)}\zeta_{\pi_b(o)}^{\nu_b(o)}
g_{\pi_a(o)}^{[\F_r:\F_p]|o|/\kappa_{p,a}(i)}g_{\pi_b(o)}^{[\F_r:\F_p]|o|/\kappa_{p,b}(j)} \\
&= \left(\zeta_{\pi_a(o)}^{\nu_a(o)}\zeta_{\pi_b(o)}^{\nu_b(o)}\right)
\left(g_{\pi_a(o)}^{\theta_{a,b}/\kappa_{p,a}(i)}g_{\pi_b(o)}^{\theta_{a,b}/\kappa_{p,b}(j)}\right)^{[\F_r:\F_p]\cdot|o|/\theta_{a,b}}.
\end{align*}
Note that both $\kappa_{p,a}(i)$ and $\kappa_{p,b}(j)$ divide $\theta_{a,b}$. 
In this expression, 
$\zeta_o :=\zeta_{\pi_a(o)}^{\nu_a(o)}\zeta_{\pi_b(o)}^{\nu_b(o)}$ 
is a root of unity of order dividing $ab$, and the term
\[g_o 
:= g_{\pi_a(o)}^{\theta_{a,b}/\kappa_{p,a}(i)}g_{\pi_b(o)}^{\theta_{a,b}/\kappa_{p,b}(j)}\] 
is a Weil integer of size $p^{\theta_{a,b}}$.  
Therefore, $\oomega(o)$ may be written in the desired form.
\end{proof}

\section{Explicit expression for the L-function and the BSD conjecture}\label{sec:Lfunction}

In this section, we provide an explicit formula for the $L$-function of the Jacobian $J$ of the curve~$C$. 
Our proof is based on a computation with character sums. In Section~\ref{sec:bsd_conj_for_j}, we remark that $J$ satisfies the BSD conjecture. BSD will be used in Section~\ref{sec:rank} to make further observations about the rank of $J$. We give an alternate cohomological proof of our explicit formula for $L(J,T)$ in Section~\ref{sec:Cohomological.Lfunction}.

\subsection[]{Definition of the $L$-function}\label{ssec:Lfunction.definition}
\newcommand{\frob}{\mathrm{Fr}}
Fix a prime number $\ell\neq p$, and let $H^1(J) \colonequals H^1_{\mathrm{et.}}(J, \overline{\Q_{\ell}})$ denote the first $\ell$-adic \'etale cohomology group of $J/K$. 
It is well-known that $H^1(J)$ is a $\overline{\Q_\ell}$-vector space of dimension $2g$ which is equipped with a natural action of the absolute Galois group of $K$.
For any place $v$ of $K$, we let $\frob_v$ denote the geometric Frobenius at $v$, let $I_v$ denote the inertia group at $v$, and let $V_{\ell}(J)$ denote the $\ell$-adic Tate module of $J$. As a Galois module, $H^1(J)$ is isomorphic to the dual of $V_\ell(J) \otimes \overline{\mathbb Q_{\ell}}$.  (This duality follows by using the short exact sequence $0 \to \mu_{\ell^n} \to \mathbb{G}_m \to \mathbb{G}_m \to 0$ of sheaves on $J$ and taking an inverse limit over $n$.)

The Hasse--Weil $L$-function of $J$ may be defined by the Euler product:
\begin{equation}\label{eq:def.Lfunction}
L(J,T) 
= \prod_{v}\det\big(1 - T^{\deg v}\,\frob_v   \mid H^1(J)^{I_v}\big)^{-1},    
\end{equation}
where the product runs over all places $v$ of $K$. 
Here,   $H^1(J)^{I_v}$ designates the $I_v$-invariant subspace of $H^1(J)$.  

The abelian variety $J$ has good reduction at a place $v$ if and only if $I_v$ acts trivially on $H^1(J)$, or equivalently, if and only if $H^1(J)^{I_v}$ has dimension $2g$ \cite{SerreTate}.\\

The power series in $T$ resulting from the formal expansion of the product \eqref{eq:def.Lfunction}
is known, by the Hasse--Weil bound on the eigenvalues of $\frob$ acting on $H^1(J)$, to converge on the complex open disc $\{T\in\C : |T|<r^{-3/2}\}$.
But actually, much more is true! We summarize deep results of Grothendieck, Deligne, and others in the following theorem.

\begin{theorem}\label{thm:weil2}
Let $J/K$ be as above. Write $g=\dim J$ for its dimension, and $N_J\in\mathrm{Div}(\P^1)$ for its conductor divisor.
\begin{enumerate}[(1.\hspace{-3pt}]
    \item Rationality)
    The $L$-function $L(J, T)$ is a rational function in $T$ with integral coefficients.  The global degree of $L(J,T)$, defined to the degree of the numerator minus the degree of the denominator, is denoted by $b(J)$. The degree $b(J)$ is related to $\deg N_J$ by
    $b(J) = \deg N_J -4g$.
    
    \item Functional equation) There is some $w(J) \in \{\pm 1\}$ such that  
    $L(J, T)$ satisfies
    \[L(J, T) = w(J) \, (rT)^{b(J)} L\left(J,(r^2T)^{-1}\right)\,.\]
    \item Riemann Hypothesis)
    If $z\in\C$ is such that $L(J, z)=0$, then $|z|=r^{-1}$.
\end{enumerate}
\end{theorem}

\begin{proof}
For the proofs of rationality, the functional equation, and the Riemann hypothesis, we refer the reader to \cite{deligneweil2}. We provide a proof of the formula for the degree $b(J)$ of $L(J,T)$ in Proposition~\ref{prop:degOfL}. 
\end{proof}
Once we compute the $L$-function of our Jacobian in Theorem~\ref{theorem:Lfunction}, we check the degree in Remark~\ref{remark:deg_check} using the formula $b(J) = \deg N_J -4g$. We will also use this formula in the cohomological computation of $L(J,T)$ in Section~\ref{sec:Cohomological.Lfunction}.

\subsection[]{Explicit expression for the $L$-function}\label{sec:ExplicitL}

We let $p, r,a,b,q$ have the same meaning as in the introduction. With the notation introduced in Section~\ref{sec:gausssums}, we state our formula for the $L$-function of $J$.

\begin{theorem}
\label{theorem:Lfunction} 
Let $O$ be the orbit set defined in \S\ref{sec:orbits} and, for any $o\in O$, define $\oomega(o)$ as in Definition~\ref{def:oomega}. The $L$-function $L(J, T)\in\Z[T]$ of $J/K$ admits the following expression:
\begin{equation}\label{equation:Lfunction}
    L(J, T) = \prod_{o\in O} 
    \left(1-\oomega(o)\, T^{|o|}\right)\,.
\end{equation}
\end{theorem}

The proof of Theorem~\ref{theorem:Lfunction} occupies the rest of Section~\ref{sec:ExplicitL}.
We start by proving a number of elementary lemmas in Section~\ref{sec:Lfunction-preliminaries}, before gathering our results to conclude the proof in Section~\ref{sec:proof-of-Lfunction}.

\subsection{Preliminary lemmas} \label{sec:Lfunction-preliminaries}

We first recall an expression for the logarithm of $L(J, T)$.
For any $\beta\in\overline{\F_r}^\times$, let $X_\beta$ denote the smooth projective curve over $\F_{r}(\beta)$ which is birational to the curve defined by the affine model $x^a + y^b =  \beta^q-\beta$.

\begin{lemma}\label{lemm:Lfunc.thefirst}
For $m \in \Z_{\geq 1}$ and $\beta \in \F_{r^{m}}$, set $A_{J}(\beta,m) = r^m + 1 - |X_{\beta}(\F_{r^{m}})|$. Then,
\[ \log L(J,T)  
= \sum_{m\geq 1}
\left(\sum_{\beta\in\F_{r^m}^\times}  A_J(\beta, m)\right)
\frac{T^m}{m}.\]

\end{lemma}

 \begin{proof}

We have shown in Proposition \ref{prop:unipotentreduction} that $J$ has unipotent reduction at all of its places of bad reduction. 
At a place $v$ of unipotent reduction for $J$, $\dim_{\Q_\ell} H^1(J)^{I_v} = 0$. (See \cite{SerreTate}.) Hence, the associated Euler factor 
$\det(1 - T^{\deg v}\,\text{Fr}_v \mid H^1(J)^{I_v})$ in $L(J, T)$ is equal to $1$.

Hence, in the Euler product \eqref{eq:def.Lfunction} defining $L(J, T)$, we may ignore the factors corresponding to places of bad reduction.
We thus have
\[
L(J, T ) =\prod_{\text{good }v} \det(1 - T^{\deg v}\, \text{Fr}_v \mid H^1(J)^{I_v})^{-1}\,. 
\]

At a place $v$ of good reduction, the inertia group $I_v$ acts trivially on $H^1(J)$ (see \cite{SerreTate} again), so that  $H^1(J)^{I_v}$ has dimension $2g$. 
We write $\alpha_{v, 1}, \dots, \alpha_{v, 2g}\in\overline{\Q_\ell}$ for the eigenvalues of $\frob_v$ acting on $H^1(J)$.
Formally expanding the power series $\log L(J, T)\in\overline{\Q_\ell}[[T]]$, we obtain that
\begin{align*}
 \log L(J, T) 
&=  - \sum_{\text{good }v} \sum_{i=1}^{2g} \log(1-\alpha_{v,i} \, T^{\deg v}) \\
&=   \sum_{\text{good }v} \sum_{i=1}^{2g}\sum_{k=1}^{\infty} \frac{ (\alpha_{v,i} \, T^{\deg v})^k}{k}\\
&=   \sum_{k=1}^{\infty} \left(\sum_{\text{good }v}  \left( \sum_{i=1}^{2g}\alpha_{v,i}^k\right)\frac{T^{k \deg v}}{k}\right)\,.
\end{align*}
We write $m = k \deg v$ and reindex the sums. Since 
\begin{align}
    \Tr(\frob_{v_\beta}^{m/\deg v_\beta}|H^1(J)) = \sum_{i=1}^{2g} \alpha_{v,i}^{m/\deg v} \,,
\end{align}
this yields
\begin{align}
\label{eq:Lfunction.lemma1}
 \log L(J, T) 
= \sum_{m=1}^\infty\left( \sum_{ \substack{ \text{good } v \\ \deg v \mid m}}
\Tr(\frob_{v}^{m/\deg v}|H^1(J))\, \deg v \, \frac{T^m}{m}\right)\,.
\end{align}

Since $K$ is the function field of $\P^1$, a place $v$ of $K$ may be viewed as the  $\Gal(\overline{\F_r}/\F_r)$-orbit of an $\overline{\F_r}$-rational point  on $\P^1$. 
The degree of $v$ is the number of elements in the associated orbit.

Let $\beta\in\P^1(\overline{\F_r})$ and $v_\beta$ be the corresponding place of $K$. 
The orbit of $\beta$ under the action of $\Gal(\overline{\F_r}/\F_r)$ has exactly $[\F_r(\beta):\F_r]$ elements. So, 
$\deg(v_\beta)= [\F_r(\beta):\F_r]$.
By construction, the numbers $\Tr(\frob_{v_\beta}^{m/\deg v_\beta}|H^1(J))$ do not depend on the choice of a representative $\beta\in \P^1(\overline{\F_r})$ of the orbit $v$.

Let $U$ be the largest subscheme of $\P^1$ such that $J_{v}$ has good reduction at all places $v \in U$. 
By Proposition~\ref{prop:unipotentreduction}, we have $U = \A^1 \smallsetminus \{z : z^q -z = 0\}\,.$ We may thus rewrite identity~\eqref{eq:Lfunction.lemma1} as
\begin{equation}\label{eq:Lfunction.lemma4}
\log L(J, T) = 
\sum_{m=1}^\infty \left( \sum_{\beta\in U(\F_{r^m})} \Tr(\frob_{v_\beta}^{m/\deg v_\beta}|H^1(J))\right) \frac{T^m}{m}.    
\end{equation}

By flat base change, we have $H^1(J) \cong H^1_{\text{\'et}}(J_v,\mathbb{Q}_\ell)$. From \cite[5.3.5]{Poonen_Curves}, we have $H^1_{\text{et.}}(J_{v}, \mathbb Q_{\ell}) \cong H^1_{\text{\'et}}(X_{v}, \mathbb Q_{\ell})$. Together, we see
\[
H^1(J)^{I_{v}} = H^1(J) \cong H^1_{\text{et.}}(J_{v}, \mathbb Q_{\ell}) \cong H^1_{\text{et.}}(X_{v}, \mathbb Q_{\ell})\,.
\]
So, the Grothendieck--Lefschetz trace formula gives 
$\Tr(\frob_{v_\beta}^{m/\deg v_\beta}|H^1(J)) = |\F_{r^m}|+1 - |X_\beta(\F_{r^m})| = A_J(\beta,m)$.
\end{proof}

We now interpret  the quantities $A_J(\beta, m)$ appearing in Lemma \ref{lemm:Lfunc.thefirst} as character sums. Write $\trivcar$ for the trivial multiplicative character.
For any $m\geq 1$ and $c \geq 2$ we set
\begin{eqnarray*}
M_{c}(r^m) & \colonequals \{\phantom{\text{nontrivial }} \text{characters } \lambda: \F_{r^m}^{\times} \to \mathbb C^{\times}\text{ such that } \lambda^c = \trivcar\}\,, \\
M'_{c}(r^m) & \colonequals \{\text{nontrivial } \text{characters } \lambda: \F_{r^m}^{\times} \to \mathbb C^{\times} \text{ such that } \lambda^c = \trivcar\}\,.
\end{eqnarray*}
We further define $M'_{a,b}(r^m) = M'_{a}(r^m) \times M'_{b}(r^m)$ and extend all multiplicative characters $\lambda$ by $\lambda(0) = 0$.
For any pair $(\lambda_2, \lambda_2)$ of multiplicative characters  on $\F_{r^m}$, any additive character $\psi$ on $\F_{r^m}$ and any $\alpha\in \F_{r^m}$, we set 
\[S_{r^m}(\lambda_1, \lambda_2, \psi, \alpha)
\colonequals\sum_{(w,z)\in(\F_{r^m})^2} \lambda_1(z) \lambda_2(w-z) \psi(\alpha w). \] 
With this new notation at hand, we may now state:
\begin{lemma}
\label{lemm:Lfunc.thesecond}
For any nontrivial additive character $\psi_r$ on $\F_r$, and any $m\geq 1$, we have
\[ \sum_{\beta\in\F_{r^m}^\times} A_J(\beta, m)
= - \sum_{\substack{\alpha\in\F_{r^m}\cap\mathbb{F}_q, \\(\lambda_1, \lambda_2)\in M'_{a,b}(r^m)}} S_{r^m}(\lambda_1, \lambda_2,\psi_r\circ\Tr_{\F_{r^m}/\F_r},\alpha).\]
\end{lemma}

\begin{remark}
It may seem odd that the right-hand side appears to depend on the choice of a nontrivial additive character $\psi_r$ while the left-hand side does not.
However, as should be clear after the proof, a different choice of $\psi_r$ merely permutes the terms $S_{r^m}(\lambda_1, \lambda_2, \psi, \alpha)$.
\end{remark}

\begin{proof} For a given $\beta\in\F_{r^m}^\times$, we
 begin by  giving an expression of $|X_\beta(\F_{r^m})|$ as a character sum. 
Recall that there is a unique point at infinity in $X_{\beta}(\overline{\F_{r^m}})$. This point is rational over $\F_{r^m}$. We have
$|X_\beta(\F_{r^m})| = 1 + \big|\big\{(x,y)\in(\F_{r^m})^2 : x^a+ y^b = \beta^q-\beta\big\}\big|$,
so that
\begin{equation}\label{eq:Lfunction.lemma2}
|X_\beta(\F_{r^m})|  - 1 =
    \sum_{x\in\F_{r^m}}\left|\big\{y\in\F_{r^n} : x^a+ y^b = \beta^q-\beta\big\}\right|.
\end{equation}
It is classical (see~\cite[Lemma 2.5.21]{Coh07}) that for any integer $N \geq 2$ and any $z\in \F_{r^m}$, we have
\begin{equation}\label{eq:character.identity.power}
\big|\big\{y\in\F_{r^m} : y^N = z\big\}\big| 
= \sum_{\lambda \in M_{N}(r^m)}\lambda(z),	
\end{equation}
The term corresponding to $\lambda=\trivcar$ contributes $1$. 
Evaluating \eqref{eq:character.identity.power} with $N=b$ and $z = -x^a + \beta^q - \beta$ into \eqref{eq:Lfunction.lemma2} and swapping the sums yields
\[ |X_\beta(\F_{r^m})|  - 1 
=\sum_{\lambda\in M_{b}(r^m)}  \sum_{x\in\F_{r^m}}\lambda(-x^a +\beta^q-\beta) 
= r^m + \sum_{\lambda\in M'_{b}(r^m)}  \sum_{x\in\F_{r^m}}\lambda(-x^a +\beta^q-\beta).\]
It follows that, for all $\beta\in\F_{r^m}$, we have
\[A_J(\beta, m) = - \sum_{\lambda\in M'_{b}(r^m)}  \sum_{x\in\F_{r^m}}\lambda(-x^a +\beta^q-\beta)\,.\]
For each $\lambda\in M'_{b}(r^m)$, we use \eqref{eq:character.identity.power} once more, this time with $N=a$, to reindex the sum over $x$ in the above display. This yields
\begin{align*}
\sum_{x\in\F_{r^m}}\lambda(-x^a +\beta^q-\beta) 
&= \sum_{z\in\F_{r^m}} \big|\big\{x\in\F_{r^m} : x^a = z\big\}\big|\,  \lambda(-z +\beta^q-\beta) \\
&= \sum_{\theta\in M_{a}(r^m)}  \sum_{z\in\F_{r^m}} \theta(z)\lambda(-z +\beta^q-\beta) 
=   \sum_{\theta\in M'_{a}(r^m)} \sum_{z\in\F_{r^m}} \theta(z) \lambda(-z +\beta^q-\beta).
\end{align*}
To justify the last equality, we note that the term corresponding to $\theta=\trivcar$ does not contribute by orthogonality of characters for $\F_{r^m}$.
We have thus proved that 
\[A_J(\beta, m) 
= \sum_{\lambda\in M'_{b}(r^m)} 
\sum_{\theta\in M'_{a}(r^m)}
  \sum_{z\in\F_{r^m}}\theta(z)\lambda(-z +\beta^q-\beta)\,.\]
Applying orthogonality of characters for $\F_{r^m}^\times$ once again, we also note that if $\theta$ and $\lambda$ are multiplicative characters such that $\theta\neq \lambda^{-1}$, the sum   $\sum_{z\in\F_{r^m}}\theta(z)\lambda(-z +\beta^q-\beta)$ vanishes if  $\beta^q -\beta=0$, including if $\beta = 0$. 
It follows that from the previous paragraph that, for all $m\geq 1$, we have 
\begin{align}
\sum_{\beta\in \F_{r^m}^\times} A_J(\beta, m)
&
= - \sum_{\beta\in  \F_{r^m}^\times} 
\sum_{\theta\in M'_{a}(r^m)}
\sum_{\lambda\in M'_{b}(r^m)}  
 \sum_{z\in\F_{r^m}}\theta(z)\lambda(-z +\beta^q-\beta) \notag\\
&
= - \sum_{\theta\in M'_{a}(r^m)}
\sum_{\lambda\in M'_{b}(r^m)} 
\left( \sum_{\beta\in  \F_{r^m}}\sum_{z\in\F_{r^m}}\theta(z)\lambda(-z +\beta^q-\beta)\right)\,. \label{eq:Lfunction.lemma3}
\end{align}
For fixed $(\theta,\lambda) \in M_{a,b}'(r^m)$,  
we now reindex the inner sum: 
\[ \sum_{\beta\in  \F_{r^m}} \sum_{z\in\F_{r^m}}\theta(z)\lambda(-z +\beta^q-\beta)
=  \sum_{w\in  \F_{r^m}} \big|\big\{\beta\in\F_{r^m} : w = \beta^q -\beta\big\}\big| \left(\sum_{z\in\F_{r^m}}\theta(z)\lambda(-z+w) \right).\]
We  now appeal to \cite[Lemma~4.5]{GriffonAS18}, which states that for any $z\in\F_{r^m}$ and any nontrivial additive character $\psi$ on $\F_{r^m}$ we have 
\begin{align}\label{eqn:Griffoneqn}
\big|\big\{\beta\in\F_{r^m} : w = \beta^q -\beta\big\}\big| 
=\sum_{\alpha\in (\F_{r^m}\cap\F_q)} \psi(\alpha w)\,.
\end{align}
Plugging \eqref{eqn:Griffoneqn} into \eqref{eq:Lfunction.lemma3} and reordering the sums, for any nontrivial additive character $\psi$ on $\F_{r^m}$ we obtain  
\[\sum_{\beta\in U(\F_{r^m})}A_J(\beta, m) 
= - \sum_{\theta\in M'_{a}(r^m)}
\sum_{\lambda\in M'_{b}(r^m)} 
		\sum_{\alpha\in (\F_{r^m}\cap\F_q)}
\left( \sum_{(w,z)\in(\F_{r^m})^2} \theta(z) \lambda(w-z)  \psi(\alpha w)\right)\,.\]
Note that the sum between brackets is equal to $S_{r^m}(\theta, \lambda, \psi, \alpha)$.

To conclude, recall that we have fixed a nontrivial additive character $\psi_r$ on $\F_r$. 
For any integer $m\geq1$, we  write the last display for $\psi = \psi_r\circ \Tr_{\F_{r^m}/\F_r}$, which is indeed a nontrivial additive character on $\F_{r^m}$. This yields that, for any $m\geq 1$, 
\[\sum_{\beta\in U(\F_{r^m})}A_J(\beta, m) 
= - \sum_{(\lambda_1, \lambda_2)\in M'_{a,b}(r^m)} \sum_{\alpha\in (\F_{r^m}\cap\F_q)}
S_{r^m}(\lambda_1, \lambda_2, \psi_r\circ \Tr_{\F_{r^m}/\F_r}, \alpha).\]
This identity together with \eqref{eq:Lfunction.lemma4} yields the lemma.
\end{proof}

Our next step towards proving Theorem~\ref{theorem:Lfunction} is to give a more recognizable form to the inner sums 
which appear in  Lemma~\ref{lemm:Lfunc.thesecond}.

\begin{lemma}\label{lemm:Lfunc.thethird}
Let $m\geq 1$. Given  a pair $(\lambda_1, \lambda_2)$ of nontrivial multiplicative characters on $\F=\F_{r^m}$, a nontrivial additive character $\psi$ on $\F_{r^m}$ and an element $\alpha\in\F_{r^m}$, we have
\[ S_{r^m}(\lambda_1, \lambda_2, \psi, \alpha) = \Gauss{\F}{\lambda_1}{\psi_{\alpha}} \Gauss{\F}{\lambda_2}{\psi_{\alpha}}\,, \]
where $\psi_\alpha$ is the additive character on $\F_{r^m}$ defined by $x\mapsto \psi(\alpha x)$.
\end{lemma}

\begin{proof} 
By definition of $S_{r^m}(\lambda_1, \lambda_2, \psi, \alpha)$, we have
\[ S_{r^m}(\lambda_1, \lambda_2, \psi, \alpha)  
= \sum_{z\in\F}\sum_{w\in\F} \lambda_1(z) \lambda_2(w-z) \psi(\alpha w)\,.\]
Re-indexing the inner sum by setting $y=w-z$, we obtain 
\begin{align*}
 S_{r^m}(\lambda_1, \lambda_2, \psi, \alpha)
 &=\sum_{y\in\F}\sum_{z\in\F}\lambda_1(z) \lambda_2(y) \psi(\alpha y+\alpha z)\\
 &=\left(\sum_{y\in\F}\lambda_1(z)\psi(\alpha y)\right)
 		\left(\sum_{z\in\F}\lambda_2(z)\psi(\alpha z)\right)\\
 &= \Gauss{\F_{r^m}}{\lambda_1}{\psi_{\alpha}} \Gauss{\F_{r^m}}{\lambda_2}{\psi_{\alpha}}\,.
\end{align*}
This concludes the proof. 
Note that both sides vanish if $\alpha=0$.
\end{proof}

 Recall that $o_{r}(n)$ denotes the multiplicative order of $r$ modulo $n$ and that $\cchi: \overline{\F_p}^\times \to \overline{\Q}^{\times}$ is the Teichm\"uller character defined in Section~\ref{sec:characters}.
\begin{lemma}\label{lemm:Lfunc.characters}
Fix an integer $c\geq 1$ which is coprime to $p$. For $i\in\Z/c\Z\smallsetminus\{0\}$, let $\kappa = o_r\big(c/\gcd(c,i)\big)$. Then, the map
\begin{align*}
 \big\{i\in\Z/c\Z\smallsetminus\{0\} : \kappa \mid m\big\} &  \to M'_{c}(r^m) \\
 i & \mapsto \left[ x \mapsto \big(\cchi\circ \norm_{\F_{r^m}/\F_{r^\kappa}}\big)(x)^{i (r^{\kappa}-1)/c}\right]
\end{align*}
is a bijection.
\end{lemma}
 
\begin{proof}
Choose any $i\in\Z/c\Z\smallsetminus\{0\}$ such that $\kappa$ divides $m$. 
The multiplicative character $\lambda: \F_{r^m}^{\times} \to \C^\times$ defined by $\lambda(x) =  (\cchi\circ\norm_{\F_{r^m}/\F_{r^\kappa}})(x)^{i (r^{\kappa}-1)/c}$ for all $x\in\F_{r^m}^\times$ has exact order $c/\gcd(i,c)$. In particular, $\lambda$ is nontrivial and has order dividing $c$, so $\lambda\in M'_{c}(r^m)$.

Conversely, let $\lambda$ be a nontrivial multiplicative character on $\F_{r^m}$ whose $c$\textsuperscript{th} power is trivial.
The Teichm\"uller character $\cchi$ generates the group of multiplicative characters on $\F_{r^m}$, so $\lambda = \cchi^\ell$ for some integer $\ell\in\{1, \dots, r^m -2\}$. 
Since $\lambda^c$ is trivial on $\F_{r^m}^\times$ and since $\cchi$ has order exactly $r^m-1$, there exists an integer $i\geq 1$ such that $\ell c = i(r^m -1)$.
Since $1 \leq \ell \leq r^m-2$, we have $1\leq i \leq c-1$.
Letting $c'=c/\gcd(c,i)$ and $i' = i/\gcd(c,i)$, we find that  $\ell c' = i'(r^m-1)$. By construction, $\gcd(c',i')=1$ and so $c'$ divides $r^m-1$.
In particular the order $\kappa$ of $r$ modulo $c'$ divides $m$ and so $i'(r^\kappa-1)/c'$ is an integer.  
We have $\ell = i(r^m -1)/c$. 
So, for all $x\in\F_{r^m}^\times$, 
\begin{align*}
\lambda(x) 
& = \cchi(x)^{i(r^m-1)/c} 
=\cchi(x)^{\frac{i'(r^\kappa-1)}{c'}(1+r^\kappa +\dots+ r^{m-\kappa})} 
=\cchi\left(x^{1+r^{\kappa} + \dots + r^{m-\kappa}}\right)^{\frac{i'(r^\kappa-1)}{c'}}\\
&= \big(\cchi\circ \norm_{\F_{r^m}/\F_{r^\kappa}}\big)(x)^{i(r^\kappa-1)/c}\,.
\end{align*}
Hence $\lambda$ has the desired form.
\end{proof}

We now connect our last results with the discussion in \S\ref{sec:orbits}--\S\ref{sec:gausssums.orbits}.
We previously introduced the set $O$ of orbits of the action of $r$ on $(\Z/a\Z \setminus \{0\}) \times (\Z/b\Z \setminus \{0\}) \times \F_q^\times$ and denoted the size of an orbit $o$ by $|o|$.  We also defined the natural projection maps 
\begin{align*}
    \pi_a: & (\Z/a\Z \setminus \{0\}) \times (\Z/b\Z \setminus \{0\}) \times \F_q^\times \to (\Z/a\Z \setminus \{0\}) \times \F_q^\times\, \qquad \text{and} \\
    \pi_b: & (\Z/a\Z \setminus \{0\}) \times (\Z/b\Z \setminus \{0\}) \times \F_q^\times \to (\Z/b\Z \setminus \{0\}) \times \F_q^\times\,
\end{align*} and fixed an additive character $\psi_0$ on $\F_p$.

For any $m\geq 1$ and $\alpha \in \F_{r^m} \cap \F_{q}$,
define an additive character $\psi_{m, \alpha} : \F_{r^m}\to\overline{\Q}$ by $\psi_{m,\alpha}(x) = (\psi_0\circ \Tr_{\F_{r^m}/\F_p})(\alpha x)$.

\begin{lemma}\label{lemm:Lfunc.thefourth}
For any $m\geq 1$, we have
\[\sum_{\substack{o\in O \text{ s.t.}\\ |o| \text{ divides } m}} |o|\, \oomega(o)^{m/|o|}
 = \sum_{ \substack{ \alpha\in (\F_{r^m}\cap\F_q)^\times, \\ (\lambda_1, \lambda_2)\in M'_{a,b}(r^m)}} \Gauss{r^m}{\lambda_1}{\psi_{m, \alpha}} \Gauss{r^m}{\lambda_2}{\psi_{m, \alpha}}\,.\]
\end{lemma}
\begin{proof}

For any integer $m\geq 1$ and any orbit $o\in O$ we note that $|\pi_a(o)|$ and $|\pi_b(o)|$ both divide $|o|$. If $|o|$ divides $m$, then $|\pi_a(o)|$ and $|\pi_b(o)|$ must also divide $m$.
Since $\nu_a(o) = |o|/|\pi_a(o)|$, we have 
\begin{equation}\label{eq:omega_relation}
    \oomega(o)^{m/|o|}
	= \Gauor{\pi_a(o)}^{m/|\pi_a(o)|} \Gauor{\pi_b(o)}^{m/|\pi_b(o)|}\,.
\end{equation}

Pick a representative $(i,j, \alpha)\in S$ of $o\in O$. 
Then, $(i, \alpha)\in S'_a$  is a representative of $\pi_a(o)$ and $(j, \alpha)\in S'_b$ is a representative of $\pi_b(o)$. 
We write $r_a = r^{|\pi_a(o)|}$. Using 
the Hasse--Davenport relation for Gauss sums and noting that $\Psi_{(i, \alpha)}\circ\Tr_{\F_{r^m}/\F_{r_a}} = \psi_{m, \alpha}$ yields
\begin{align*}
\Gauor{\pi_a(o)}^{m/|\pi_a(o)|} 
&=  \Gauss{r_a}{\llambda_{(i, \alpha)}}{\Psi_{(i, \alpha)}}^{m/\pi_a(o)} 
= \Gauss{r^m}{\llambda_{(i, \alpha)}\circ\norm_{\F_{r^m}/\F_{r_a}}}{\Psi_{(i, \alpha)}\circ\Tr_{\F_{r^m}/\F_{r_a}}} \\
&= \Gauss{r^m}{\llambda_{(i, \alpha)}\circ\norm_{\F_{r^m}/\F_{r_a}}}{\psi_{m, \alpha}}\,.
\end{align*}
A similar computation shows
\begin{align*}
\Gauor{\pi_b(o)}^{m/|\pi_b(o)|} 
&= \Gauss{r^m}{\llambda_{(j, \alpha)}\circ\norm_{\F_{r^m}/\F_{r_b}}}{\psi_{m, \alpha}}\,.
\end{align*}

If $o$ is the orbit of $(i,j,\alpha)\in S$, then $|o|$ divides $m$ if and only if 
(i) $\alpha\in \F_{r^m}$, 
(ii)  the order of $r$ modulo $a/\gcd(a,i)$ divides $m$ (which happens if and only if $a$ divides $i(r^m-1)$) and 
(iii)  the order of $r$ modulo $b/\gcd(b,j)$ divides $m$ (which happens if and only if $b$ divides $j(r^m-1)$).

Recall that we have set $\kappa_{r,a}(i) = o_r(a/\gcd(a,i))$. We have
\begin{align} \label{eqn:MessyProductOfGaussSums}
\sum_{\substack{o\in O \\ |o| \text{ divides } m}} |o|\,\oomega(o)^{m/|o|}
= \sum_{\substack{ (i,j,\alpha)\in S \\ \alpha\in\F_{r^m}^\times \\ \kappa_{r,a}(i) \mid m \\ \kappa_{r,b}(j) \mid m }} 
\Gauss{r^m}{\llambda_{(i, \alpha)}\circ\norm_{\F_{r^m}/\F_{r_a}}}{\psi_{m, \alpha}}
\Gauss{r^m}{\llambda_{(j, \alpha)}\circ\norm_{\F_{r^m}/\F_{r_b}}}{\psi_{m, \alpha}}\,.
\end{align}

Set $\kappa = \kappa_{r,a}(i).$ Then, $\kappa$ divides $o_{r}(a)$ which divides $\pi_{a}(o)$. Also, note that for any finite field $\F$ of
characteristic $p$ and any extension $\F'$
of $\F$, we have $\cchi|_{\F}\circ N_{\F'/\F} = (\cchi|_{\F'})^{|\F'^\times|/|\F^\times|}$. Together, these imply that 
\[
  \llambda_{(i, \alpha)}\circ\norm_{\F_{r^m}/\F_{r_a}}
 = (\cchi \circ \norm_{\F_{r^m}/\F_{r_a}})^{\frac{i(r_a^\kappa -1)}{a}}
 = (\cchi \circ \norm_{\F_{r^m}/\F_{r_a}}\circ \norm_{\F_{r^m}/\F_{r^\kappa}})^{\frac{i(r^\kappa -1)}{a}} 
 = (\cchi \circ \norm_{\F_{r^m}/\F_{r^\kappa}})^{\frac{i(r^\kappa -1)}{a}}\,.
\]
So, for any $m\geq 1$,  Lemma~\ref{lemm:Lfunc.characters} says that as $i$ varies over all elements of $(\Z/a\Z\smallsetminus\{0\})$ satisfying $\kappa_{r,a}(i)\mid m$, the character  $\llambda_{(i, \alpha)}\circ\norm_{\F_{r^m}/\F_{r_a}}$ varies over all characters $\lambda_1 \in M_{a}'(r^m)$. Similarly, as $j$ varies over all elements of $(\Z/b\Z\smallsetminus\{0\})$ satisfying $\kappa_{r,b}(j)\mid m$, the character  $\llambda_{(j, \alpha)}\circ\norm_{\F_{r^m}/\F_{r_b}}$ varies over all characters $\lambda_2 \in M_{b}'(r^m)$. 
Finally, recalling that $S = (\mathbb{Z}/a\mathbb{Z}\setminus \{0\})\times(\mathbb{Z}/b\mathbb{Z}\setminus \{0\})\times\F_q^\times$, we see that if $(i,j,\alpha) \in S$, then $\alpha \in \F_{q}^{\times}$. Altogether, we conclude that reindexing the sum on the right-hand side of \eqref{eqn:MessyProductOfGaussSums} gives the desired result.
\end{proof}

\subsection{Proof of Theorem \ref{theorem:Lfunction}} \label{sec:proof-of-Lfunction}

We make use of the notation introduced in the previous subsection.
By Lemma \ref{lemm:Lfunc.thefirst}, we have
\[\log L(J,T)  
= \sum_{m\geq 1}
\left(\sum_{\beta\in\F_{r^m}^\times}  A_J(\beta, m)\right)
\frac{T^m}{m}.\]
Combining Lemmas \ref{lemm:Lfunc.thesecond} and \ref{lemm:Lfunc.thethird} yields that, for all $m\geq 1$,
\[ \sum_{\beta\in\F_{r^m}^\times} A_J(\beta, m)
= - \sum_{\substack{\alpha\in\F_{r^m}\cap\mathbb{F}_q, \\(\lambda_1, \lambda_2)\in M'_{a,b}(r^m)}} \Gauss{\F}{\lambda_1}{\psi_{m,\alpha}} \Gauss{\F}{\lambda_2}{\psi_{m,\alpha}}\,.\]
Here, we may ignore the term $\alpha=0$ because $\Gauss{\F}{\lambda_1}{\psi_{m,0}} \Gauss{\F}{\lambda_2}{\psi_{m,0}}$ vanishes.
We combine this identity with Lemma \ref{lemm:Lfunc.thefourth} to obtain 
\[- \log L(J,T)  
= \sum_{m\geq 1}
\left( \sum_{\substack{o\in O \text{ s.t.}\\ |o| \text{ divides } m}} |o|\, \oomega(o)^{m/|o|}\right)
\frac{T^m}{m}.\]
On the other hand, expanding the logarithm, we see that 
\begin{align*}
-\log\prod_{o\in O}(1 - \oomega(o) T^{|o|}) 
	&= \sum_{o\in O} \log\left(1 - \oomega(o) T^{|o|}\right) 
	= \sum_{o\in O}\sum_{n\geq 1} \frac{\big(\oomega(o) T^{|o|}\big)^n}{n} \\
	&=\sum_{m\geq 1}\left(\sum_{\substack{o\in O \\ |o| \text{ divides } m}} |o|\, \oomega(o)^{m/|o|}\right)\cdot\frac{T^m}{m}.
\end{align*}
Therefore, 
\[\log L(J, T) = \log\prod_{o\in O}(1 - \oomega(o) T^{|o|}).\]
Exponentiating this identity concludes the proof of Theorem \ref{theorem:Lfunction}.

\hfill$\Box$

\begin{remark}\label{remark:deg_check}
We verify the degree of $L(J,T)$ using Theorem~\ref{thm:weil2}: 
\[\deg L(J,T)  = b(J) = \deg N_J  - 4g.\]
From this formula and the computation of $\deg N_J $ in Proposition~\ref{prop:DegOfCond}, we find
\[\deg(L(J,T)) = (a-1)(b-1)(q+1) - 4\frac{(a-1)(b-1)}{2} = (a-1)(b-1)(q-1)\,.\]

Alternately, from our computations in Theorem~\ref{theorem:Lfunction}, the degree of the $L(J,T)$ is $\sum_{o\in O}|o|$, where $O$ is the set of orbits $\langle r\rangle$ on $S = (\Z/a\Z\setminus\{0\})\times (\Z/b\Z\setminus\{0\})\times \mathbb{F}_q^\times$, where $r$ acts on $(i,j,\alpha)\in S$ via $r\cdot(i,j,\alpha) = (ri,rj,\alpha^{1/r})$, as defined in Section~\ref{sec:orbits}. The sum of the sizes of these orbits is equal to the size of $S$, namely $(a-1)(b-1)(q-1)$.
\end{remark}

\subsection[]{The BSD conjecture for $J$}\label{sec:bsd_conj_for_j}

The special value $L^\ast(J)$ of the $L$-function of $J$ at $T=r^{-1}$ is defined as
\[L^\ast(J) := \left. \frac{L(J, T)}{(1-rT)^{v}}\right|_{T=r^{-1}}, \quad\text{where } v=\ord_{T=r^{-1}} L(J,T).\]
This definition makes sense since the $L$-function is a rational function of $T$. (See Theorem~\ref{thm:weil2}.) 
By definition of $L(J, T)$, the function $\mathcal{L}: s\mapsto L(J, r^{-s})$ is positive on $[3/2, \infty)$. By the Riemann Hypothesis for $L$-functions of abelian varieties over $K$, the function $\mathcal{L}$ does not vanish on $(1, 3/2]$. The special value $L^\ast(J)$ is thus nonnegative. 
Since $L^\ast(J)$ is by definition a nonzero rational number, we conclude that $L^\ast(J)\in \Q_{>0}$.

Let $\widehat{J}$ denote the dual abelian variety to $K$ and let
\[
\langle \cdot,\cdot \rangle :J(K) \times \widehat{J}(K) \to \Q
\]
denote the canonical N\'eron--Tate height divided by $\log r$. Then, $\langle \cdot,\cdot \rangle$ is a bilinear pairing which is nondegenerate modulo torsion. Choosing a basis $P_1,\dots,P_r$ for $J(K)$ modulo torsion and a basis $\widehat{P_1},\dots,\widehat{P_r}$ for $\widehat{J}(K)$ modulo torsion, the regulator of $J$ is defined to be
\[
\Reg(J):= |\det\langle P_i,\widehat{P_j}\rangle_{1\le i,j \le r}|.
\]

These definitions provide us with the setting for the Birch and Swinnerton-Dyer conjecture, Theorem~\ref{thm:bsd}:

\begin{customthm}{\ref{thm:bsd}}
Let $C$ and $J$ be as above.
The abelian variety $J$ satisfies the Birch and Swinnerton-Dyer conjecture. 
This means that
\begin{itemize}
    \item The algebraic and analytic ranks of $J$ coincide:
    $\ord_{T=r^{-1}}L(J,T)=\rank J(K)$.
    \item The Tate--Shafarevich group $\Sh(J)$ is finite.
    \item The BSD formula holds:
    \begin{equation}
    L^*(J)=\frac{|\Sh(J)|\,\Reg(J)\,\prod_v c_v(J)}{H(J)\,r^{-g}\, |J(K)_{\mathrm{tors}}|^2},
    \end{equation}
    where the $c_v(J)$ are the local Tamagawa numbers of $J$ and $\Reg(J)$ is the regulator.
\end{itemize}
\end{customthm}

We refer the reader to \cite[\S6.2.3]{Ulmer_CurvesJacobians} for more background about the Birch and Swinnerton-Dyer conjecture for Jacobians over function fields.

\newcommand{\X}{\mathcal{X}}
\begin{proof}
Theorem~\ref{thm:bsd} is a special case of \cite[Theorem 3.1.2]{PriesUlmer2016}.
\end{proof}

This result will allow us to make conclusions about $\mathrm{rank}\,J(K)$ in Section~\ref{sec:rank}.

\begin{remark}
The more typical statement of the BSD formula is
\begin{equation}
    L^*(J)=\frac{|\Sh(J)|\,\Reg(J)\,\prod_v c_v(J)}{H(J)\, r^{-g}\, |J(K)_{\mathrm{tors}}|\, |J^\vee(K)_{\mathrm{tors}}|}.
    \end{equation}
In our case, $J$ is principally polarized since $J$ is the Jacobian of a curve, so $J \cong J^\vee$. 
In particular, $|J(K)_{\mathrm{tors}}|\, |J^\vee(K)_{\mathrm{tors}}| = |J(K)_{\mathrm{tors}}|^2$, and our statement agrees with the typical one.
\end{remark}

\section{Cohomological computation of $L(J, T)$}
\label{sec:Cohomological.Lfunction}

Our goal in this section is to provide an alternative computation of the $L$-function $L(J,T)$ using the geometry of the minimal proper regular SNC model $\mathcal S$ of $C$.
In particular, we compute the zeta function of $\mathcal S$ in two different ways -- first by decomposing it via the fibers over $\mathbb P^1$ and a second time by understanding the cohomology of $\mathcal S$ in terms of a product of curves which dominates $\mathcal S$. \\

Throughout the section, we denote by $H^n(-)$ the $n$\textsuperscript{th} $\ell$-adic cohomology group of a variety over $\F_r$. 
That is, $H^n(X)$ denotes $H^n_{\text{\'et}}(X\times_{\F_r}\overline{\F_r}, \overline{\Q_\ell})$ for a prime $\ell\neq p$. 
This cohomology group is endowed with a natural action of the geometric $r$\textsuperscript{th} power Frobenius $\Frob_r$.

The following linear algebra fact (also used in \cite{GriffonUlmer}, \cite{Ulmer2006}) will be useful for the linear algebra arguments in our cohomology computation:
\begin{lemma}\label{lemm:easylinearalgebra}
Let $V$ be a finite-dimensional vector space with subspaces $W_i$ indexed by $i\in\Z/m\Z$ such that $V=\bigoplus_{i\in\Z/m\Z}W_i$, 
and let $\phi:V\to V$ be a linear map such that $\phi(W_i)\subset W_{i+1}$ for all $i\in\Z/m\Z$. 
Then
\[\charpol{\phi}{V} = \det\ (1 - \phi^m T^m|W_0)\,.\]
\end{lemma}

\subsection{Preliminaries about Artin--Schreier curves}\label{subsec:ArtinSchreier}
For any prime-to-$p$ integer $d\geq 1$ and any power $q$ of $p$, let $X_{d,q}$ be the smooth projective curve over $\F_r$ defined by the affine equation
\[ X_{d,q}:\qquad w^d = z^q -z\,.\]
Since $d$ and $q$ are relatively prime, $X_{d,q}$ admits a unique point at infinity which we denote by $P_\infty\in X_{d,q}$. 
We note that $P_{\infty}$ is $\F_r$-rational. 
A straightforward application of the Riemann--Hurwitz formula yields that $X_{d,q}$ has genus $(q-1)(d-1)/2$. 
Hence, $\dim_{\Q_\ell}H^1(X_{d,q})= (q-1)(d-1)$.

Recall from \S\ref{sec:orbits} that we defined 
$S'_d = (\Z/d\Z\smallsetminus\{0\})\times\F_q^\times$ and endowed it with an action by $\langle r \rangle$, 
and let $O_{d}'$ be the set of orbits of $S_{d}'$ under this action. 
Moreover for any $(i,\alpha) \in S_{d}'$, we defined (in \S\ref{sec:gausssums.orbits}) an additive character $\llambda_{(i,\alpha)}$ and a multiplicative character $\Psi_{(i,\alpha)}$ on $\F_{r^{|o|'}}$. 
By construction, $\llambda_{(i,\alpha)}$ induces a character $\lambda_{(i,\alpha)}$ of $\mu_{d}$ by composition with the quotient map \[(\F_{r^{|o'|}})^{\times} \to (\F_{r^{|o'|}})^\times/\ker\lambda_{(i, \alpha)} \simeq\mu_{d/(d,i)}\subset \mu_d,\] 
and $\Psi_{(i,\alpha)}$ induces an additive character $\psi_{(i,\alpha)}$ of $\F_{q}$ by composition with the trace map $\Tr_{\F_{r^{|o'|}}/\F_{q}}$. The map which takes $(i,\alpha)$ to the product character $\lambda_{(i,\alpha)}\psi_{(i,\alpha)}$ is a bijection between $S_{d}'$ and the group of characters of $\mu_{d} \times \F_{q}$. 

The curve $X_{d,q} \times_{\F_r} \overline{\F_{r}}$ is naturally endowed with an action of $\mu_d\times\F_q$, defined as follows: for any $\zeta \in\mu_d$ and any $\alpha\in\F_q$, set
$(\zeta, \alpha)\cdot (w,z) := (\zeta w, z+\alpha)$
for any $(w,z)\in X_{d,q}\smallsetminus\{P_{\infty}\}$, and $(\zeta, \alpha) \cdot P_{\infty} = P_{\infty}$. By the functoriality of cohomology, this induces an action of  $\mu_d\times\F_q$ on $H^1(X_{d,q})$. For any $(i, \alpha)\in S'_d$, we denote by $H^1(X_{d,q})^{(i, \alpha)}$ the subspace of $H^1(X_{d,q})$ on which $\mu_d\times\F_q$ acts as multiplication by $\lambda_{(i,\alpha)}\psi_{(i,\alpha)}\,.$ 

By \cite{katz1981crystalline}, each $H^1(X_{d,q})^{(i, \alpha)}$ has dimension $1$. In particular, $H^1(X_{d,q})$ decomposes as a direct sum of lines:
\begin{equation}\label{eq:eigenvalue.Fr11}
H^1(X_{d,q}) = \bigoplus_{(i, \alpha)\in S'_d} H^1(X_{d,q})^{(i, \alpha)}.
\end{equation}
The action of $\Frob_r$ on $H^1(X_{d,q})$ sends the line $H^1(X_{d,q})^{(i, \alpha)}$ indexed by $(i, \alpha)\in S'_d$ onto the line indexed by $(ri, \alpha^{1/r})$. 
We deduce from the above that, for any orbit $o'\in O'_{d}$, the $|o'|$\textsuperscript{th} iterate of $\Frob_r$ stabilizes the line $H^1(X_{d,q})^{(i, \alpha)}$ for any representative $(i, \alpha)\in o'$.
By \cite{katz1981crystalline}, the eigenvalue of $(\Frob_r)^{|o'|}$ acting on the line $H^1(X_{d,q})^{(i, \alpha)}$ is the Gauss sum $\Gauor{o'}$ which we defined in \S\ref{sec:gausssums.orbits}, Definition~\ref{def:oomega}.
In other words, we have
\begin{equation}\label{eq:eigenvalue.Fr}
\charpol{\Frob_r^{|o'|}}{H^1(X_{d,q})^{(i, \alpha)}} = 1- \Gauor{o'} T,
\end{equation}
for any $(i, \alpha)\in o'$.
Furthermore, the direct sum
\[H^1(X_{d,q})_{o'}:=\bigoplus_{(i,\alpha)\in o'}H^1(X_{d,q})^{(i, \alpha)}\]
is stable under the action of $\Frob_r$, and the action of $\Frob_r$ cyclically permutes the summands thereof. 
By Lemma~\ref{lemm:easylinearalgebra}, we thus have 
\begin{equation*}
\charpol{\Frob_r}{H^1(X_{d,q})_{o'}} =1 - \Gauor{o'} T^{|o'|}\,.
\end{equation*}
We conclude that 
\begin{equation*}
\charpol{\Frob_r}{H^1(X_{d,q})}	 
= \prod_{o'\in O'_{d}} \left(1-\Gauor{o'}T^{|o'|}\right)\,, 
\end{equation*}
is the $L$-function of the curve $X_{d,q}/\F_r$ (\ie{}, the numerator of its Hasse--Weil $\zeta$-function, viewed as a rational function in $T$).

\subsection{Domination by a product of curves}\label{subsec:domination_by_product}

Let $a,b\geq 1$ be relatively prime integers which are both coprime to $p$, and let $q$ be a power of $p$. 

Let $X_{a}$ and $Y_{b}$ be smooth projective curves over $\overline{\F_{r}}$ defined by the (singular) affine equations
\begin{align*}
X_{a}: x^a = u_1\,, \\
Y_{b}: y^b = u_2\,.
\end{align*}
Let $\infty_{a}$ denote the unique point at infinity on $X_{a}$ and let $\infty_{b}$ denote the unique point at infinity on $Y_{b}$. 
Let $\mathcal P$ be the product $X_{a} \times Y_{b}$ and let $\pi:\mathcal S_{0}\to \P^1_{\overline{\F_{r}}}$ be the minimal proper regular model of the curve with affine equation $x^a + y^b = u$ over $\overline{\F_r}(u)$.

The surface $\mathcal P$ is equipped with a rational map $\pi_{0}: \mathcal P \dashrightarrow \P^1$ defined on the affine patch by
\begin{center}
	\begin{tabular}{rccc}
		$\pi_{0}$\ : & $\mathcal P$ & $\dashrightarrow$ & $\P^1$, \\
		& $((x,u_1), (y,u_2))$ & $\mapsto$ & $u_1 + u_2$\,.
	\end{tabular}   
\end{center}
The rational map $\pi_{0}$ also maps $\{\infty_{a}\} \times (Y_{b} \smallsetminus \{\infty_{b}\})$ and $\{\infty_{a}\} \times (Y_{b} \smallsetminus \{\infty_{b}\})$ to $\infty \in \P^1$, and has a unique point of indeterminacy at $(\infty_{a},\infty_{b})$. As is explained in the proof of Proposition~3.1.5 of \cite{PriesUlmer2016}, one can resolve the indeterminacy in $\pi_{0}$ through a series of blow-ups at the point of indeterminacy. Moreover, as \cite{PriesUlmer2016} explains in Remark~3.1.6, the exceptional fiber of the last blow-up maps isomorphically to $\mathbb P^1$ and all other fibers map to $\infty \in \P^1$. Let $\mathcal R$ be the result of this blow-up. Examining the construction and comparing to the recipe for constructing minimal proper regular SNC models from \cite{Dokchitser2018}, we find that in fact, $\mathcal R$ is the minimal proper regular model of the curve with affine equation $x^a + y^b = u$ over $\overline{\F_r}(u)$. 

Let $\mathcal{P}_{q,q} = X_{a,q} \times Y_{b,q}$. The surface $\mathcal P_{q,q}$ is a Galois cover of $\mathcal P$ with Galois group $\mathbb F_{q} \times \mathbb F_{q}$. Let $\mathcal R_{q,q}$ be the fiber product $\mathcal R \times_{\mathcal{P}} \mathcal{P}_{q,q}.$ Then, $\mathcal R_{q,q}$ is a Galois cover of $\mathcal R$ with Galois group $\mathbb F_{q} \times \mathbb F_{q}$. There is an `antidiagonal' action of $\F_{q}$ on $\mathcal P_{q,q}$ and $\mathcal R_{q,q}$ where $\alpha$ acts by $(\alpha, -\alpha)$ and this action preserves fibers of the rational map $\mathcal P_{q,q}$ to $\mathbb P^1$. Let $\mathcal P_{q} \colonequals \mathcal P_{q,q}/\F_{q}$ and $\mathcal R_{q} \colonequals \mathcal R_{q,q}/\F_{q}$ be the quotients by this action. By construction, $\mathcal P_{q}$ is a $\F_{q}$-Galois cover of $\mathcal P$ and $\mathcal R_{q}$ is a $\F_{q}$-Galois cover of $\mathcal R$. We can also recognize $\mathcal P_q$ and $\mathcal R_{q}$ as pullbacks. We have $\mathcal P_{q} = \mathcal P \times_{\P^1_{u}} \P^1_{t}$ and $\mathcal R_{q} = \mathcal R \times_{\mathcal{P}} \mathcal{P}_{q}$. We summarize these maps in the following commutative diagram:

\begin{equation*}
\begin{tikzcd}
 \mathcal R_{q,q} \arrow[r, "/\F_{q}"]\arrow[d]\arrow[dr, phantom, "\scalebox{1.5}{$\lrcorner$}" , very near start, color=black] & \mathcal R_{q} \arrow[r]\arrow[d]\arrow[dr, phantom, "\scalebox{1.5}{$\lrcorner$}" , very near start, color=black] & 
 \mathcal R \arrow[d]\\
 \mathcal P_{q,q} = X_{a,q} \times Y_{b,q} \arrow[r, "/\F_{q}"]\arrow[d, dashed] &  \mathcal P_{q} \arrow[r]\arrow[d, dashed]\arrow[dr, phantom, "\scalebox{1.5}{$\lrcorner$}", very near start, color=black]  &  \mathcal P = X_{a} \times Y_{b} \arrow[d, dashed, "\pi_{0}"] \\
 \P^1_{t} \arrow[r, equals] & \P^1_{t} \arrow["u = t^q - t"]{r} & \P^1_{u}
\end{tikzcd}
\end{equation*}

We now relate the surfaces appearing in the commutative diagram above to the minimal proper regular SNC model $\mathcal S$ of $\mathcal C_{a,b}$, as defined in \S\ref{sec:geometry}.

First, let $\pi:\mathcal S_{0}\to \P^1_{\overline{\F_{r}}}$ be the minimal proper regular model of the curve with affine equation $x^a + y^b = u$ over $\overline{\F_r}(u)$. There is a rational map $\phi: \mathcal P \to \mathcal S_{0}$ defined on the affine patch by 
\begin{center}
	\begin{tabular}{rccc}
		$\phi$\ : & $\mathcal P$ & $\dashrightarrow$ & $\mathcal S_0$, \\
		& $((x,u_1), (y,u_2))$ & $\mapsto$ & $(x,y,u_1 + u_2)$\,.
	\end{tabular}   
\end{center}
The rational map $\phi$ has a unique point of indeterminacy at $(\infty_{a}, \infty_{b})$, and this indeterminacy can be resolved by the same series of blow-ups that resolves $\pi_{0}$, yielding a morphism $\phi: \mathcal R \to \mathcal S_{0}$. In fact, we have already remarked that $\mathcal R$ is the minimal proper regular model of the curve $x^a + y^b = u$, and $\phi: \mathcal R \to \mathcal S_{0}$ is an isomorphism.

Now, set $\mathcal S_{q} \colonequals \mathcal S_{0} \times_{\P^1_u} \P^1_t$ where the second fiber maps $\P^1_t \to \P^1_u$ via the Artin--Schreier map $t \mapsto t^q - t$, so that $\mathcal S_{q}$ is a model of $x^a + y^b = t^q - t$. The rational map $\phi_{q,q}: \mathcal P_{q,q} \dashrightarrow \mathcal S_{q}, ((x,t_1),(y,t_2)) \mapsto (x,y,t_1+t_2)$ is invariant under the antidiagonal $\F_{q}$-action. The induced rational map $\phi_{q}: \mathcal P_{q} \dashrightarrow \mathcal S_{q}$ from the quotient is the same as the pullback of $\phi: \mathcal P \dashrightarrow \mathcal S_{0}$. We now resolve the indeterminacy of these rational maps.

The isomorphism $\phi: \mathcal R \to \mathcal S_{0}$ pulls back to an isomorphism $\phi_{q}: \mathcal R_{q} \to \mathcal S_{q}$ which resolves the indeterminacy of $\phi_{q}: \mathcal P_{q} \dashrightarrow \mathcal S_{q}$. Moreover, the induced map $\mathcal R_{q,q} \to \mathcal S_{q}$ given by composing $\phi_{q}$ with the antidiagonal quotient resolves the indeterminacy of the rational map $\phi_{q,q}: \mathcal P_{q,q} \to \mathcal S_{q}$.

In Section~\ref{subsec:FrobOnH2}, these morphisms will allow us to relate the action of Frobenius on the `antidiagonal $\F_q$'-invariant subspace of $H^2(\mathcal P_{q,q})$ to the action of Frobenius on $H^2(\mathcal S_{q})$ modulo its `trivial lattice'.

We summarize in Figure~\ref{fig.diagram.maps} the maps considered here in a commutative diagram, where dashed arrows denote rational maps and solid arrows are everywhere defined. The maps from $\mathcal R_{q,q}, \mathcal R_{q},$ and $\mathcal R$ resolve the indeterminacy of the maps from $\mathcal P_{q,q}, \mathcal P_{q}$ and $\mathcal P$ with the same targets.\\

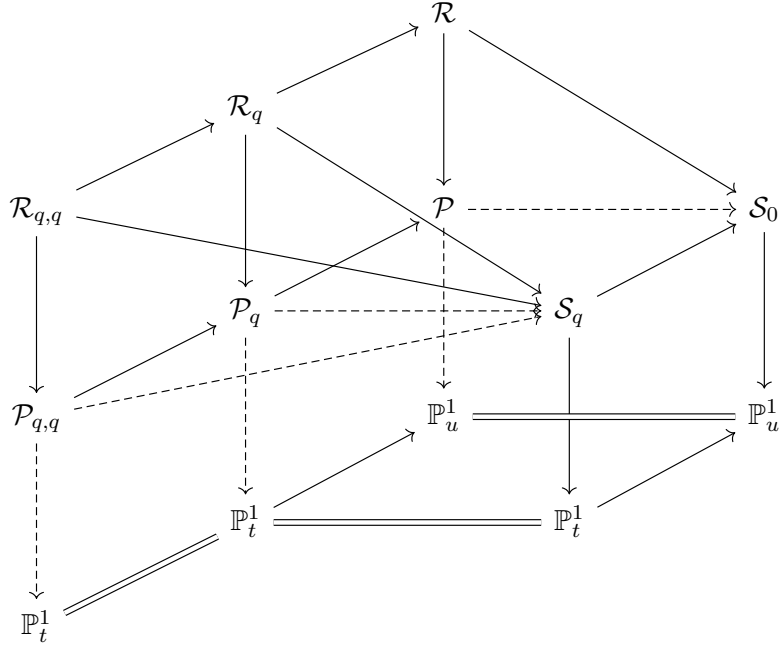
\begin{figure}[h]
\begin{center}
\begin{tikzcd}
                 & &                & & \mathcal R\arrow[rrrdd]\arrow[dd] &                & & \\
                 & & \mathcal R_{q} \arrow[rru]\arrow[rrrdd]\arrow[dd]  & &            &                & & \\
\mathcal R_{q,q}\arrow[rru]\arrow[dd]\arrow[rrrrrd] & &                & & \mathcal P\arrow[rrr,dashed]\arrow[dd,dashed] &                & & \mathcal S_{0} \arrow[dd]\\
                 & & \mathcal P_{q}\arrow[rru]\arrow[rrr,dashed] \arrow[dd,dashed] & &            & \mathcal S_{q}\arrow[dd] \arrow[rru] & & \\
\mathcal P_{q,q}\arrow[rru]\arrow[dd,dashed] \arrow[rrrrru,dashed]& &                & & \P^1_u \arrow[rrr,equals]    &                & & \P^1_{u} \\
                 & & \P^1_t  \arrow[rru] \arrow[rrr,equals]      & &            & \P^1_{t}\arrow[rru]       & & \\
\P^1_t  \arrow[rru,equals]         & &                & &            &                & &         
\end{tikzcd}
\caption{Summary of maps}\label{fig.diagram.maps}
\end{center}
\end{figure}

Finally, we relate $\mathcal S_{q}$ to $\mathcal S$. In Section~\ref{subsec:cohom_of_S}, this relationship will allow us to identify the action of Frobenius on $H^2(\mathcal S)$ modulo its `trivial lattice' to the action of Frobenius on $H^2(\mathcal S_{q})$ modulo its `trivial lattice'.

Upon restricting to the fibers over $\P^1 \smallsetminus (\F_{q} \cup \{\infty\})$, the surfaces $\mathcal S$ and $\mathcal S_{q}$ become isomorphic as models of $\mathcal C_{a,b}$. However, since $\mathcal S_{q}$ is a ramified cover of $\mathcal S_0$, the surface $\mathcal S_{q}$ may not be a regular model for $\mathcal C_{a,b}$, and there need not be morphisms between $\mathcal S_{q}$ and $\mathcal S$ in either direction. 

Now, $\mathcal S_{q} \to \mathcal S_{0}$ is \'etale away from the fiber above infinity, so the only singularities of $\mathcal S_{q}$ lie on the fiber above infinity. When blowing up these singularities to get a proper regular model, the exceptional fibers all map to $\infty \in \P^1_{t}$. After further blow-ups at the singularities on fibers, one gets a proper regular SNC model $\mathcal S'$ of $\mathcal C_{a,b}$ equipped with a blow-up map $\mathcal C_{a,b} \to \mathcal S_{q}$. The exceptional fibers of the blow-ups are components of the singular fibers (above $\F_{q}$ and $\infty$). By the minimality of $\mathcal S$ and since $\mathcal S, \mathcal S'$, and $\mathcal S_{q}$ are all isomorphic away from the singular fibers, the birational isomorphism $\mathcal S' \to \mathcal S$ defined away from the singular fibers extends to a morphism which is defined by iteratively contracting certain $-1$ curves which are contained in singular fibers of the composition $\mathcal S' \to \mathcal S_{q} \to \P^1_{t}$.

\subsection{Cohomology of $\mathcal S$ in degree $1$}\label{subsec:cohom_of_S_1}

Our next goal is to show that the $H^1$ of the minimal proper regular SNC model $\mathcal S$ of $C$ is trivial by comparing it with the cohomology of the product of Artin--Schreier curves $\mathcal P_{q,q}$ constructed in Section~\ref{subsec:domination_by_product}.

First, we relate the cohomology of $\mathcal R_{q}$ to the cohomology of the curves $X_{a,q}$ and $Y_{b,q}$. Since we construct $\mathcal R_{q}$ from $\mathcal P_{q}$ by repeatedly blowing up at a point and the exceptional divisor (as a union of $\mathbb P^1$s) has trivial $H^1$, the blow-up formula (see \cite{MilneEC}) gives 
\begin{align} \label{eqn:H1-blowup}
    H^1(\mathcal R_{q}) \cong H^1(\mathcal P_{q})\,.
\end{align}
 Since $\mathcal P_{q} = (X_{a,q} \times Y_{b,q})/\F_q$, we have
 \begin{align} \label{eqn:H1-invariants}
    H^1(\mathcal P_{q}) \cong H^1(X_{a,q} \times Y_{b,q})^{\F_{q}}\,.
\end{align}
The Kunneth formula gives  
 \begin{align} \label{eqn:H1-Kunneth}
    H^1(X_{a,q} \times Y_{b,q}) \cong (H^1(X_{a,q}) \otimes H^0(Y_{b,q}) \oplus H^0(X_{a,q}) \otimes H^1(Y_{b,q}))^{\F_{q}} \,.
\end{align}
Now, $\F_{q}$ acts trivially on $H^0(X_{a,q})$ and $H^0(Y_{b,q})$, and we saw in Section~\ref{subsec:ArtinSchreier} that the subspaces of $H^1(X_{a,q})$ and $H^1(Y_{b,q})$ fixed by $\F_{q}$ are both trivial. So, combining \eqref{eqn:H1-blowup}, \eqref{eqn:H1-invariants}, and \eqref{eqn:H1-Kunneth}, we find
$H^1(\mathcal R_{q}) = \{0\}\,.$ Since $\mathcal R_{q} \to \mathcal S_{q}$ is a dominant morphism, the induced map $H^1(\mathcal R_{q}) \to H^1(\mathcal S_{q})$ is surjective, whence $H^1(\mathcal S_{q})$ is trivial. Using the blow-up formula as in the justification of \eqref{eqn:H1-blowup} gives $H^1(\mathcal S_{q}) \cong H^1(\mathcal S)$. We conclude that $H^1(\mathcal S) = \{0\}\,.$

\subsection{Cohomological interpretation of the $L$-function}\label{subsec:FrobOnH2}

Our goal in this subsection is to relate $L(J,T)$ to the characteristic polynomial of Frobenius acting on a certain quotient of $H^2(\mathcal S)$. 

As before, let $K=\F_r(t)$. We choose an algebraic closure $\overline{K}$ of $K$ and a separable closure $K^{\mathrm{sep}}$ within $\overline{K}$. Denote by $G$ the absolute Galois group of $K$. 
Fix a pair $(a,b)$ of positive coprime integers which are both coprime to $p$ as well as a power $q$ of $p$.
Write $C=C_{a,b,q}$ and $J=J_{a,b,q}$.

For any place $v$ of $K$, we let  $\Frob_{v}$ denote the \emph{geometric} Frobenius at $v$. (The geometric Frobenius $\Frob_{v}$ is a well-defined up to conjugacy in $G$.)
Recall from \S\ref{ssec:Lfunction.definition} that the $L$-function of $J$ is defined by
\begin{align}\label{eqn:L-function_def}
   L(J,T) \colonequals \prod_{v} \charpol{\Frob_v}{H^1(\Jcal_v)^{I_v}}^{-1}.
\end{align}
If $v$ is a place of bad reduction of $J$, we know from  Proposition~\ref{prop:unipotentreduction} that $J$ has unipotent reduction at $v$. Hence, by \cite[pg. 504, Remark 2]{SerreTate}, the action of  inertia group at $v$ on $H^1(\Jcal_v)$ only fixes  the trivial subspace, so that $H^1(\Jcal_v)^{I_v} =\{0\}$. 
On the other hand, if $v$ is a place of good reduction of $J$,  we have $H^1(\Jcal_v)^{I_v} = H^1(\Jcal_v)$ since $I_{v}$ acts trivially. 
Furthermore, at such a place $v$, the space $H^1(\Jcal_v)$ is canonically isomorphic to $H^1(\mathcal S_v)$ by (for instance) \cite[5.3.5]{Poonen_Curves}, compatibly with the action of $\Frob_v$. 
The Euler product in \eqref{eqn:L-function_def} thus simplifies to
\begin{align} \label{eqn:L-function_def_simp}
   L(J,T) = \prod_{v \text{ good}} \charpol{\Frob_v}{H^1(\mathcal S_v)}^{-1},
\end{align}
where the product is restricted to places of good reduction of $J$.
In order to shorten notation, we set
$P_v(T) := \charpol{\Frob_v}{H^1(\mathcal S_v)}$ for any place $v$ of $K$. \\

For a variety $X$ over $\F_r$, recall (e.g. from
\cite[Def. 3.4.1]{Poonen_Curves}) that its zeta function is defined by
\[
Z(X,T) = \prod_{ P \in |X|} \left(1 - T^{\deg P}\right)^{-1}\,,\]
where the product runs over the set of closed points of $X$.
If $X$ is smooth and projective, by Grothendieck--Lefschetz trace formula (see \cite[Corollary 3.7]{SGA4.5_Deligne}), we have
\[Z(X,T) = \prod_{i=0}^{2 \dim X} (-1)^{i+1} \charpol{\Frob_r}{H^i(X)}.
\]

In particular, we have $Z(\P^1_{\F_r}, T) = \big({(1-T)(1 - rT)}\big)^{-1}$.

We showed in Section~\ref{subsec:cohom_of_S_1} that $H^1(\mathcal S)=\{0\}$. 
It follows from Poincar\'e duality (see \cite[Appendix C.3]{Hartshorne}) that $H^3(\mathcal S)=\{0\}$ as well.
These remarks show that
\begin{equation}\label{eq:zetaS.expr1}
    Z(\mathcal S, T) = \frac{1}{(1-T)\, \charpol{\Frob_r}{H^2(\mathcal S)} \, (1-r^2 T)}.
\end{equation}
Similarly, for any place $v$ of good reduction, we have
\[ Z(\mathcal S_{v}, T)  = \frac{P_v(T)}{(1-T^{\deg v})(1 - (rT)^{\deg v})}.\]
Since $\mathcal S$ is a disjoint union of the fibers of the map $\mathcal{S}\to\P^1$, we can also express $Z(\mathcal S, T)$ in terms of the zeta functions of the fibers: 
\[Z(\mathcal S, T) = \prod_{v} Z(\mathcal S_{v}, T) =  \prod_{v \text{ good}} Z(\mathcal S_{v}, T) \prod_{v \text{ bad}} Z(\mathcal S_{v}, T)\,.\]

Combining the last two displayed formulas and \eqref{eqn:L-function_def_simp}, we find that 
\begin{align*}
     \prod_{v \text{ good}} Z(\mathcal S_{v}, T) 
     & = \prod_{v \text{ good}} \frac{P_v(T)}{(1-T^{\deg v})(1 - (rT)^{\deg v})}
     = \prod_{v \text{ good}} \frac{1}{P_{v}(T)^{-1}}\,\frac{1}{(1-T^{\deg v})(1 - (rT)^{\deg v})} \\
    & =\left(\prod_{v \text{ good}} \frac{1}{P_{v}(T)^{-1}}\right) 
     \left(\prod_{v} \frac{1}{(1-T^{\deg v})(1 - (rT)^{\deg v})} \right) \left(\prod_{v \text{ bad}} (1-T^{\deg v})(1 - (rT)^{\deg v})\right) \\
    & = \frac{Z(\P^1_{\F_r}, T) Z(\P^1_{\F_r}, rT)}{L(J,T)} \left(\prod_{v \text{ bad}} (1-T^{\deg v})(1 - (rT)^{\deg v})\right)
\end{align*}
This gives us another expression for $Z(\mathcal S, T)$:
\begin{equation}\label{eq:zetaS.expr2}
Z(\mathcal S, T) = \frac{1}{(1-T)(1-r T)^2(1-r^2 T) L(J,T)} \prod_{v \text{ bad}} Z(\mathcal S_{v},T) (1-T^{\deg v})(1 - (rT)^{\deg v})\,.
\end{equation}

In fact, we can simplify this further since we know (from Section \ref{subsec:MinPropRegSNCModel}) that the fiber $\mathcal S_{v}$ at a place $v$ of bad reduction is a tree of $\mathbb P^1$s. 
For any such place $v$, let $m_{v}$ be the number of irreducible components of $\mathcal{S}_{v}$. 
Then, a straightforward computation shows that
\begin{align*}
Z(\mathcal S_{v},T)  & = \frac{Z(\P^1_{\F_v},T)^{m_v}}{Z(\Spec \F_v,T)^{m_v - 1}}   
 = \frac{1}{(1-T^{\deg v})(1 - (rT)^{\deg v})^{m_{v}}}\,.
\end{align*}
 Plugging this into \eqref{eq:zetaS.expr2} yields that 
\begin{equation}\label{eq:zetaS.expr3}
Z(\mathcal S, T) = \frac{1}{(1-T)(1-r T)^2(1-r^2 T) L(J,T)} \prod_{v \text{ bad}} {(1 - (rT)^{\deg v})^{1-m_{v}}} \,. 
\end{equation}

Comparing formulas \eqref{eq:zetaS.expr1}  and \eqref{eq:zetaS.expr3} for $Z(\mathcal S, T)$ and rearranging terms, we find
\begin{align}
   L(J,T)  & = \frac{P_2(T)}{(1-rT)^2}\prod_{v \text{ bad}} Z(\mathcal S_{v},T) (1-T^{\deg v})(1 - (rT)^{\deg v}) \notag\\
   & = \frac{P_2(T)}{(1-rT)^2} \, \prod_{v \text{ bad}} (1 - (rT)^{\deg v})^{1- m_{v}}\,.\label{eq:Lfunc.zetaS}
\end{align}

Let $s_{\infty} : \mathbb P^1 \to \mathcal S$ be the `infinity section' $s_\infty$ which maps each point $t \in \mathbb P^1$ to the unique `point at infinity' on the fiber $\mathcal S_{t}$. Let $\TrivLat \subset H^2(\mathcal S)$ be the trivial lattice, that is 
the subspace spanned by the images under the cycle class map of (the image of) $s_{\infty}$ and all components of fibers of $\mathcal S \to \P^1$. 

Let $D$ be an irreducible (over $\mathbb F_{r}$) component of a fiber of $\mathcal S \to \P^1$. 
After base change to $\overline{\F_{r}}$, we can decompose $D$ as $D_{\overline{\F_{r}}} = \bigcup_{j \in \Z/n\Z} D_{j}$ with indices chosen so that $\Frob_{r} D_{j} = D_{j+1}$. 
Let $W_{j}$ be the subspace of $H^2(\mathcal S)$ spanned by the image of $1_{D_{j}}$ under $i_{*}:H^0(D_{j})(-1) \to H^2(\mathcal S)$. 
We have $\Frob_{r} W_{j} \subset W_{j+1}$, and $\Frob_{r^n}$ acts on each $W_{j}$ by multiplication by $r^n$. Since $W_{j}$ is one-dimensional, we find $\det(1 - \Frob_{r}^{n} T^n|W_{0}) = 1 - r^n T^n$. Hence, by Lemma~\ref{lemm:easylinearalgebra}, the characteristic polynomial of $\Frob_{r}$ acting on the subspace of $H^2(\mathcal S)$ spanned by the classes of the components of $D_{\overline{\F_{r}}}$ is $(1-(rT)^n)$.

Now, the trivial lattice $\TrivLat$ has a basis consisting of the image of $s_\infty$ (which is defined over $\mathbb F_{r}$), the fiber over any $\mathbb F_{r}$-rational point of $\mathbb P^1$ (which is again defined over $\mathbb F_{r}$) and the components of the singular fibers which do not meet $s_\infty$. We conclude that 
\[
\charpol{\Frob_{r}}{\TrivLat} = (1-rT)^2\prod_{v \text{ bad}} (1 - (rT)^{\deg v})^{m_{v} - 1}\,.
\]
Combining \eqref{eq:Lfunc.zetaS} with the above finally yields the following:
\begin{proposition}\label{prop.Lfunc.H2S} 
We have
\[L(J, T) = \charpol{\Frob_r}{H^2(\mathcal S)/\TrivLat}\,.\]	
\end{proposition}
With our computation of the degree of the conductor of $J/K$ (see Proposition~\ref{prop:DegOfCond}), the N\'eron--Ogg--Shafarevich formula (see Appendix \ref{app:Conductor}) yields that $\deg L(J, T)=(a-1)(b-1)(q-1)$. 
It follows from the above  that 
\begin{equation}\label{eq.dimH2S}
    \dim H^2(\mathcal S)/\TrivLat = (a-1)(b-1)(q-1).
\end{equation}

\subsection{Cohomology of $\mathcal S$ in degree $2$}\label{subsec:cohom_of_S}

Our next goal is to relate the $H^2$ of the minimal proper regular SNC model $\mathcal S$ of $C$ to the cohomology of the product of Artin--Schreier curves $\mathcal P_{q,q}$ constructed in Section~\ref{subsec:domination_by_product}. Our strategy will mirror that of Section~\ref{subsec:cohom_of_S_1}. The main differences are that the blow-up divisor has nontrivial $H^2$, which we will need to track more carefully, and that we will need to use \eqref{eq.dimH2S} to show that the surjection we construct is actually an isomorphism. 

First, we relate the cohomology of $\mathcal R_{q}$ to the cohomology of the curves $X_{a,q}$ and $Y_{b,q}$. Let $B$ be the subspace of $H^2(\mathcal R_{q})$ spanned by the pullbacks of the blow-up divisor from $\mathcal R \to \mathcal P$ (see Section \ref{subsec:domination_by_product}). 
Successively applying the blow-up formula, taking invariants, and applying the K\"unneth formula, we find
\begin{align*}
H^2(\mathcal R_{q}) & \cong H^2(\mathcal P_q) \oplus B 
                     \cong H^2((X_{a,q} \times Y_{b,q})/\F_q) \oplus B 
                     \cong H^2(X_{a,q} \times Y_{b,q})^{\F_{q}} \oplus B \\
					& \cong (H^1(X_{a,q}) \otimes H^1(Y_{b,q}))^{\F_{q}} \oplus (H^0(X_{a,q}) \otimes H^2(Y_{b,q}))^{\F_{q}} \oplus (H^2(X_{a,q}) \otimes H^0(Y_{b,q}))^{\F_{q}} \oplus B.
\end{align*}

Now let $\TrivLat_{q}$ be the subspace of $H^2(\mathcal S_{q})$ which is spanned by components of fibers of $\mathcal S_{q} \to \mathbb P^1$ together with the class of the `infinity section' $s_{\infty,q}: \P^1 \to \mathcal S_{q}$ which takes $t \in \A^1 \subset \P^1$ to the unique `point at infinity' on that fiber. Recall that $\mathcal S$ is the minimal proper regular SNC model of $C$ and that we have defined $\TrivLat \subset H^2(\mathcal S)$ to be the trivial lattice. Since $\mathcal S$ and $\mathcal S_q$ are related by a series of blow-ups and blow-downs where the exceptional fibers lie in the fibers over $\P^1$, we automatically have $H^2(\mathcal S_{q})/\TrivLat_{q} \cong H^2(\mathcal S)/\TrivLat$. 

The blow-up divisor in $\mathcal R_{q}$ maps to the union of (the image of) the infinity section $s_{\infty,0}$ and the fiber at infinity of $\mathcal S_{0}$. Similarly, the blow-up divisor in $\mathcal R_{q}$ maps to the union of the infinity section $s_{\infty,q}$ and the fiber at infinity of $\mathcal S_{q}$. Moreover, the classes in $H^0(X_{a,q}) \otimes H^2(Y_{b,q})$ and $H^2(X_{a,q})\otimes H^0(Y_{b,q})$ are generated by the strict transforms of the images of $X_{a,q} \times \infty_{b}$ and $\infty_{a} \times Y_{b,q}$, which also map to the fiber above $\infty\in \P^1$ in $\mathcal S_{q}$.

All told, we find that the image of $(H^0(X_{a,q})\otimes H^2(Y_{b,q})) \oplus (H^2(X_{a,q}) \otimes H^0(Y_{b,q})) \oplus B$ under the induced map $H^2(\mathcal R_{q}) \to H^2(\mathcal S_{q})$ is contained in $\TrivLat_{q}$. Since $\mathcal R_{q} \to \mathcal S_{q}$ is a dominant morphism, the induced map $H^2(\mathcal R_{q}) \to H^2(\mathcal S_{q})$ is surjective and induces a Galois-equivariant canonical surjection
\[
\varpi: (H^1(X_{a,q}) \otimes H^1(Y_{b,q}))^{\F_{q}} \to H^2(\mathcal S_{q})/\TrivLat_{q} \cong H^2(\mathcal S)/\TrivLat\,.
\]
From the description of $(H^1(X_{a,q})\otimes H^1(Y_{b,q}))^{\F_q}$ obtained in Section~\ref{subsec:grand.conclusion}  below (see \eqref{eq:decomp.h1tensorh1}), we see that that space has dimension $(a-1)(b-1)(q-1)$. Formula \eqref{eq.dimH2S} in the previous subsection yields that $H^2(\mathcal S)/\TrivLat$ has the same dimension. 
We deduce that $\varpi$ is a Galois-equivariant isomorphism. Therefore, 
\begin{equation}\label{eq.Lfunc.H2.H1H1}
\charpol{\Frob_{r}}{H^2(\mathcal S)/\TrivLat} 
= \charpol{\Frob_{r}}{(H^1(X_{a,q}) \otimes H^1(Y_{b,q}))^{\F_{q}}}\,.
\end{equation}

\subsection{Computation of the $L$-function}
\label{subsec:grand.conclusion}

Combining Proposition \ref{prop.Lfunc.H2S} with \eqref{eq.Lfunc.H2.H1H1}, we find that 
\[L(J,T) = \charpol{\Frob_{r}}{(H^1(X_{a,q}) \otimes H^1(Y_{b,q}))^{\F_{q}}}.\]
Finally, we use the facts about the cohomology of Artin--Schreier curves from Section~\ref{subsec:ArtinSchreier} to give a more explicit expression for $L(J,T)$.
Recall from Section~\ref{subsec:ArtinSchreier} that we have 
\[
H^1(X_{a,q}) = \bigoplus_{(i,\alpha) \in S_{a}'} H^{1}(X_{a,q})^{(i,\alpha)}\,
\quad\text{ and }\quad 
H^1(Y_{b,q}) = \bigoplus_{(i,\alpha) \in S_{b}'} H^{1}(Y_{b,q})^{(i,\alpha)}\,.
\]
In each of these direct sums indexed by elements of $S'_a=(\Z/a\Z\smallsetminus\{0\}) \times \mathbb F_{q}$ or $S'_b$  
respectively, each summand $H^{1}(X_{a,q})^{(i,\alpha)}$ and $H^{1}(Y_{b,q})^{(i,\alpha)}$ is one-dimensional. 
This means that 
\[
H^1(X_{a,q}) \otimes H^1(Y_{b,q}) = \bigoplus_{(i_1,\alpha_1) \in S_{a}'} \bigoplus_{(i_2,\alpha_2) \in S_{b}'} H^{1}(X_{a,q})^{(i_1,\alpha_1)} \otimes   H^{1}(Y_{b,q})^{(i_2,\alpha_2)}
\]
decomposes as a direct sum of lines.
Tracing through the definitions, one sees that, among the lines $H^{1}(X_{a,q})^{(i_1,\alpha_1)} \otimes   H^{1}(Y_{b,q})^{(i_2,\alpha_2)}$, the $\F_q$-invariant lines are those indexed by pairs $(i_1, \alpha_1), (i_2, \alpha_2)$ with $\alpha_1=\alpha_2$.
So,
\begin{align}\label{eq:decomp.h1tensorh1}
(H^1(X_{a,q}) \otimes H^1(Y_{b,q}))^{\F_{q}} & = \bigoplus_{(i_1, i_2, \alpha) \in S}  H^{1}(X_{a,q})^{(i_1,\alpha)} \otimes   H^{1}(Y_{b,q})^{(i_2,\alpha)}\,.
\end{align}
We now compute the characteristic polynomial of Frobenius on this space in the same way that we computed the characteristic polynomial of Frobenius acting on $H^1(X_{d,q})$ in Section~\ref{subsec:ArtinSchreier}.
For any orbit $o \in O = O_{r,a,b,q}$ (as defined in Section~\ref{sec:orbits}) the $|o|$\textsuperscript{th} iterate of $\Frob_{r}$ stabilizes the line $H^{1}(X_{a,q})^{(i_1,\alpha)} \otimes   H^{1}(Y_{b,q})^{(i_2,\alpha)}$ for any representative $(i_1,i_2,\alpha) \in o'$. 
For any $(i_1,i_2,\alpha) \in o$, we deduce from the computation following~\eqref{eq:eigenvalue.Fr11} in Section~\ref{subsec:ArtinSchreier} that the eigenvalue of $(\Frob_{r})^{|o|}$ acting on the line $H^{1}(X_{a,q})^{(i_1,\alpha)} \otimes   H^{1}(Y_{b,q})^{(i_2,\alpha)}$ is $\oomega(o) = \Gauor{\pi_{a}(o)}^{\nu_{a}(o)}\Gauor{\pi_{b}(o)}^{\nu_{b}(o)}$. 
In other words, for any  $(i_1,i_2,\alpha) \in o$, we have
\[ \charpol{(\Frob_r)^{|o|}}{H^{1}(X_{a,q})^{(i_1,\alpha)} \otimes   H^{1}(Y_{b,q})^{(i_2,\alpha)}} = 1-\oomega(o)T.\]
Since $\Frob_{r}$ cyclically permutes the lines $H^{1}(X_{a,q})^{(i_1,\alpha)} \otimes   H^{1}(Y_{b,q})^{(i_2,\alpha)}$ for $(i_1,i_2,\alpha) \in o$, Lemma~\ref{lemm:easylinearalgebra} yields
\[
\charpol{\Frob_r}{\bigoplus_{(i_1, i_2, \alpha) \in o}  H^{1}(X_{a,q})^{(i_1,\alpha)} \otimes   H^{1}(Y_{b,q})^{(i_2,\alpha)}} = 1 - \oomega(o) T^{|o|}\,.
\]
Taking the product over all orbits $o\in O$, we finally obtain
\[
L(J,T) 
= \prod_{o\in O} \left( 1 - \oomega(o) T^{|o|} \right)\,.
\]
This confirms our result in Theorem \ref{theorem:Lfunction}.

\section[]{Rank and $\gp$-adic valuation of Gauss sums}\label{sec:rank}

By the BSD conjecture (Theorem \ref{thm:bsd}), we have
\begin{align} \label{eqn:BSD1.restated}
\rank J(K) = \ord_{T=r^{-1}}L(J, T)\,. 
\end{align}
In this section, we use our explicit expression for $L(J, T)$ from Theorem \ref{theorem:Lfunction} to study $\rank J(K)$ in terms of the parameters $a,b$, and $q$. 

\begin{lemma}\label{lem:rank-formula-orbits}
The rank of $J(K)$ is given by
\begin{equation}\label{eq:rank.formula}
    \rank J(K)
=  \left|\big\{o\in O : \oomega(o) = r^{|o|}\big\}\right|.
\end{equation}
\end{lemma}

\begin{proof}
Using \eqref{eqn:BSD1.restated} for the first equality and Theorem~\ref{theorem:Lfunction} for the second, we have
\[
\rank J(K) = \ord_{T=r^{-1}}L(J, T) = \ord_{T=r^{-1}}\prod_{o \in O} (1 - \oomega(o) T^{|o|}) = \sum_{o \in O} \ord_{T=r^{-1}} (1 - \oomega(o) T^{|o|})\,.
\]
The result follows immediately from the observation that
\[
\ord_{T = r^{-1}} (1 - \oomega(o)T^{|o|}) =
\begin{cases}
1 & \text{ if } \oomega(o) = r^{|o|}\,, \\
0 & \text{ otherwise.}
\end{cases}
\]
\end{proof}

\begin{theorem}\label{thm:rank.upperbound}
We have
\[0 \leq \rank J(K) \leq (a-1)(b-1)(q-1) = 2g(q-1).\]
\end{theorem}

\begin{proof} 
From \eqref{eq:rank.formula}, we see that $\rank J(K) \leq |O|$. Since $O$ is a set of orbits on a set of cardinality $(a-1)(b-1)(q-1),$ we have $|O|\leq (a-1)(b-1)(q-1)$. 
\end{proof}

In the remainder of this section, we estimate the rank of $J(K)$ more precisely than in Theorem~\ref{thm:rank.upperbound} under various assumptions on $a, b,$ and $q$\,. In \S\ref{sec:rank.0}, we provide conditions on $a,b,q$ so that $\rank J(K) = 0$. In \S\ref{sec:rank.large}, we provide conditions so that $\rank J(K)$ is ``large,'' that is, such that the upper bound in Theorem~\ref{thm:rank.upperbound} is tight.

In order to refine our bounds on $\rank J(K)$, we estimate the right-hand side of \eqref{eq:rank.formula} using explicit results about the Gauss sums appearing in $\oomega(o)$. We gather the necessary results in subsections \ref{ss:explicit.Gauss.sums} and \ref{ss:p.adic.val.Gauss}.

\subsection{Explicit Gauss sums}\label{ss:explicit.Gauss.sums}
Let $n\geq 2$ be a prime-to-$p$ integer. As in \S\ref{sec:orbits}, we consider the set $S'_n := (\Z/n\Z\smallsetminus\{0\})\times \F_q^\times$ equipped with its action of $\langle r\rangle$.
We write $O'_n$ for the set of orbits of this action.
In this subsection, we describe situations where the values of the Gauss sums $\Gauor{o'}$ (for $o'\in O'_n$) may be explicitly determined. We refer to \S\ref{sec:gausssums.orbits} for the definition of $\Gauor{o'}$.\\

Recall that for any prime-to-$p$ integer $n \geq 1$, we denote by $o_p(n)$ the multiplicative order of $p$ modulo $n$ \ie{}, $o_p(n)$ is the least integer $e\geq 1$ such that $p^e\equiv 1 \bmod{n}$.

\begin{definition}[Supersingular Integer]\label{def:supersingular_integer}
A positive prime-to-$p$ integer $n$ is called \emph{supersingular (for~$p$)} if there exists a positive integer $\nu\geq 1$ such that $p^\nu\equiv -1\pmod{n}\,.$
\end{definition}

\begin{lemma}\label{lemma:even-order-orbit}
Suppose that $n$ is supersingular for $p$ and $[\F_{r}: \F_{p}]$ is odd. Let $o' \in O'_{n}$ be an orbit with representative $(i,\alpha)$. If $2i \neq n$, then the cardinality of $o'$ is even. 
\end{lemma}

\begin{proof}
Note that if $p^{\nu} \equiv -1 \pmod{n}$ and $d|n$, then $p^{\nu} \equiv -1 \pmod{d}$. So, if $n$ is supersingular for $p$, then any divisor of $n$ is supersingular for $p$.

If $d > 2$ is a divisor of $n$ and $\nu_{0}$ is the least positive integer such that $p^{\nu_{0}} \equiv -1 \pmod{d},$ we have $o_{p}(d) = 2 \nu_{0}$. 
In particular, the order $o_{p}(d)$ is even. Since $r$ is an odd power of $p$, the multiplicative order of $r$ modulo $d$ is also even.

Given $o' \in O'_{n}$, choose a representative $(i, \alpha) \in S_{n}'$. 
Since $2i \neq n$, we have $n/\gcd(n,i) > 2$. In particular, the previous paragraph implies that $o_{r}(n/\gcd(n,i))$ is even. On the other hand, we know from equation \eqref{eqn:orbitsize.prime} that
\[
|o'|=\lcm\left(o_{r}\left(\frac{n}{\gcd(n,i)}\right),[\F_r(\alpha):\F_r]\right)\,,
\]
whence we conclude that $|o'|$ is even.
\end{proof}

We now describe situations where one can compute $\Gauor{o'}$ explicitly.

\begin{lemma}\label{lemma:gauss-sum-quadratic}
Let $p \neq 2$ be an odd prime. Let $n \geq 2$ be an even integer and let $o' \in O'_{n}$ be an orbit with representative $(n/2, \alpha) \in S'_{n}$.  Then,
\[
\Gauor{o'}^2 = (-1)^{(p-1)|o'|\,[\F_{r}:\F_{p}]} \, r^{|o'|}\,.
\]
If $[\F_{r}:\F_{p}]$ is a multiple of $4$, then
\begin{equation}\label{eqn:quadratic-Gauss-Sum-simple-2}
\Gauor{o'} = \llambda_{(n/2,\alpha)}(\alpha)^{-1}  r^{|o'|/2}\,.
\end{equation}
If $[\F_{r}: \F_{p}]$ is a multiple of $4$ and $\alpha$ is an square in $(\F')^\times$, then
\begin{equation} \label{eqn:quadratic-Gauss-Sum-extra-simple}
\Gauor{o'} = r^{|o'|/2}\,.
\end{equation}
\end{lemma}

\begin{proof}
Write $\F'$ for the extension of $\F_{r}$ of degree $|o'|$.
By Definition~\ref{def:bold_g}, 
$\Gauor{o'} = \Gauss{\F'}{\llambda_{(n/2, \alpha)}}{\Psi_{(n/2, \alpha)}}$. Now, $\llambda_{(n/2,\alpha)} = \cchi^{(r^{|o'|}-1)/2}$ is a quadratic character on $(\F')^{\times}$. The first claim then follows from a short computation on Gauss sums for quadratic characters dating back to Gauss. See \cite[Lemma~6.1]{Cyclo} for a proof.

For the second claim, we note that if $[\F_{r}:\F_{p}]$ is a multiple of~$4$, then $[\F':\F_{p}]$ is a multiple of~$4$. Let~$\F$ denote the subextension of $\F'/\F_p$ with $[\F':\F] = 4\,.$ We deduce from equation \eqref{eqn:additive-char-Gauss-sum}  in~\S\ref{ssec:gauss.sums.basic} that 
\begin{equation}\label{eq:expl.gauss.0}
    \Gauor{o'} = \llambda_{(n/2, \alpha)}(\alpha)^{-1}\Gauss{\F'}{\llambda_{(n/2, \alpha)}}{\psi_{\F', 1}}\,.
\end{equation}
Then, the Hasse--Davenport relation (\eqref{eqn:HasseDavenportRelation} in \S\ref{ssec:gauss.sums.basic}) implies that
\begin{equation}\label{eq:expl.gauss.2b}
\Gauss{\F'}{\llambda_{(n/2, \alpha)}}{\psi_{\F', 1}} 
= \Gauss{\F'}{{\cchi|_{\F}}^{(|\F|-1)/2} \circ\norm_{\F'/\F}}{\psi_{\F, 1}\circ\Tr_{\F'/\F}} 
= \Gauss{\F}{\cchi^{n(|\F|-1)/2}}{\psi_{\F, 1}}^{4}\,.
\end{equation}
Since $\cchi^{n(|\F|-1)/2}$ is a quadratic character on $\F$, the same computation of Gauss as in the first claim yields that
\[
\Gauss{\F}{\cchi^{n(|\F|-1)/2}}{\psi_{\F, 1}}^{4} = |\F|^2 = |\F'|^{1/2}\,.
\]
The second claim follows by combining the previous three equations.

The third claim is immediate from the fact that $\llambda_{(n/2,\alpha)}$ is a quadratic character on $\F'$\,.
\end{proof}

Let us recall the following result of Shafarevich and Tate, as stated in~\cite[Lemma 8.3]{Ulmer2002}.
\begin{lemma}[Shafarevich--Tate] \label{lem:Shafarevich-Tate}
Let $\F_0$ be a finite field extension of $\F_p$, and $\F/\F_0$ be a quadratic extension.
Let $\psi = \psi_{\F, 1}$ be the standard nontrivial additive character on $\F$. Let $\chi$ be a nontrivial multiplicative character on $\F$ which is trivial upon restriction to $\F_0$. 
For any element $x\in(\F)^{\times}$ with $\Tr_{\F/\F_0}(x) = 0$, we have
 \[ \Gauss{\F}{\chi}{\psi} = -\chi(x)\, |\F_0|\,.\]
\end{lemma}

We  use Lemma~\ref{lem:Shafarevich-Tate} to prove the following:

\begin{lemma}\label{lemma:evaluate-gauss-sum}
Let $p \neq 2$ be an odd prime. 
Let $n \geq 2$ be a supersingular integer, and 
let $o'\in O'_n$ be an orbit with representative $(i, \alpha)\in S'_n$ such that $2i \neq n$.
Let $\nu_{i}$ be the smallest positive integer such that $p^{\nu_{i}} \equiv -1$ modulo $n/\gcd(n,i)$.
Then, 
\begin{equation} \label{eqn:Gauss-Sum-complicated}
\Gauor{o'} = (-1)^{\left(1 + \frac{i(p^{\nu_i} + 1)}{n}\right) \frac{|o'|\, [\F_r:\F_p]}{2\nu_i}} \llambda_{(i,\alpha)}(\alpha)^{-1} r^{|o'|/2}\,.
\end{equation}
In particular, if $[\F_{r}:\F_{p}]$ is a multiple of $4\nu_{i}$, then
\begin{equation}\label{eqn:Gauss-Sum-simple-2}
\Gauor{o'} = \llambda_{(i,\alpha)}(\alpha)^{-1}  r^{|o'|/2}\,.
\end{equation}
If $[\F_{r}: \F_{p}]$ is a multiple of $4\nu_{i}$ and $\alpha$ is an $n$\textsuperscript{th} power in $(\F')^\times$, then
\begin{equation} \label{eqn:Gauss-Sum-extra-simple}
\Gauor{o'} = r^{|o'|/2}\,.
\end{equation}
\end{lemma}

Before the proof, we remark that by construction, the exponent of $-1$ in \eqref{eqn:Gauss-Sum-complicated} is an integer.

\begin{proof}
Let $\F'$ denote the extension of $\F_r$ of degree $|o'|$.
By Definition~\ref{def:bold_g}, we have 
$\Gauor{o'} = \Gauss{\F'}{\llambda_{(i, \alpha)}}{\Psi_{(i, \alpha)}}$. We deduce from equation \eqref{eqn:additive-char-Gauss-sum}  in~\S\ref{ssec:gauss.sums.basic} that 
\begin{equation}\label{eq:expl.gauss.1}
    \Gauor{o'} = \llambda_{(i, \alpha)}(\alpha)^{-1}\Gauss{\F'}{\llambda_{(i, \alpha)}}{\psi_{\F', 1}}\,.
\end{equation}
Set $n' = n/\gcd(n,i)$. Recall that the character $\llambda_{(i, \alpha)} = \cchi^{i(r^{|o'|} -1)/n}$ has exact order $n'$. 
We now focus on providing an explicit expression for the Gauss sum $\Gauss{\F'}{\llambda_{(i, \alpha)}}{\psi_{\F', 1}}$\,.

Since $n'$ divides $n$ and $n$ is supersingular for $p$, we see that $n'$ is also supersingular for $p$. As in the statement of Lemma~\ref{lemma:evaluate-gauss-sum}, let $\nu_i$ denote the smallest positive integer such that $p^{\nu_i}\equiv -1\bmod{n'}$. Since $2i \neq n$, we have $n' > 2$. Hence, the order of $p$ modulo $n'$ is $o_p(n') = 2\nu_i$. 

Let $\F_0$ denote the extension of $\F_p$ of degree $\nu_i$ and let $\F$ denote its quadratic extension. We claim that $\F$ is a subextension of $\F'/\F_p$. Indeed, $[\F':\F_p] =  [\F_r:\F_p] |o'|$ is a multiple of $[\F_r:\F_p] o_r(n')$ and 
\[
[\F_r:\F_p] o_r(n') 
= \frac{[\F_r:\F_p]}{\gcd([\F_r:\F_p],o_p(n'))} o_p(n') = \frac{[\F_r:\F_p]}{\gcd([\F_r:\F_p],o_p(n'))} [\F:\F_{p}]
\]
is in turn an integer multiple of $[\F:\F_p]$.

By construction, $n'$ divides $|\F|-1$. So, $n$ divides $i(|\F|-1)$. 
In particular, we deduce that 
\[\llambda_{(i, \alpha)} = {\cchi|_{\F'}}^{i(|\F'|-1)/n} 
=({\cchi|_\F}\circ\norm_{\F'/\F})^{i(|\F|-1)/n}.\]
By the Hasse--Davenport relation ( \eqref{eqn:HasseDavenportRelation} in \S\ref{ssec:gauss.sums.basic}), we have 
\begin{equation}\label{eq:expl.gauss.2}
\Gauss{\F'}{\llambda_{(i, \alpha)}}{\psi_{\F', 1}} 
= \Gauss{\F'}{{\cchi|_{\F}}^{i(|\F|-1)/n} \circ\norm_{\F'/\F}}{\psi_{\F, 1}\circ\Tr_{\F'/\F}} 
= \Gauss{\F}{\cchi^{i(|\F|-1)/n}}{\psi_{\F, 1}}^{[\F':\F]}\,.
\end{equation}

Consider the multiplicative character $\chi=\cchi^{i(|\F|-1)/n}$ on $\F$. The character $\chi$ has exact order $n'$. In particular, the order of $\chi$ is greater than $2$.
Since $n'$ divides $p^{\nu_{i}}+1$, the restriction of $\chi$ to the quadratic subextension $\F_0$ of $\F$ is trivial.

Now, let $g$ be a generator of the cyclic group $\F^{\times}$. Set $x = g^{(p^{\nu_{i}} +1)/2}$. Since $|\F^{\times}|/|\F_{0}^{\times}| = p^{\nu_{i}} + 1$, we have $x \in \F^{\times} \smallsetminus \F_{0}^\times$ and $x^2 \in \F_{0}^{\times}$. So, $\Tr_{\F/\F_{0}}(x) = 0\,.$

With this choice of $x$, Lemma~\ref{lem:Shafarevich-Tate} gives $\Gauss{\F}{\chi}{\psi_{\F, 1}} = -\chi(x)  |\F|^{1/2}$\,. Moreover,
\[\chi(x) = \cchi\left(g^{\frac{p^{\nu_i}+1}{2}}\right)^{i(|\F|-1))/n}
= \cchi\left(g^{\frac{|\F|-1}{2}}\right)^{i(p^{\nu_i}+1)/n} 
= \cchi\left(-1 \right)^{i(p^{\nu_i}+1)/n} = (-1)^{i(p^{\nu_i}+1)/n}.\]
It follows that 
\begin{equation}\label{eq:expl.gauss.3}
    \Gauss{\F}{\chi}{\psi_{\F, 1}}= (-1)^{1+i(p^{\nu_i}+1)/n} |\F|^{1/2}.
\end{equation}
We now put \eqref{eq:expl.gauss.1}, \eqref{eq:expl.gauss.2}, and \eqref{eq:expl.gauss.3} together to deduce that 
\[\Gauor{o'} = \llambda_{(i, \alpha)}(\alpha)^{-1} (-1)^{[\F':\F]\, \left(1 + i(p^{\nu_i}+1)/n\right)}\,|\F'|^{1/2}\,.\]
Finally, we note that 
\[ [\F':\F] = 
\frac{[\F':\F_r][\F_r:\F_p]}{[\F:\F_p]} 
= \frac{|o'|\, [\F_r:\F_p]}{2{\nu_i}}.\]
This completes the proof of \eqref{eqn:Gauss-Sum-complicated}.

If $[\F_{r}:\F_{p}]$ is a multiple of $4\nu_{i}$, then 
\[\frac{|o'|\,[\F_r:\F_p]}{2\nu_i}\left(1 + \frac{i(p^{\nu_i}+1)}{n}\right)\]
is even and $\Gauor{o'} = \llambda_{(i, \alpha)}(\alpha)^{-1}|\F'|^{1/2}$.
Finally, if $\alpha\in\F_q^\times$ is a $n$\textsuperscript{th} power in $(\F')^\times$, we have $\llambda_{(i,\alpha)}(\alpha)=1$ because the order of $\llambda_{(i,\alpha)}$ divides $n$.
\end{proof}

\subsection[]{Denominators of $\gp$-adic valuation of Gauss sums}\label{ss:p.adic.val.Gauss}
 
We work with the same notation as in the previous subsection. 
Recall that we have fixed a prime ideal $\mathfrak{p}$ of $\overline{\Q}$ above $p$.
This choice allowed us to define the Teichm\"uller character $\cchi:\overline{\F_p}^\times\to\Qbar^\times$, in~\S\ref{sec:characters}.
Recall also that $\pval$ denotes the valuation on $\Qbar$ associated to $\gp$, normalised so that $\pval(r)=1$. Throughout this section, given $x \in \mathbb R$, we let $\{x\}$ denote the fractional part of $x$. 

Let $n\geq 2$ be an integer coprime to $p$. For any orbit $o'\in O'_n$, the $\gp$-adic valuation of the Gauss sum $\Gauor{o'}$ is a nonnegative rational number.

For any orbit $o'\in O'_n$, we write 
${\pval(\Gauor{o'})}/{|o'|}$ as a reduced fraction: 
\[ \frac{\pval(\Gauor{o'})}{|o'|}
= \frac{\Num(o')}{\Den(o')},\]
for integers $\Num(o')\geq 0$, $\Den(o')\geq 1$ such that 
$\gcd(\Num(o'),\Den(o'))=1$. 

Our goal in this section is to control  $\Den(o')$ under various hypotheses on $p, r$, and $n$. We begin with an immediate consequence of Lemmas~\ref{lemma:gauss-sum-quadratic}~and~\ref{lemma:evaluate-gauss-sum}.

\begin{lemma}\label{lemma:num.den.supersing}
Suppose $n\geq 2$ is supersingular for $p$. Then, for all $o'\in O'_n$, $\Num(o')/\Den(o')=1/2$.
\end{lemma}

When $n$ is not supersingular for $p$, we  need to do more work to control $\Den(o')$. Our main tool is the following lemma, 
which gives an explicit formula for $\pval(\Gauor{o'})/|o'|$.

 For $x \in \R$, let

\begin{lemma}\label{lemm:gauss.stickelberger} 
Let $n\geq 2$ be an integer coprime to $p$. Let $o'\in O'_n$ be an orbit and pick a representative $(i,\alpha)\in S'_n$ of $o'$. Let $\mu = [\F_r:\F_p] \,|o'|$. 
Write $i\in\Z$ for any lift of $i\in\Z/n\Z$ to $\Z$.
Then,
\begin{equation}\label{eq:gauss.stickelberger}
    \frac{\pval(\Gauor{o'})}{|o'|} 
 = \frac{1}{\mu} \sum_{k=0}^{\mu-1} \left\{ \frac{-i p^k}{n}\right\}\,,
\end{equation}
 where $\{x\}$ denote the fractional part of $x\in\R$. 
\end{lemma}

The proof of Lemma~\ref{lemm:gauss.stickelberger} relies on a version of Stickelberger's Theorem. We use Lemma~6.14 from \cite{Cyclo}, which we restate here in our notation 
for the reader's convenience. The extra factor $[\F_r:\F_p]$ appearing in our statement comes from our different choice of normalization for $\pval$.

\begin{theorem}[Stickelberger's Theorem]\label{theo:stickelberger}
 Let $\F$ be a finite extension of $\F_p$ with degree $\mu=[\F:\F_p]$.  
Fix an integer $s$ such that $0<s < p^\mu-1$. 
For any nontrivial additive character $\psi$ on $\F$, we have
\[
\pval\left(\Gauss{\F}{(\cchi|_{\F^\times})^{-s}}{\psi}\right)  = \frac{1}{[\F_r:\F_p]}\sum_{k=0}^{\mu-1} \left\{ \frac{s p^k}{p^\mu - 1} \right\}\,,
\]
where $\{x\}$ denote the fractional part of $x\in\R$. Here, as above, $\cchi$ denotes the Teichm\"uller character.
\end{theorem}

\begin{proof}[Proof of Lemma \ref{lemm:gauss.stickelberger}] 
Let $(i,\alpha)\in S'_n$ be a representative of the orbit $o'\in O'_n$. 
Let $\F'$ denote the finite field extension of $\F_r$ of degree $|o'|$.
By Definition~\ref{def:bold_g} in \S\ref{sec:gausssums.orbits}, 
\[ \Gauor{o'} = \Gauss{\F'}{\llambda_{(i, \alpha)}}{\Psi_{(i, \alpha)}} = \Gauss{\F'}{\left(\cchi|_{(\F')^\times}\right)^{i(r^{|o'|}-1)/n}}{\Psi_{(i, \alpha)}} \,.\]

Since $\alpha\neq 0$, the additive character $\Psi_{(i,\alpha)}$ on $\F'$ is nontrivial.

Note that $[\F':\F_p] = |o'|\cdot [\F_r:\F_p] = \mu$ and $r^{|o'|} = p^{\mu}$. Moreover, $r^{|o'|}$ acts trivially on $(\Z/n\Z)^{\times}\,,$ so $i(r^{|o'|}-1)/n$ is an integer. Applying Stickelberger's Theorem (Theorem~\ref{theo:stickelberger}) gives
\[\frac{\pval(\Gauor{o'})}{|o'|}
=  \frac{1}{[\F_r:\F_p]\, |o'|} \sum_{k=0}^{\mu-1} \left\{\frac{-i (r^{|o'|} - 1)}{n}\frac{p^k}{p^\mu-1}\right\}
=  \frac{1}{\mu} \sum_{k=0}^{\mu-1} \left\{\frac{-ip^k}{n}\right\}\,.\]
\end{proof}

\begin{corollary}\label{cor:Den_lower_bound}
Let $n\geq 1$ be a prime-to-$p$ integer. For any orbit $o'\in O'_n$, we have 
\[\frac{1}{n} \leq \frac{\pval(\Gauor{o'})}{|o'|} = \frac{\Num(o')}{\Den(o')}\leq 1-\frac{1}{n}\,.\]
In particular, $1 \leq \Num(o') < \Den(o')\,.$
\end{corollary}

\begin{proof}
Let $(i, \alpha)\in S'_n$ be a representative of $o'$. We lift $i\in\Z/n\Z\smallsetminus\{0\}$ to $i\in\Z$.

In the notation of Lemma \ref{lemm:gauss.stickelberger}, for any $k\in\{0, \dots, \mu-1\}$, 
we have $1/n \leq \left\{-ip^k/n \right\} \leq (n-1)/n$ because $i$ is not a multiple of $n$, and $p$ is relatively prime to $n$. 
To conclude, sum these inequalities over all $k$ from $0$ to $\mu - 1$ and apply \eqref{eq:gauss.stickelberger} from Lemma~\ref{lemm:gauss.stickelberger}.
\end{proof}

 We now prove a more precise estimate on the denominator of $\pval(\Gauor{o'})/|o'|$. 
 The following may be viewed as a bound on the denominators of slopes of the $\gp$-adic Newton polygon of the $L$-function of the projective curve defined over $\F_r$ by $y^n = t^q-t$.
 
 \begin{proposition}\label{prop:DenDivides_nopn}
 Let $n\geq 2$ be an integer coprime to $p$. 
 Let $o'\in O'_n$ be an orbit with representative $(i, \alpha)\in S'_n$.
 Then,
 \[ \Den(o')\text{ divides } \frac{n}{\gcd(n,i)} o_p\left(\frac{n}{\gcd(n,i)}\right)   \] 
 In particular, $\Den(o')$ divides $n\,o_p(n)$.
 \end{proposition}

 \begin{proof}
 In this proof, we use the same notation as in that of Lemma~\ref{lemm:gauss.stickelberger}. 
 With $\mu = |o'| [\F_r:\F_p]$, we know from Lemma~\ref{lemm:gauss.stickelberger} that
\begin{equation}\label{eq:gauss.pval.eq1}
\frac{\Num(o')}{\Den(o')}
= \frac{\pval(\Gauor{o'})}{|o'|} 
= \frac{1}{\mu} \sum_{k=0}^{\mu-1} \left\{\frac{-i p^k}{n}\right\}\,.
\end{equation}
 
 To lighten notation, set $\kappa = o_{p}(n/\gcd(n,i)).$ 
 We remark that $\kappa$ divides $o_{p}(n)$, which divides $o_{r}(n) [\F_{r}:\F_{p}],$ which in turn divides $|o'| [\F_{r}:\F_{p}]\,.$ In particular, $\kappa$ divides $\mu$.

 In the sum on the right-hand side of~\eqref{eq:gauss.pval.eq1},  write the Euclidean division of any index $k\in\{0,\dots, \mu-1\}$ by $\kappa$ as $k = x\kappa + y$ with $y\in\{0, \dots, \kappa-1\}$ and $x\in\{0, \dots, \mu/\kappa\}$. One may then rewrite the sum in the form
 \[\frac{1}{\mu} \sum_{k=0}^{\mu-1} \left\{\frac{-i p^k}{n}\right\} 
 = \frac{1}{\mu} \sum_{y=0}^{\kappa-1} \sum_{x=0}^{\mu/\kappa-1} \left\{\frac{-i p^y  p^{x\kappa}}{n}\right\}\,.\]
Since $\kappa = o_{p}(n/\gcd(n,i))$, we have $ip^{\kappa} \equiv i \pmod{n}$, so the inner sums (over $x$) are equal as $y$ varies. More precisely, 
\[
\sum_{x=0}^{\mu/\kappa-1} \left\{\frac{-i p^y  p^{x\kappa}}{n}\right\}
 = \sum_{x=0}^{\mu/\kappa-1} \left\{\frac{-i p^y}{n}\right\} = \frac{\mu}{\kappa}  \left\{\frac{-i p^y}{n}\right\}\,.
\]
Summing this equality over all $y\in\{0, \dots, \kappa-1\}$, we deduce that 
\begin{align}\label{eqn:Stick-proof}
\frac{\Num(o')}{\Den(o')}
= \frac{1}{\mu} \sum_{y=0}^{\kappa-1} \frac{\mu}{\kappa}\left\{\frac{-i p^y}{n}\right\}
= \frac{1}{\kappa} \sum_{y=0}^{\kappa-1}\left\{\frac{-i p^y}{n}\right\}\,.
\end{align}

Each term $\{-ip^{y}/n\}$ in the right-most sum in \eqref{eqn:Stick-proof} is a rational number with denominator $n/\gcd(n,i p^y) = n/\gcd(n,i)$. So, the right-most sum in \eqref{eqn:Stick-proof} is a rational number with denominator dividing $n/\gcd(n,i)$. 
After division by $\kappa=o_{p}(n/\gcd(n,i))$, we conclude that ${\Den(o')}$ divides $ o_{p}(n/\gcd(n,i))\cdot n/\gcd(n,i)$.

The order of $p$ modulo any divisor of $n$ divides the order of $p$ modulo $n$, so $o_{p}(n/\gcd(n,i))$ divides $o_p(n)$. This proves the second assertion of the proposition.
 \end{proof} 
 
 \subsection[]{Explicit $\gp$-adic valuations of $\omega(o)$}
 We now come back to the general setting of this paper. 
 We fix a finite extension $\F_r$ of $\F_p$. For any pair $(a,b)$ of relatively prime integers which are both coprime to $p$, and for any power $q$ of $p$, we consider the Jacobian $J$ of the curve $C$ over $K=\F_r(t)$.
 
 As was shown in Section \ref{sec:ExplicitL}, the $L$-function of $J$ involves certain character sums $\oomega(o)$, indexed by orbits $o\in O =O_{a,b,q,r}$.
 By Definition~\ref{def:oomega}, we have 
 \[\forall o\in O, \qquad 
 \oomega(o) 
 = \Gauor{\pi_a(o)}^{|o|/|\pi_a(o)|}\Gauor{\pi_b(o)}^{|o|/|\pi_b(o)|},\]
 where $\pi_a:O\to O'_a$ and $\pi_b:O\to O'_b$ are the maps introduced in Section \ref{sec:orbits}.
 For any orbit $o\in O$, in the notation introduced in \S\ref{ss:p.adic.val.Gauss}, we thus have
 \begin{equation}\label{eq:omega.pval}
 \frac{\pval(\oomega(o))}{|o|} 
 = \frac{\Num(\pi_a(o))}{\Den(\pi_a(o))} + \frac{\Num(\pi_b(o))}{\Den(\pi_b(o))}. 
 \end{equation}
In the upcoming subsection, it will be useful to know of situations in which $\pval(\oomega(o))\neq |o|$.

From the previous subsection, we deduce the following:

\begin{lemma} \label{lem:bad-Gauss-sums}
Let $a,b, q, r$ be as above.
Assume that one of the following holds:
\begin{enumerate}[(1)]
    \item $a o_p(a)$ and $b o_p(b)$ are relatively prime;
    \item  $a o_p(a)$ is odd, and $b$ is supersingular for $p$; or 
    \item  $a$ is supersingular for $p$, and $b o_p(b)$ is odd.
\end{enumerate}
Then, for any orbit $o\in O=O_{a,b,q,r}$, we have $\pval(\oomega(o))\neq |o|$.
\end{lemma}

\begin{proof}
Let $o\in O$ be an orbit.
If condition {\it (1)} is satisfied, then $\gcd(\Den(\pi_a(o)), \Den(\pi_b(o))=1$ by Proposition~\ref{prop:DenDivides_nopn}.
Hence, $\Den(\pi_a(o))\neq \Den(\pi_b(o))$ unless both $\Den(\pi_a(o)) = 1$ and $\Den(\pi_b(o)) = 1$. This situation does not occur, by Corollary~\ref{cor:Den_lower_bound}.

 If $a$ is supersingular for $p$, then $\Den(\pi_a(o))=2$ by Lemma~\ref{lemma:num.den.supersing}. 
 By Proposition~\ref{prop:DenDivides_nopn}, $\Den(\pi_b(o))$ divides $b o_r(b)$. Hence, if $b o_{r}(b)$ is odd,  so is $\Den(\pi_b(o))$.
 In particular, if {\it(2)} is satisfied, then $\Den(\pi_a(o))\neq\Den(\pi_b(o))$. 
 The case where {\it(3)} holds is treated in a similar way, by switching the roles of $a$ and $b$.
 
 In all three situations, we have shown that $\Den(\pi_a(o))\neq \Den(\pi_b(o))$. Since 
 two reduced fractions with different denominators cannot sum to $1$, the result now immediately follows from \eqref{eq:omega.pval}. 
\end{proof}

Lemma~\ref{lem:infinitude-of-bad-omega} shows that there are infinitely many choices for $a$ and $b$ satisfying each of the hypotheses of Lemma~\ref{lem:bad-Gauss-sums}.

\begin{lemma} \label{lem:infinitude-of-bad-omega}
For any fixed $p$, each of the following conditions: 
\begin{enumerate}[(1)]
    \item $a o_p(a)$ and $b o_p(b)$ are relatively prime;
    \item  $a o_p(a)$ is odd, and $b$ is supersingular for $p$; 
    \item  $a$ is supersingular for $p$, and $b o_p(b)$ is odd.
\end{enumerate}
is satisfied for infinitely many $a$ and $b$. Moreover, each condition is satisfied for infinitely many \emph{primes} $a$ and $b$.
\end{lemma}

\begin{proof}

First, consider condition {\it(2)}. 
If $k$ is an odd positive integer and 
$a$ is any odd divisor of $p^{k} - 1$, then $o_{p}(a)$ divides $k$. So, $a o_{p}(a)$ is odd too. 
We claim that there are infinitely many such integers $a$. 
Indeed, for any odd integer $k$, the integer $a = (p^k - 1)/(p-1)$ is odd.
On the other hand, there are infinitely many supersingular prime numbers $b$, all but finitely many of which are coprime to any particular choice of $a$.
Condition {\it(3)} can be satisfied by exchanging the role of \(a\) and \(b\).


We now consider condition {\it(1)}. Choose any odd prime \(k \geq 3\) so that \(p \not\equiv 1 \pmod{k}\) and take \(a = (p^{k} - 1)/(p-1) \). Choose any odd prime $\ell$ which is relatively prime to both $k$ and $a$ and which does not divide $o_{p}(k)$. There are infinitely many such $\ell$. If we set $b = (p^{\ell} - 1)/(p-1)$, then $b \not\equiv 0 \pmod{k}$. We have \(ao_p(a) = ak\) and \(bo_{p}(b) = b\ell\). By construction, \(\gcd(a,\ell) = \gcd(k,\ell) = 1\), and \(\gcd(b,k) = 1\). Finally,
\[
\gcd(a,b) = \frac{\gcd(p^k - 1, p^{\ell} - 1)}{p-1} = \frac{p^{\gcd(k,\ell)} - 1}{p-1} = \frac{p-1}{p-1} = 1\,.
\]

Modifying these constructions slightly and still keeping $p$ fixed, we may arrange that $a$ and $b$ are both primes, as we now explain. 

Let $T$ be the set of primes $k$ so that $p^{k} - 1$ is a product of primes dividing $p-1$. We first show that $T$ is finite. By work of Siegel, given any set $S$ of primes, the set of solutions to $x - y = 1$ in $S$-units $x$ and $y$ is finite. Let $S$ be the set of primes dividing $p(p-1)$. Then, for each $k \in T$, the pair $x = p^k$, $y = p^{k} - 1$, is a solution to the $S$-unit equation. Hence, by Siegel's Theorem, $T$ is finite. In particular, if we choose distinct odd primes \(k,  \ell \notin T\) in the preceding constructions, we may choose $a$ and $b$ to be odd prime factors of $p^k-1$ and $p^{\ell} - 1$ respectively, and which do not divide $p-1$. 
We conclude that $a o_{p}(a)$ and $b o_{p}(b)$ will still be relatively prime odd integers.

A similar argument shows that there are infinitely many supersingular primes $b$ for $p$. So, conditions {\it(2)} and {\it(3)} are also satisfied for infinitely many primes $a$ and $b$.
\end{proof}

 \subsection[]{Rank $0$}\label{sec:rank.0}
 It follows from \eqref{eq:rank.formula} that
\[\rank J(K) = \ord_{T=r^{-1}}L(J, T) \leq \left|\big\{ o\in O : \pval(\oomega(o)) = |o|\big\}\right|.\]
Hence, to show that the rank is ``small'' it suffices to give conditions on $a,b,q$ that ensure that ``many'' orbits $o\in O$ satisfy $\pval(\oomega(o)) \neq |o|$.
We prove:

\begin{customthm}{\ref{thm:rankzero}}
Suppose that the pair $(a,b)$ satisfies one of the following:
\begin{enumerate}[(1)]
    \item $a o_p(a)$ and $b o_p(b)$ are relatively prime;
    \item  $a o_p(a)$ is odd, and $b$ is supersingular for $p$; or  
    \item  $a$ is supersingular for $p$, and $b o_p(b)$ is odd.
\end{enumerate}
Then, for any power $q$ of $p$, we have 
$\ord_{T=r^{-1}} L(J,T) = \rank J(K) = 0$.
\end{customthm}

\begin{proof}
The conditions here are the same as in Lemma \ref{lem:bad-Gauss-sums}. 
That Lemma asserts that, for all orbits $o \in O= O_{r,a,b,q}$, the $\gp$-adic valuation of $\oomega(o)$ does not match that of $r^{|o|}$ (which equals $|o|$).

The assertion is then immediate from \eqref{eq:rank.formula}. 
\end{proof}

\begin{example}
Let $\F_r=\F_{67^n}$ for some $n\geq 1$. For $p=67$, the pair $a=5$ and $b=7$ satisfies condition {\it (3)} of Theorem~\ref{thm:rankzero}. 
So, if $q$ is any power of $67$, the Jacobian $J=J_{a,b,q}$ satisfies $\rank J(\F_r(t))=0$.
\end{example}

For a fixed odd prime $p$, the set of parameters $a,b$ for which the conditions of Theorem~\ref{thm:rankzero} hold is infinite, as shown in Lemma~\ref{lem:infinitude-of-bad-omega}.

\begin{remark} One can provide a second proof of the BSD conjecture (Theorem \ref{thm:bsd}) in the case that $L(J,r^{-1}) \neq 0$, as follows.
By a theorem of Tate \cite{Tate1965}, one has 
\[0 \leq \rank J(K) \leq \ord_{T=r^{-1}}L(J, T).\]
(This essentially follows from injectivity of the cycle class map.) If the parameters $a,b,q$ are such that $L(J,T)$ does not vanish at $T=r^{-1}$, we deduce from the above that $\rank J(K) = \ord_{T=r^{-1}}L(J, T) =0$. 
In other words, the ``weak BSD conjecture'' holds for $J$.
\end{remark}

\subsection{Large ranks}\label{sec:rank.large}
We now provide a sufficient condition on $a,b$ and $q$ for the rank of $J(K)$ to be ``large''. We actually prove a more precise result, estimating the rank of $J(K)$ under certain assumptions. First, we prove a lemma to calculate $\oomega(o)$ for $o \in O$.

\begin{lemma}\label{lem:oomega-value}
Assume that $p \neq 2$ is an odd prime.
Let $a$ and $b$ be relatively prime positive integers which are both supersingular for $p$. Let $\nu_a, \nu_b\geq 1$ be the least positive integers such that $p^{\nu_{a}} \equiv -1 \pmod{a}$ and $p^{\nu_{b}} \equiv -1 \pmod{b}$.
Suppose also that $[\F_{r}:\F_{p}]$ is a multiple of both $4\nu_{a}$ and $4\nu_{b}$. 

If $(i,j,\alpha)$ is any representative of the orbit $o\in O$, then
\begin{equation}
    \oomega(o) = \llambda_{(i,\alpha)}(\alpha)^{-1} \llambda_{(j,\alpha)}(\alpha)^{-1} r^{|o|}\,.
\end{equation}
In particular, $\oomega(o) = r^{|o|}$ if and only if $\alpha\in \F_{q}$ is an $(ab)$\textsuperscript{th} power in $\F_r(\alpha)$ for any representative $(i,j,\alpha)$ of $o$ (equivalently for all representatives $(i,j,\alpha)$ of $o$).
\end{lemma}

\begin{proof}
Since $4\nu_{a}$ divides $[\F_{r}:\F_{p}]$ and $p^{\nu_{a}} \equiv -1 \pmod{a}$, we see that $r \equiv 1 \pmod{a}$. Hence $\langle r \rangle$ acts trivially by multiplication on $\Z/a\Z\smallsetminus\{0\}$.
Similarly, $r \equiv 1 \pmod{b}$, so $\langle r \rangle$ acts trivially by multiplication on $\Z/b\Z\smallsetminus\{0\}$.
Hence, the orbit $o$ 
is of the form $\{(i,j,\alpha(o)^{1/r^t}) : t \in \Z\}$ for some $(i,j,\alpha)\in S$ depending on $o$. 
We then have $|o| = |\pi_{a}(o)| = |\pi_{b}(o)|$.

In particular, 
\[
\oomega(o) = \Gauor{\pi_{a}(o)} \Gauor{\pi_{b}(o)}.
\]
We may now apply Lemma~\ref{lemma:gauss-sum-quadratic} (resp.\ Lemma~\ref{lemma:evaluate-gauss-sum}) to compute $\Gauor{\pi_{a}(o)}$ 
when $2i=n$ (resp.\ $2i \neq n$). 
Since $4\nu_{a}$ divides  $[\F_{r}:\F_{p}]$ and the $\nu_{i}$'s appearing in Lemmas~\ref{lemma:gauss-sum-quadratic}~and~\ref{lemma:evaluate-gauss-sum} applied to $\Gauor{\pi_{a}(o)}$ are divisors of $\nu_{a}$, we have $4\nu_{i}|[\F_{r}:\F_{p}]$. So, equations \eqref{eqn:quadratic-Gauss-Sum-simple-2} and \eqref{eqn:Gauss-Sum-simple-2} hold. We find that $\Gauor{\pi_{a}(o)} = \llambda_{(i,\alpha)}(\alpha)^{-1}  r^{|o'|/2}$. Computing $\Gauor{\pi_{b}(o)}$ in the same way yields that 
\[
\oomega(o) = \Gauor{\pi_{a}(o)} \Gauor{\pi_{b}(o)} = \llambda_{(i,\alpha)}(\alpha)^{-1} \llambda_{(j,\alpha)}(\alpha)^{-1} r^{|o|}\,.
\]
Now, $\llambda_{(i,\alpha)}$ and $\llambda_{(j,\alpha)}$ are characters of relatively prime orders $a$ and $b$, so 
$\llambda_{(i,\alpha)}(\alpha)^{-1} \llambda_{(j,\alpha)}(\alpha)^{-1} = 1$ if and only if both $\llambda_{(i,\alpha)}(\alpha) = 1$ and $\llambda_{(j,\alpha)}(\alpha) = 1$.

Since $|\pi_{a}(o)|$ and $|\pi_{b}(o)|$ are both equal to the size of the orbit of $r$ acting on $\F_{q}^{\times}$, the extensions of $\F_{r}$ of degree $|\pi_{a}(o)|$ and $|\pi_{b}(o)|$ are both $\F_{r}(\alpha)$. This is the extension over which both $\llambda_{(i,\alpha)}$ and $\llambda_{(j,\alpha)}$ are defined. To conclude, observe that $\llambda_{(i,\alpha)}(\alpha) = \llambda_{(j,\alpha)}(\alpha) = 1$ if and only if $\alpha$ is an $(ab)$\textsuperscript{th} power in $\F_r(\alpha)$.
\end{proof}

\begin{customthm}{\ref{thm:analytic-rank-lower-bounds}}
Let $p \neq 2$ be an odd prime.
Let $a$ and $b$ be relatively prime positive integers which are both supersingular for $p$. Let $\nu_a, \nu_b\geq 1$ be the least positive integers such that $p^{\nu_{a}} \equiv -1 \pmod{a}$ and $p^{\nu_{b}} \equiv -1 \pmod{b}$.
Suppose also that $[\F_{r}:\F_{p}]$ is a multiple of both $4\nu_{a}$ and $4\nu_{b}$.

Then, we have
\[
(a-1)(b-1) \left \lceil \frac{1}{\log_p(q)}\left( \frac{q-1}{ab}  - \frac{p\sqrt{q} - 1}{p-1}\right)\right \rceil \leq \rank J(K)\,.
\]
\end{customthm}

\begin{proof}[Proof of Theorem~\ref{thm:analytic-rank-lower-bounds}]

Combining Lemma~\ref{lem:rank-formula-orbits} and Lemma~\ref{lem:oomega-value} gives that $\rank J(K)$ is equal to the number of orbits $o \in O$ such that a representative $(i,j, \alpha)$ satisfies the property that $\alpha$ is an $(ab)$\textsuperscript{th} power in $\F_r(\alpha)$.

We first bound the number of $\alpha \in \F_{q}^{\times}$ such that $\alpha$ is an $(ab)$\textsuperscript{th} power in $\F_p(\alpha)$. 
We remark that~$\F_{q}^{\times}$ contains at least $(q-1)/ab$ distinct values which are $(ab)$\textsuperscript{th} powers. 
Indeed the image of the map $x\in\F_q^\times\mapsto x^{ab}\in\F_q^\times$ has order $|\F_q^\times|/\gcd(ab,|\F_q^\times|)=(q-1)/\gcd(ab,q-1)$. Note that $\gcd(ab,q-1)\leq ab$. 
Now, at most $q^{1/2} + q^{1/2}p^{-1} + \cdots + 1 = (p\sqrt{q} - 1)(p-1)$ elements of $\F_{q}$ lie in a proper subfield, since each proper subfield has order a distinct power of $p$ which is at most $\sqrt{q}.$ Hence, there are at least 
\[
\frac{q-1}{ab}  - \frac{p\sqrt{q} - 1}{p-1}
\]
distinct values $\alpha \in \F_{q}$ such that $\F_{p}(\alpha) = \F_{q}$ and $\alpha$ is an $(ab)$\textsuperscript{th} power in $\F_{p}(\alpha)$. Each orbit of $\langle r \rangle$ on $\F_{q}^{\times}$ contains at most $[\F_q:\F_p] = \log_{p}(q)$ elements and so contains at most $\log_{p}(q)$ many such $\alpha$.

Since $\langle r \rangle$ acts trivially on both $\Z/a\Z$ and $\Z/b\Z$ under the hypotheses, the number of orbits $o \in O$ such that a representative $(i,j, \alpha)$ satisfies the property that $\alpha$ is an $(ab)$\textsuperscript{th} power in $\F_r(\alpha)$ is at least
\[
(a-1)(b-1) \left \lceil \frac{1}{\log_p(q)}\left( \frac{q-1}{ab}  - \frac{p\sqrt{q} - 1}{p-1}\right)\right \rceil\,. 
\]
\end{proof}

\begin{theorem} \label{thm:rank.upperbound.met}
Let $p \neq 2$ be an odd prime.
Let $a$ and $b$ be relatively prime positive integers which are both supersingular for $p$. Let $\nu_a, \nu_b\geq 1$ be the least positive integers such that $p^{\nu_{a}} \equiv -1 \pmod{a}$ and $p^{\nu_{b}} \equiv -1 \pmod{b}$.
Suppose that $[\F_{r}:\F_{p}]$ is a multiple of $4\nu_{a}, 4\nu_{b},$ and $ab(q-1)$. Then,

\[ \rank J(K) =  (a-1)(b-1)(q-1) = 2g(q-1)\,.\] 
In other words, the upper bound in Theorem~\ref{thm:rank.upperbound} is met.
\end{theorem}

\begin{proof}
Under these assumptions, the product $ab(q-1)$ divides $r-1$, hence $\langle r \rangle$ acts trivially on $S$. Hence each orbit $o \in O$ has $|o| = 1$. Moreover, each $\alpha \in \F_{q}$ is an $(ab)$\textsuperscript{th} power in $\F_{r}$ (and therefore also in $\F_{r}(\alpha)$.)
Then, Lemma~\ref{lem:rank-formula-orbits} and Lemma~\ref{lem:oomega-value} together imply
\[
\rank J(K) = |O| = (a-1)(b-1)(q-1)\,.
\]
\end{proof}

\begin{remark}\label{rem:hypotheses-of-rank-lower-bounds}[Hypotheses of Theorems~\ref{thm:analytic-rank-lower-bounds}~and~\ref{thm:rank.upperbound.met}] 
For any fixed $p$, there are infinitely many choices of $a,b,r$ satisfying the hypotheses of Theorem~\ref{thm:analytic-rank-lower-bounds} and Theorem~\ref{thm:rank.upperbound.met}, as we now explain.

For any choice of $a$ and $b$, a positive density of primes $p$ satisfy $p \equiv -1 \pmod{ab}$. In that case we may take $\nu_{a} = \nu_{b} = 1$. 
Let $\F$ be the smallest extension of $\F_{p}$ such that $4$ divides $[\F:\F_{p}]$\,. The hypotheses of Theorem~\ref{thm:analytic-rank-lower-bounds} hold whenever $\F_{r} \supset \F$. Let $t$ be the order of $p$ in $\mathbb Z/ab(q-1) \mathbb Z$. 
Let $\F'$ be the smallest extension of of $\F_{p}$ such that both $4$ and $t$ divide $[\F':\F_{p}]$\,.
The hypotheses of Theorem~\ref{thm:rank.upperbound.met} are satisfied whenever $\F_{r} \supset \F'$\,.

In fact, if $a$ and $b$ are prime, $a$ and $b$ are supersingular for $p$ whenever $p$ has even order in both $(\mathbb Z/a\mathbb Z)^{\times}$ and $(\mathbb Z/b\mathbb Z)^{\times}$. Again, Theorem~\ref{thm:analytic-rank-lower-bounds} holds whenever $\F_{r}$ contains an appropriate finite extension of $\F_{p}$. The same is true for Theorem~\ref{thm:rank.upperbound.met}.

\end{remark}

\begin{remark}
Theorem \ref{thm:analytic-rank-lower-bounds} implies that when both $a$ and $b$ are supersingular for $p$ and $[\F_{r}: \F_{p}]$ is a fixed multiple of some number depending only on $a, b,$ and $p$, the analytic rank of $J$ is unbounded as $q \to \infty$.  This means that if we take $a$ and $b$ to be distinct primes, the Jacobians of the curves $y^b+x^a=t^q-t$ as $q$ varies give a family of simple abelian varieties of dimension $(a-1)(b-1)/2$ which satisfy BSD and which have unbounded algebraic and analytic rank. The dimension can be made arbitrarily large by increasing $a$ and $b$.

\end{remark}

\section{Size of the special value} 
\label{sec:specialvalue}

Recall that the special value $L^\ast(J)$ is defined as
\[L^\ast(J) := \left. \frac{L(J, T)}{(1-rT)^{v}}\right|_{T=r^{-1}}, \quad\text{where } v=\ord_{T=r^{-1}} L(J,T).\]
As discussed in Section~\ref{sec:bsd_conj_for_j}, the Riemann Hypothesis for $L(J, T)$ implies that $L^\ast(J)$ is a positive rational number.
The goal of this section is to prove the following estimate on $L^\ast(J)$:

\begin{customthm}{\ref{thm:specialvalue}}
For fixed $a,b$ as above, as $q\to\infty$ through powers of $p$, we have
\[\frac{\log L^\ast(J)}{\log H(J)} = o(1)\,, \]
where the implicit constants depend only on $a,b$ and $p$.
\end{customthm}

Throughout this section, we will use Vinogradov's asymptotic notation. Namely, for two functions $f, g$ of a variable $x$ on $[0, \infty)$, we use $f(x)\ll_a g(x)$ to mean that there is a constant $C>0$ (depending at most on the mentioned parameter(s) $a$) such that $|f(x)|\leq C g(x)$ for $x\to\infty$.

\subsection{Preliminary estimates}
The proof of Theorem~\ref{thm:specialvalue} requires two preliminary estimates that we now state.

We choose, once and for all, an algebraic closure $\Qbar$ of $\Q$. 
We write $\log:\C\to\C$ for the 
branch of the complex logarithm such that the imaginary part of $\log z$ belongs to $(-\pi, \pi]$ for all $z\in\C$.
For a given $\theta\in \frac{1}{2}\Z_{\geq 0}$, 
an algebraic integer will be called a \emph{Weil integer of size $p^\theta$} if its absolute value in any complex embedding of $\Qbar$ is $p^\theta$. (These are sometimes called Weil integers of weight $2\theta$.)

\begin{theorem}
\label{theorem:diophantine.estimates}
Let $p$ be a prime number, and $\theta\in \frac{1}{2}\Z_{\geq 0}$.
Let $z\in\Qbar$ be a Weil integer of size $p^\theta$, and $\zeta\in\Qbar$ be a root of unity.
For any integer $L\neq 0$, either $\zeta \cdot (zp^{-\theta})^L=1$ or, in any complex embedding $|\cdot|$ of $\Qbar$ in $\C$, we have
\begin{equation}
    \log\left|1-\zeta \cdot (zp^{-\theta})^L\right|
    \geq -c_0 -c_1 \log |L|,
\end{equation}
where $c_0, c_1>0$ are effective constants depending at most on $p$, $\theta$, the degree of $z$ over $\Q$, and the (multipicative) order of $\zeta$.
\end{theorem}
 We refer the reader to \cite[Thm 11.6]{GriffonUlmer} for a proof of Theorem~\ref{theorem:diophantine.estimates}. The main ingredient in the proof is a lower bound for linear forms in logarithms of algebraic numbers due to Baker--W\"ustholz in~\cite{BakerWustholz93}.

We  also need some estimates on the orbits in $O$. 
As before, $p$ is a prime number and $r$ is a fixed power of $p$.
For any relatively prime integers $a,b$ which are coprime to $p$, and for any power $q$ of $p$, we let 
$S:=(\Z/a\Z)\smallsetminus\{0\}\times(\Z/b\Z)\smallsetminus\{0\}\times\F_q^\times$. 
As in \S\ref{sec:orbits}, let $O$ denote the set of orbits for the action of~$\langle r\rangle$ on $S$.
\begin{lemma}\label{lemma:estimates.orbits}
For fixed $a,b$ as above, the following bounds hold as $q\to\infty$ through powers of $p$.
\begin{enumerate}[(1)]
\item $\sum_{o\in O} |o| = |S| = (a-1)(b-1)(q-1) \ll q$,

\item $\sum_{o\in O}1 =|O| \ll q/\log q$,
\item $\sum_{o\in O}\log |o| \ll q \log\log q/\log q$.
\end{enumerate}
The implied constants  depend at most on the product $ab$.
\end{lemma}

\begin{proof} As defined in Section~\ref{sec:orbits}, the set $S$ is a subset of $S'_{ab} = (\Z/ab\Z)\smallsetminus\{0\} \times \F_q^\times$. 
Hence $O$ may be viewed as a subset of the set $O'_{ab}$ of orbits for the action of $\langle r \rangle$ on $S'_{ab}$. Lemma~11.4.1 of \cite{GriffonUlmer} directly gives the required bounds.
\end{proof}

\subsection{Size of the special value}
For any $a,b,q$ as above, 
 for any orbit $o\in O$, recall that we have defined
\[\oomega(o) = \Gauor{\pi_a(o)}^{\nu_a(o)} \Gauor{\pi_b(o)}^{\nu_b(o)}.\]
Let $O_0$ denote the set of orbits $o\in O$ such that $\oomega(o) = r^{|o|}$,
and $O_\ast:= O\smallsetminus O_0$ denote its complement.
We require the following special case of Theorem \ref{theorem:diophantine.estimates}:

\begin{proposition}\label{prop:spval.prelim}
There exist constants $c_2, c_{3} > 0$ depending only on $a, b, p$ and $r$ such that for any orbit $o\in O$, either $\oomega(o) = r^{|o|}$ or
\[\log\left|1-\frac{\oomega(o)}{r^{|o|}}\right| \geq -c_2 -c_3\log|o|\,.\]
\end{proposition}
\begin{proof} It suffices to treat the case when $o\in O_\ast$, since otherwise $\oomega(o)= r^{|o|}$. 
Recall from \S\ref{sec:gausssums.orbits} that we may write $\oomega(o) = \zeta_o \cdot g_o^{L_o}$, where 
$\zeta_o$ is an $(ab)$\textsuperscript{th} root of unity, 
$g_o$ is a Weil integer of size $p^{\theta_{a,b}}$, and 
$L_o = [\F_r:\F_p]|o|/\theta_{a,b}$, with $\theta_{a,b}=\lcm(o_p(a),o_p(b))$.
We thus have
\[\log\left|1-\frac{\oomega(o)}{r^{|o|}}\right|
=\log\left|1-\zeta_o \cdot \big(g_o  p^{-\theta_{a,b}}\big)^{L_o}\right|\,.\]
 Applying Theorem~\ref{theorem:diophantine.estimates} and the definition of $L_{0}$ yields that
\[\log\left|1-\zeta_o \cdot \big(g_o  p^{-\theta_{a,b}}\big)^{L_o}\right|
\geq -c_0-c_1\log|L_o| \geq (-c_{0} - c_{1} \log [\F_{r}:\F_{p}]) - c_{1} \log |o|,\]
for some constants $c_0$ and $c_1$ depending on at most $p$, the integer $\theta_{a,b}$, the degree of $g_o$ over $\Q$ and the order of $\zeta_o$.
These three quantities can be bounded solely in terms of $a,b,$ and $p$, as we now explain. 
The root of unity $\zeta_o$ has order at most $ab$, 
the Gauss sum $g_o$ has degree at most $[\Q(g_o):\Q] \leq [\Q(\zeta_{a}, \zeta_b, \zeta_p):\Q] \leq abp$,
and $\theta_{a,b}\leq o_p(a)o_p(b)\leq \phi(a)\phi(b)\leq ab$.
\end{proof}

We are now ready to prove Theorem~\ref{thm:specialvalue}.

\begin{proof}[Proof of Theorem~\ref{thm:specialvalue}]
Combining the definition of $L^\ast(J)$ with the explicit expression for the $L$-function from Theorem~\ref{theorem:Lfunction} yields that 
\begin{equation*} 
    L^\ast(J) = \prod_{o\in O_0} |o| \prod_{o\in O_\ast}\left(1- \frac{\oomega(o)}{r^{|o|}}\right).
\end{equation*}
From this, we deduce that 
\begin{equation}\label{equation:size.spvalue.1}
\frac{\log L^\ast(J)}{q}
= \frac{1}{q}\sum_{o\in O_0} \log |o|
+ \frac{1}{q}\sum_{o\in O_\ast} \log \left|1-\frac{\oomega(o)}{r^{|o|}}\right|.
\end{equation}
We now estimate the two terms on the right-hand side separately. Lemma \ref{lemma:estimates.orbits}{\it(3)} gives
\begin{equation}\label{equation:size.spvalue.2}
0\leq\frac{1}{q}\sum_{o\in O_0}\log|o|
\leq\frac{1}{q}\sum_{o\in O^\times}\log|o|
\ll \frac{q}{q} \frac{\log\log q}{\log q}
\ll \frac{\log\log q}{\log q}.    
\end{equation}
As $q$ tends to infinity through powers of $p$, this term is $o(1)$.

We estimate the second term on the right-hand side of \eqref{equation:size.spvalue.1} in two steps. 
We begin by proving a suitable upper bound.
Since  
$|\oomega(o)|=r^{|o|}$ for all $o\in O$,
the triangle inequality implies that
\[ \frac{1}{q}\sum_{o\in O_\ast} \log \left|1-\frac{\oomega(o)}{r^{|o|}}\right| 
\leq \frac{|O_\ast|}{q} \log 2 \leq  \frac{|O|}{q}\log 2\,.\]
We know from Lemma~\ref{lemma:estimates.orbits}{\it(2)} that $|O|/q \ll (\log q)^{-1} $ as $q$ tends to infinity.

We now prove the required lower bound. 
By Proposition \ref{prop:spval.prelim}, we have
\[ - \frac{1}{q}\sum_{o\in O_\ast} \log \left|1-\frac{\oomega(o)}{r^{|o|}}\right| 
\leq \frac{1}{q}\sum_{o\in O_\ast} c_2 + c_3 \log|o| 
\leq c_2 \frac{|O|}{q} + c_3   \frac{1}{q} \sum_{o\in O_\ast}\log |o|\,. \]
By Lemma~\ref{lemma:estimates.orbits}{\it(2)}, we have $|O|/q \ll (\log q)^{-1}$.  Lemma~\ref{lemma:estimates.orbits}{\it(3)} implies that $\sum_{o\in O_\ast}\log |o|$ is $o(q)$ as $q\to\infty$.
Thus, the second terms on the right-hand side of  \eqref{equation:size.spvalue.1} satisfies
\begin{equation}\label{equation:size.spvalue.3}
  - \frac{\log\log q}{\log q}\ll \frac{1}{q}\sum_{o\in O_\ast} \log \left|1-\frac{\oomega(o)}{r^{|o|}}\right| \ll \frac{1}{\log q}
\end{equation}
as $q\to\infty$ through powers of $p$. 
Summing the inequalities \eqref{equation:size.spvalue.2} and \eqref{equation:size.spvalue.3} yields that
\[ - \frac{\log\log q}{\log q} \ll \frac{\log L^\ast(J)}{q} \ll \frac{1}{\log q}, \]
as $q\to\infty$ through powers of $p$.
We conclude that
\[ \frac{|\log L^\ast(J)|}{q} = O\left(\frac{\log \log q}{\log q}\right)
\qquad \text{as } q\to\infty\,.\] 
Our estimate from the height $H(J)$ in Lemma~\ref{lem:height-computation}  shows that the ratio $q/\log H(J)$ remains bounded (in terms of constants depending only on $a$ and $b$) as $q$ varies.
We conclude that
\[\frac{|\log L^\ast(J)|}{\log H(J)} = \frac{|\log L^\ast(J)|}{q} \frac{q}{\log H(J)}  = o(1)\,. \]
The implicit constants depend at most on $a,b,p,$ and $r$.
This concludes the proof of Theorem \ref{thm:specialvalue}.
\end{proof}

\subsection{Analogue of the Brauer--Siegel theorem}\label{sec:BrauerSiegel}
Combining Theorem~\ref{thm:specialvalue} and the Birch and Swinnerton-Dyer conjecture (Theorem \ref{thm:bsd}), we arrive at the following estimate.
\begin{customcor}{\ref{cor:brauersiegel}}\label{cor:brauersiegelII}
For given $a,b,$ and $r,$ as $q\to\infty$ runs through powers of $p$, we have
\[\log\big(|\Sh(J)|\,\Reg(J)\big) \sim \log H(J).\]
\end{customcor}
In the interpretation suggested by \cite{HindryPacheco}, this result provides an analogue of the Brauer--Siegel theorem for the family $(J_{a,b,q})_{q}$ of Jacobians.

Note that, except for a few examples in \cite[\S10.4, \S11.4]{Ulmer2019}, the relationship between the asymptotic growth rate of the product $|\Sh(A)|\,\Reg(A)$ and the asymptotic growth rate of the height $H(A)$ has not previously been elucidated in any sequence of abelian varieties $A$ of dimension greater than $1$. We note that there are several sequences of elliptic curves for which similar behaviour has been described. See \cite{HindryPacheco, GriffonPHD16, GriffonLegendre18, GriffonAS18, GriffonUlmer} for examples.

\begin{proof}
By the BSD formula (see \eqref{eq:BSD.formula} in Theorem~\ref{thm:bsd}), we have
\begin{equation*}
\frac{\log\big(|\Sh(J)|\,\Reg(J)\big)}{\log H(J)} 
= 1 - \frac{\log r^g}{\log H(J)} 
+ \frac{2 \log |J(K)_{\mathrm{tors}}|}{\log H(J)}
- \frac{\log \prod_v c_v}{\log H(J)} + \frac{\log L^\ast(J)}{\log H(J)}.
\end{equation*}
For a fixed pair $(a,b)$, the genus $g$ of $C=C_{a,b,q}$ is constant as $q$ varies. Hence the term $\log r^{g} / \log H(J) $ is $o(1)$ as $q\to\infty$. 
By Theorem 3.8 in \cite{HindryPacheco}, we have
\[\log|J(K)_{\mathrm{tors}}| = o\big(\log H(J)\big),\]
as $q\to\infty$ for fixed $a,b$, and $r$.
Furthermore, since the local Tamagawa numbers $c_v$ are all equal to~$1$ (see Proposition~\ref{prop:TamagawaNumber1}), we have 
 ${\log \prod_v c_v}  = 0$. 

Now, Theorem~\ref{thm:specialvalue} shows that the term ${\log L^\ast(J)}/{\log H(J)} $ is also $o(1)$ as $q\to\infty$. 
All in all, we obtain 
\[\frac{\log\big(|\Sh(J)|\,\Reg(J)\big)}{\log H(J)}  = 1+o(1),\]
\emph{ce qu'il fallait d\'emontrer.}
\end{proof}

\section{Large Tate--Shafarevich Groups}\label{sec:sha}

In this section we prove Theorem~\ref{thm:largesha}, which we recall for convenience:

\begin{customthm}{\ref{thm:largesha}}

Fix parameters $a,b$, and $r$ which satisfy the hypotheses of Theorem~\ref{thm:rankzero}. 
Then, as $q$ runs through powers of $p$, we have 
\[
|\Sh(J)| = H(J)^{1+ o(1)}.
\]

\end{customthm}

\begin{proof}
By Corollary \ref{cor:brauersiegelII}, we have 
\[\frac{\log\big(|\Sh(J)|\,\Reg(J)\big)}{\log H(J)}  = 1+o(1).\]
Theorem~\ref{thm:rankzero} shows that given
the hypotheses made on $(a,b)$, the analytic rank of $J$ is $0$ and so $\Reg(J) = 1$. Hence, we have
\[\frac{\log|\Sh(J)|}{\log H(J)} = 1 +o(1), \]
as $q\to\infty$ through powers of $p$. 
\end{proof}

\begin{corollary}
There are arbitrarily large integers $d\geq 1$ such that there exists an infinite sequence of $K$-simple Abelian varieties $A$ over $K$ of dimension $d$  satisfying
\[|\Sh(A)| = H(A)^{1+o(1)}
\qquad \text{as }H(A)\to\infty\,.\]
\end{corollary}
\begin{proof}
Let $d_0\geq 1$ be any integer. 
By Lemma~\ref{lem:infinitude-of-bad-omega}, we may choose a pair of coprime integers $(a,b)$ such that $a$ and $b$ are both prime, $(a-1)(b-1)\geq 2d_0$, and one of the conditions of Theorem \ref{thm:rankzero} is satisfied.
For such a pair $(a,b)$, consider the sequence  $(J_{a,b,q})_q$ of Jacobian varieties of dimension $d=(a-1)(b-2)/2$ indexed by powers $q$ of $p$. 
Since both $a$ and $b$ are prime,  Theorem~\ref{thm:simplicity} says that for any power $q$ of $p$, the Jacobian $J_{a,b,q}$ is $K$-simple.
By Theorem~\ref{thm:largesha}, the sequence 
$(J_{a,b,q})_q$ satisfies $|\Sh(J_{a,b,q})| = H(J_{a,b,q})^{1+ o(1)}$ as $q$ grows.
\end{proof}

\small
\bibliographystyle{alpha}
\bibliography{biblio.bib}

\normalsize
\vfill

\noindent\rule{7cm}{0.5pt}

\smallskip

\noindent
Sarah {\sc Arpin} ({\it \href{sarah.arpin@colorado.edu}{sarah.arpin@colorado.edu}}) --
{\sc University of Colorado Boulder, } 
Boulder, CO 80309 (USA).
\medskip

\noindent
Richard {\sc Griffon} ({\it \href{richard.griffon@uca.fr}{richard.griffon@uca.fr}}) --
{\sc Laboratoire de Math\'ematiques B. Pascal, 
Universit\'e Clermont--Auvergne,} 
Campus des C\'ezeaux,
3 place Vasarely,
TSA~60026 CS~60026, 
63178 Aubi\`ere Cedex (France).

\medskip
\noindent
Libby {\sc Taylor} ({\it \href{lt691@stanford.edu}{lt691@stanford.edu}}) --
{\sc Stanford University,} 
380 Serra Mall, Stanford, CA 94305 (USA).
\medskip

\noindent
Nicholas {\sc Triantafillou} ({\it \href{nicholas.triantafillou@gmail.com}{nicholas.triantafillou@gmail.com}}) --
{\sc University of Georgia,}
Boyd Graduate Research Center, Athens, GA 30602 (USA).
\medskip

\newpage

\appendix

\section{Conductor Computations}\label{app:Conductor}

Recall that $N_{J} \in \text{Div}(\mathbb{P}^1)$ is the conductor divisor of $J/K$.

\begin{proposition}\label{prop:degOfL}
We prove the statement from Theorem~\ref{thm:weil2} regarding the global degree $b(J)$ of the $L$-function $L(J,T)$:
\[b(J) = \deg(N_J) - 4g.\]

\end{proposition}

\begin{proof}

We begin by defining the conductor divisor $N_J$ as a divisor on the base $\P^1$. 
The action of inertia $I_v$ on the $\ell$-adic Tate module $V_\ell$ is tame\footnote{\cite{SerreTate} proves this when $p>2g+1$.  In our case, we can remove the hypothesis on $p$ as follows.  $J$ becomes trivial after a degree $ab$ field extension.  Over this extension, the action of inertia is trivial, so descending back to $K$ gives that the ramification degree must divide $ab$.  But $ab$ is prime to $p$, so the ramification must be tame.}.  For any place $v$ of $K$, define
\[
f(v):=\dim(V_\ell) - \dim(V_\ell^{I_v}),
\]
and let the conductor of $J$ be the divisor $N_J:=\sum_{v}f(v)v$ on $\P^1$. 
By~\cite{Serre}, $f(v)=0$ whenever $v$ is a place of good reduction for $J$.  Plugging in $\dim(V_\ell)=2g$ gives
\[\deg(N_J)= \sum_{v \text{ bad reduction}}(2g - \dim(V_\ell^{I_v})) \deg v,\]
where the sum  is over places $v$ of $K$ where $J$ has bad reduction.
Now, we investigate the $L$-function and see how its global degree relates to $\deg N_J$. Begin with the definition:
\[L(J,T):= \prod_{v}\det(1 - T\frob_v^{-1}|V_\ell^{I_v})^{-1}. \]

This product can be split up into products over good and bad places of $C$:
\[L(J,T):= \prod_{\text{good }v}\det(1 - T\text{Fr}_v^{-1}|V_\ell^{I_v})^{-1}\prod_{\text{bad }v}\det(1 - T\text{Fr}_v^{-1}|V_\ell^{I_v})^{-1}. \]
Let $\tilde{L}(J,T):= \underset{\text{good }v}{\prod}\det(1 - T\text{Fr}_v^{-1}|V_\ell^{I_v})^{-1}$. This gives a decomposition of the global degree:
\[\deg(L(J,T)) = \deg(\tilde{L}(J,T)) - \underset{\text{bad }v}{\sum}\dim(V_\ell^{I_v}).\]
Since $L(J,T)$ is rational, and since the sum $\underset{\text{bad }v}{\sum}\dim(V_\ell^{I_v})$ is finite, the ``complement'' $\tilde{L}(J,T)$ is also rational.
From here, we need a more precise formula for $\deg(\tilde{L}(J,T))$.  Let $U$ denote the affine open subset of $\P^1$ above which $J$ has good reduction. 
Since $U$ is a punctured $\P^1$, by the \'etale-singular cohomology comparison theorem, we have $\chi(U,\overline{\Q_\ell}):=\dim H^0(U,\overline{\Q}_\ell)-\dim H^1(U,\overline{\Q}_\ell)+\dim H^2(U,\overline{\Q}_\ell)=2-2g(\P^1)-r$, 
where $g(\P^1)$ is the genus of $\P^1$ 
and $r$ is the number of geometric points over which $J$ has bad reduction.  That is, $r$ is the sum of the degrees of places of bad reduction for $J$, namely $r=\underset{\text{bad }v}{\sum} \deg v$.
Therefore $\chi(U,\overline{\Q_\ell})=2-r$.  

\newcommand{\Fcal}{\mathcal{F}}
The Grothendieck--Ogg--Shafarevich formula (see~\cite{Castillejo16}) yields that
\[\chi(U,\Fcal)=\chi(U,\overline{\Q_\ell})\cdot\rank(\Fcal)-\sum_{x\in \P^1\setminus U}(\rank(\Fcal)+Sw_x(\Fcal)),\]
where in our case $\Fcal=V_\ell$, which is a lisse $\ell$-adic sheaf of rank $\dim V_\ell = 2g$ on $U$.  
Since the action of inertia on $V_\ell$ is tame (see \cite[Corollary 2, p. 497]{SerreTate}), this implies that
\[
\chi(U,\Fcal)=\chi(U,\overline{\Q_\ell})\cdot \rank(\Fcal)=2g(2-r).
\]
Now, since $\deg \tilde{L}(J,T)=-\chi(U,\Fcal)$, we deduce that
$\det \tilde{L}(J,T) =-2g(2-r)$. 
Putting this back into the equation for $\deg L(J,T)$ gives
\begin{align*}
    \deg(L(J,T)) &= \deg(\tilde{L}(J,T)) - \sum_{\text{bad $v$}}\dim(V_\ell^{I_v})
    =-4g + \sum_{\text{bad $v$}}2g - \sum_{\text{bad $v$}}\dim(V_\ell^{I_v})\\
    &= \sum_{\text{bad $v$}}(2g -\dim(V_\ell^{I_v})) - 4g
    = \deg(N_J)-4g.
\end{align*}
\end{proof}

\end{document}